\documentclass{amsbook}

\usepackage{amssymb}
\usepackage{amscd}
\usepackage[latin1]{inputenc}

\usepackage[active]{srcltx}

\newtheorem{theorem}{Theorem}[chapter]
\newtheorem{lemma}[theorem]{Lemma}

\newtheorem{corollary}[theorem]{Corollary}
\newtheorem{prop}[theorem]{Proposition}
\newtheorem{conj}[theorem]{Conjecture}
\newtheorem{question}[theorem]{Question}
\newtheorem{notation}[theorem]{Notation}

\theoremstyle{definition}
\newtheorem{definition}[theorem]{Definition}
\newtheorem{example}[theorem]{Example}

\theoremstyle{remark}
\newtheorem{remark}[theorem]{Remark}

\numberwithin{section}{chapter}
\numberwithin{equation}{chapter}

\newcommand{\skipit}[1]{{}}
\newcommand{\prfend}{\hbox to7pt{\hfil}
\par\vskip-\baselineskip\hbox to\hsize
{\hfil\vbox {\hrule width6pt height6pt}}\vskip\baselineskip}

\newcommand{\fa}{\mathfrak a}
\newcommand{\fb}{\mathfrak b}

\newcommand{\ffi}{\varphi}
\newcommand{\al}{\alpha}

\DeclareMathOperator{\im}{im}

\newcommand{\cH}{\mathcal{H}}

\newcommand{\mif}{\mbox{if} ~}
\newcommand{\s}{\; | \;}

\newcommand{\cocoa}{CoCoA}

\DeclareMathOperator{\Ann}{Ann}
\DeclareMathOperator{\codim}{codim}

\newcommand{\myarrow}[2]{\hbox to #1pt{\hfil$\to$\hfil}{\hskip-#1pt{\raise
10pt\hbox to#1pt{\hfil$\scriptscriptstyle #2$\hfil}}}}

\makeindex

\begin{document}

\frontmatter

\title{On the shape of a pure $O$-sequence}

%    Remove any unused author tags.

%    author one information
\author[Mats Boij]{Mats Boij}
\address{Department of Mathematics, Royal Institute of Technology, S-100 44 Stockholm, Sweden}
\curraddr{}
\email{boij@kth.se}
\thanks{}

%    author two information
\author[Juan C. Migliore]{Juan C. Migliore}
\address{Department of Mathematics, University of Notre Dame, Notre Dame, IN 46556}
\curraddr{}
\email{Juan.C.Migliore.1@nd.edu}
%\thanks{Partially  sponsored by the National Security Agency under Grant Numbers H98230-07-1-0036 and H98230-09-1-0031.}

%    author three information
\author[Rosa M. Mir\'o-Roig]{Rosa M. Mir\'o-Roig}
\address{Facultat de Matem\`atiques, Department d'\`Algebra i Geometria, Gran Via des les Corts Catalanes 585, 08007 Barcelona, Spain}
\curraddr{}
\email{miro@ub.edu}
%\thanks{Partially supported by MTM2010-15256. }

%    author four information
\author[Uwe Nagel]{Uwe Nagel}
\address{Department of Mathematics, University of Kentucky, 715 Patterson Office Tower,
Lexington, KY 40506-0027}
\curraddr{}
\email{uwenagel@ms.uky.edu}
%\thanks{Partially supported by the National Security Agency under Grant Numbers H98230-07-1-0065 and H98230-09-1-0032.}

%    author five information
\author[Fabrizio Zanello]{Fabrizio Zanello}
\address{Department of Mathematical Sciences, Michigan Technological University, Houghton, MI 49931}
\curraddr{Department of Mathematics, MIT, Office 2-336, Cambridge, MA 02139-4307}
\email{zanello@math.mit.edu}
\thanks{}

%    \date is required; it is the date received by the editor.
%\date{today}

%\subjclass[2000]{Primary }

%\keywords{keyword}

\maketitle

%\footnote[]{2010 {\em Mathematics Subject Classification.} Primary: 13H10; Secondary: 13D40, 13E10, 14M05.\\
%{\em Key words and phrases.} Type 2 level algebra. Weak Lefschetz Property}

%    Dedication.  If the dedication is longer than a line or two,
%    remove the centering instructions and the line break.
%\cleardoublepage
%\thispagestyle{empty}
%\vspace*{13.5pc}
%\begin{center}
%  Dedication text (use \\[2pt] for line break if necessary)
%\end{center}
%\cleardoublepage

%    Change page number to 7 if a dedication is present.
\setcounter{page}{1}

\tableofcontents

\newpage
\thispagestyle{empty}

\begin{center} {\Large {\bf Abstract}}
\end{center}
{\ }\\
\\
A  \emph{monomial order ideal} is a finite collection $X$ of (monic) monomials such that, whenever $M\in X$ and $N$ divides $M$, then $N\in X$. Hence $X$  is a poset, where  the partial order is  given by divisibility. If all, say $t$, maximal monomials of $X$ have the same degree, then $X$ is \emph{pure} (of type $t$).

A \emph{pure $O$-sequence} is the vector, $\underline{h}=(h_0=1,h_1,...,h_e)$,
counting the monomials of $X$ in each degree. Equivalently, pure $O$-sequences can be characterized as the
$f$-vectors of pure multicomplexes, or, in the language of commutative algebra, as the $h$-vectors of monomial Artinian level
algebras.

Pure $O$-sequences had their origin in one of  the early works of Stanley's in this area, and
have since played a significant role in at least three different disciplines: the study of
simplicial complexes and their $f$-vectors, the theory of level algebras, and the theory of matroids.
This monograph is intended to be the first systematic study of the theory of pure $O$-sequences.

Our work, which makes an extensive use of both algebraic and combinatorial techniques, in particular includes:
\begin{itemize}
\item[(i)] A characterization of the first half of a pure $O$-sequence, which yields the exact converse to a $g$-theorem of Hausel;
\item[(ii)] A study of (the failing of) the unimodality property;
\item[(iii)] The problem of enumerating pure $O$-sequences, including a proof that \emph{almost all} $O$-sequences are pure, a natural bijection between integer partitions and type 1 pure $O$-sequences, and the asymptotic enumeration of socle degree 3 pure $O$-sequences of type $t$;
\item[(iv)] A study of the \emph{Interval Conjecture for Pure $O$-sequences} (ICP), which  represents perhaps the strongest possible structural result short of an (impossible?) full characterization;
\item[(v)] A pithy connection of the ICP with Stanley's conjecture on the $h$-vectors of matroid  complexes;
\item[(vi)] A more specific study of pure $O$-sequences of type 2, including a proof of the Weak Lefschetz Property in codimension 3 over a field of characteristic zero. As an immediate corollary, pure $O$-sequences of codimension 3 and type 2 are unimodal (over an arbitrary field);
\item[(vii)] An analysis, from a commutative algebra viewpoint, of the extent to which the Weak and  Strong Lefschetz Properties can fail for monomial algebras;
\item[viii)] Some observations about pure $f$-vectors, an important special case of pure $O$-sequences.
\end{itemize}

\let\oldthefootnote\thefootnote
\renewcommand{\thefootnote}{}\footnotetext{
Monomial algebra. Unimodality. Differentiable sequence.  Level
algebra. Gorenstein algebra. Enumeration. Interval conjecture.
$g$-element. Weak Lefschetz Property. Strong Lefschetz Property.
Matroid simplicial complex. Macaulay's inverse system.}
\renewcommand{\thefootnote}{\oldthefootnote}

%    Include unnumbered chapters (preface, acknowledgments, etc.) here.
%\include{

\newpage
\thispagestyle{empty}

\noindent {\bf Acknowledgments.} Many of the ideas for this monograph were developed during a meeting the five authors held at Notre Dame in Fall 2008. They are sincerely grateful to the Department of Mathematics of the University of Notre Dame for providing hospitality and partial financial support.  Part of this work was also completed during a Research in Pairs activity at the Universit\`a di Trento in July 2009, under the auspices of the Centro Internazionale per la Ricerca Matematica (CIRM).  The authors are very grateful to   CIRM for the financial support, to the members of the Mathematics Department for their hospitality, and to Augusto Micheletti for his kindness and many efforts on their behalf.

The fifth author also wishes to thank Scott Nollet, who asked him, during a talk, the farsighted question whether his Interval Conjecture for arbitrary level algebras \cite{Za2} could extend to the monomial case. Also, he wishes to thank Erik Stokes for many  useful discussions on matroids, as well as for converting this manuscript in the Memoirs style.  We would like to thank Bernadette Boyle, David Cook II, Robin Hartshorne, Svante Linusson and Megan Patnott  for reading the manuscript, pointing out some errors and making useful suggestions. We also thank the referee for several helpful comments that improved our presentation.

We would like to acknowledge our substantial use of the computer system \linebreak \cocoa\ \cite{cocoa}, which helped us find many interesting examples and conjectures, and pointed the way to some of our results.

The second  author was partially  sponsored by the National Security Agency under Grant Numbers H98230-07-1-0036 and H98230-09-1-0031.  The third author was partially supported by MTM2007-61104. The fourth author was partially supported by the National Security Agency under Grant Numbers H98230-07-1-0065 and H98230-09-1-0032.
%}

\mainmatter
%    Include main chapters here.
%\include{

\chapter{Introduction}

The problem of characterizing pure $O$-sequences has a long history,
dating back to Stanley's 1977 seminal paper \cite{St1}. The goal of
this work is to shed  more light on this fascinating and perhaps
intractable problem, which lies at the crossroads between the
combinatorics of graphs and  simplicial complexes, the theory of
matroids, design theory, commutative algebra, and most recently, even plane partitions and bijective combinatorics. In particular, in
Chapter~\ref{ICP section}, we will present and begin a study of the
{\em Interval Conjecture for Pure $O$-sequences} (ICP),  which is
perhaps one of the best numerical results, short of the full
characterization, that one can hope to prove  on pure $O$-sequences.

Recall that a finite, non-empty set $X$ of (monic) monomials in the
indeterminates $y_1,y_2,...,y_r$ is called a {\em monomial order ideal} \label{useorderideal} if whenever $M\in X$ and $N$ is a monomial dividing $M$, we have $N\in X$. The {\em $h$-vector} of the set $X$ is the vector $\underline{h}=(h_0=1,h_1,...,h_e)$ counting the number of monomials of $X$ in each degree. A monomial order ideal $X$ is called {\em pure} if all maximal monomials of $X$ (the order being given by divisibility)  have the same degree. A {\em pure $O$-sequence} is the $h$-vector of a pure monomial order ideal.

The notion of a monomial order ideal comes from order theory, where in general, an \emph{order ideal} of a given poset $(S,\leq )$ is a subset of $S$ which is closed with respect to the relation $\leq $. Therefore, our monomial order ideals $X$ are the special case of (finite) order ideals in the poset $S$ (with respect to divisibility) of all monomials in $y_1,y_2,\dots $. Perhaps the most famous instances of order ideals are those  in the power set $S=2^V$ of a vertex set $V$ (a poset with respect to inclusion); these order ideals are best known as (abstract) \emph{simplicial complexes}, and can also be regarded as a special case (in fact, the \emph{squarefree} case) of our monomial order ideals, once we naturally identify a vertex $v_i\in V$ with a variable $y_i$. Since in this monograph we will only be concerned with \emph{monomial} order ideals, however, we will often abuse notation and simply refer to them as order ideals. 

The theory of pure $O$-sequences, as we said, began over thirty
years ago with Stanley's paper \cite{St1}, and has since attracted a
great deal of attention.    Some specific connections between pure $O$-sequences and other mathematical objects of interest are the following.

\medskip

\noindent $\bullet$ In algebraic combinatorics, pure $O$-sequences coincide with the $f$-vectors of pure
multicomplexes (i.e.\ simplicial complexes where the vertices are allowed to have multiplicities).

\medskip

\noindent $\bullet$ The best known and most studied special kind of pure $O$-sequence, perhaps, consists of the $f$-vectors of pure  simplicial complexes, i.e. {\em pure $f$-vectors}.  Recall that a pure simplicial complex is one where the facets all have the same dimension, and the component $f_i$ of the $f$-vector counts the faces of dimension $i$.  The characterization of these is also unknown (see some of the standard texts on the subject, such as \cite{MS,St3}, and  the results mentioned below).

\medskip

\noindent $\bullet$ In combinatorial commutative algebra,  thanks to the theory of {\em Macaulay's inverse systems} (also known as {\em Matlis Duality} - see \cite{Ge,IK} for an introduction to this theory), pure order ideals of monomials are in bijective correspondence with (standard graded)  monomial Artinian level algebras, and  pure $O$-sequences with their Hilbert functions. The  general theory of level algebras, which then followed mostly independent paths from that of pure order ideals, also originated with Stanley's work \cite{St1} (along with \cite{St2}). Among the several sources, we refer to the memoir \cite{GHMS} and to the  works \cite{We,Za} and their bibliographies for some of the most recent developments of this subject.

\medskip

\noindent $\bullet$ Recall that a projective plane of order $d$ is a finite projective plane with $d+1$ points on each line.  Such a plane has $q = d^2 +d+1$ points and $q$ lines.  The existence of a projective plane of order $d$ is equivalent to the existence of a pure $O$-sequence of the form
\[
\underline{h} = \left (1,q,q\binom{d+1}{2},q\binom{d+1}{3},\dots,q \binom{d+1}{d}, q\binom{d+1}{d+1} \right ),
\]
which in turn holds if and only if this $\underline{h}$ is a pure $f$-vector.
We will recall this connection in more detail in Chapter \ref{pure f-vectors}.  See also \cite{bjorner}, page 38.

\medskip

\noindent $\bullet$ In a similar way, the theory of pure $O$-sequences is related to the existence of error-correcting codes and designs (see e.g.\ \cite{KOP}).

\medskip

\noindent $\bullet$ Another important motivation for the study of pure
$O$-sequences comes from the theory of matroid  simplicial
complexes. It is  more beautiful than it is unsurprising that also
this theory began with the paper \cite{St1}. Indeed, in Section 7,
Stanley conjectured that the $h$-vector of a matroid complex (which
he actually called a ``$G$-complex'') is always a pure $O$-sequence.
Most of the huge amount of work done on matroids over the last
thirty-four years, involving ideas and techniques coming from
several different disciplines,  has in fact been motivated by that
intriguing conjecture, which remains wide open today (see,
for instance, \cite{HaSt, MNRV, NeNe, Oh, Schw, St3, Sto, Wh,Wh2} and their
bibliographies).

\medskip

\noindent $\bullet$ The very latest research direction  on pure
monomial order ideals is surprising, and sees them even connect to
plane partitions, families of non-intersecting lattice paths, and enumerative combinatorics.
In \cite{CookN}, D.\ Cook II and the fourth author observed that the
determinant  deciding the Weak Lefschetz Property (WLP) for a certain family of monomial almost complete intersections in three variables is the number of lozenge tilings of some hexagon. Then in \cite{LZ}, J. Li
and the fifth author  proved (algebraically) that the determinant of
the matrix deciding the Weak Lefschetz Property (WLP) for a
codimension 3 monomial Artinian level complete intersection is
exactly the function enumerating plane partitions inside a suitable
box (\cite{kra,Percy}). Note that plane partitions inside a box are in
bijection to lozenge tilings and also to certain families of non-intersecting lattice paths. Next, C. Chen, A. Guo, X. Jin, and G. Liu
\cite{CGJL}, explained  \emph{bijectively} and considerably extended
this connection. This seems only the beginning of the story. On the
algebraic  side, the class of (monomial) order ideals  for which the
WLP is decided by a function counting plane partitions might be
significantly large (see already \cite{BMZ,CookN} for more). More
generally, in \cite{CN2} it is shown that in some cases the WLP is
decided by a function counting families of non-intersecting lattice
paths, as studied by Gessel and Viennot (\cite{GV1}, \cite{GV2}).
Further, the  work of \cite{CGJL}  suggests that also the
combinatorics behind this theory has only begun surfacing. The work in
\cite{CGJL} already  connects other families of maps
between monomial spaces to three of the symmetry classes of plane
partitions, and it seems likely that more interesting  families of
plane partitions, including perhaps some of the other seven symmetry
classes, can be enumerated by  determinants of natural maps between
monomial spaces.

\medskip 

The main general results known today on pure $O$-sequences are due to Hibi (\cite{Hi}, Theorem 1.1 and Corollary 1.3) and to Hausel (\cite{Ha}, Theorem 6.3).  These are stated precisely in Theorem \ref{hibi} and Theorem \ref{hausel}.  Briefly, though, Hibi showed that all pure $O$-sequences $\underline{h}=(h_0=1,h_1,...,h_e)$ are {\em flawless} (meaning that $h_i\leq h_{e-i}$ for all indices $i\leq e/2$) and that, consequently, $\underline{h}$ is non-decreasing in the ``first half" (i.e., until degree $\lfloor e/2 \rfloor $). Hausel extended Hibi's work by proving a  ``$g$-result'' for pure $O$-sequences: namely, not only is the ``first half" of $\underline{h}$ non-decreasing, but it is also {\em differentiable}, that is, its first difference is the $h$-vector of some (not necessarily pure) order ideal (or, in the language of commutative algebra, it satisfies Macaulay's theorem). In fact, Hausel's $g$-result was \emph{algebraic}, not only \emph{enumerative} (that is, the differentiability property is a consequence of the injectivity of certain maps between vector spaces). Notice that the corresponding results to Hibi's and Hausel's are already known to hold true for matroid complexes (thanks to Chari \cite{Ch,Ch2} and Swartz \cite{Sw}; Hausel supplied another proof of those theorems in the same paper \cite{Ha}, among a few other interesting algebraic $g$-results).

However, whereas the ``first half" of a pure $O$-sequence $\underline{h}$ enjoys a property  as nice as differentiability, not even the condition of flawlessness, unfortunately, suffices to impose a good behavior on the second half of $\underline {h}$. In fact, even  {\em unimodality} may fail for pure $O$-sequences (a vector is defined to be {\em unimodal} when it does not increase after a strict decrease); the first example of a non-unimodal pure $O$-sequence is due to Stanley, who (always in \cite{St1}) proved that $(1,505,2065, 3395, 3325, 3493)$ \label{stanley example} is the $f$-vector of a Cohen-Macaulay simplicial complex, and therefore is a pure $O$-sequence.

The problem of finding non-unimodal pure $O$-sequences has since been studied by several authors, including Erd\"os and Bj\"orner, and numerous examples have been found even for special classes of pure $O$-sequences, such as the $f$-vectors of pure simplicial complexes or even the independence vectors of well-covered graphs (see, e.g., \cite{AMSE,Bj2,LM,MT}). In fact, Michael and Traves \cite{MT} even proposed the ``Roller-Coaster Conjecture'' for the independence vectors of well-covered graphs, which basically states that such sequences can be non-unimodal with as many peaks as we wish.

Examples of non-unimodal pure $O$-sequences have also been given recently, with algebraic methods, by the first and the last author  \cite{BZ}, in any number of variables $r\geq 3$ (all combinatorial examples, instead, need that $r$ be fairly large).

In general, especially in the light of the non-unimodality results, today the widespread belief is that achieving a full characterization of pure $O$-sequences is virtually hopeless. One of the goals of our monograph is to see this problem under a new perspective, by proposing the Interval Conjecture for Pure $O$-sequences (ICP). The ICP at least guarantees a very strong structural result for the set of pure $O$-sequences, which is perhaps as much as one can realistically  hope for. Namely, the ICP says  that if the entries of two pure $O$-sequences, $\underline{h}$ and $\underline{h}'$, coincide in all degrees but in one - say, in degree $i$, $h'_i=h_i+\alpha $ for some $\alpha >0$ - then for each value of $\beta =0,1,...,\alpha $, there exists a pure $O$-sequence $\underline{h}''$ coinciding with $\underline{h}$ and $\underline{h}'$ in all degrees different from $i$, and such that $h''_i=h_i+\beta $.
An analogous conjecture was  recently  proposed by the last author \cite{Za2} for arbitrary Artinian level algebras, and is still wide open also in that context.

 We will now outline the structure of this monograph. In
Chapter~\ref{defs and prelim results} we  fix the notation and collect the definitions as
well as the basic facts on Hilbert functions needed later on. In
Chapter~\ref{pureness and diff}
  we study  the relationship intercurrent between pure and differentiable $O$-sequences.
 The main result there will be a surprisingly simple characterization of the ``first half" of pure
 $O$-sequences: they exactly coincide with the differentiable $O$-sequences (thus providing a
 converse to Hausel's theorem). In particular, since the truncation of a pure $O$-sequence
 is also a pure $O$-sequence, it follows, when it comes to non-decreasing sequences,  that every
 differentiable  $O$-sequence is pure.

Curiously, we will prove that the converse to the latter result
only holds for socle degrees $e\leq 3$.  In fact, for each $e\geq
4$, we will  construct an example of a non-decreasing  pure
$O$-sequence which is not differentiable. We will also give an (in
general sharp)  bound on the growth of the first difference of a non-decreasing pure $O$-sequence (of any socle degree), yielding a
nice consequence concerning the {\em Weak Lefschetz Property} (WLP).  Recall that this property just means that multiplication by a general linear form  has maximal rank from any component of the algebra to the next.

The techniques used in constructing the latter examples will also
be applied to provide the first example of a non-unimodal, socle
degree 4 pure $O$-sequence (previously, the least socle degree known
to allow non-unimodality was 5, thanks to the above example of
Stanley's from \cite{St1}). This easily implies that the largest
socle degree forcing all pure $O$-sequences to be unimodal is
exactly 3.

We will conclude the chapter on unimodality by proving a ``roller-coaster'' theorem (analogous to that conjectured by Michael-Traves \cite{MT}, and still open, for the special case of well-covered graphs): namely, for any positive integer $N$ and any integer $r\geq 3$, there exist pure $O$-sequences of codimension $r$ having exactly $N$ peaks. In particular, using the standard technique of {\em distraction} of monomial Artinian ideals to one-dimensional ideals of reduced sets of points, a nice geometric consequence of our result is that, in any projective space $\mathbb P^r$ such that $r \geq 3$, there exist level sets of points having an arbitrary number, $N$, of peaks in their $h$-vectors.

In Chapter~\ref{ICP section}, we will prove the ICP in a few interesting cases,
most importantly for pure $O$-sequences whose largest degree $e$
(called the {\em socle degree}) is at most 3. Using the ICP for
$e=2$ and the Brown-Colbourn inequalities (see \cite{BC}), we are
also able to give a very short alternative proof that Stanley's
conjecture on the $h$-vector of matroid complexes holds true for
matroids of dimension 1 (this fact was first  shown by Stokes
in his Ph.D. Thesis \cite{Sto,Sto2}).

In Chapter \ref{enum}, we begin the study of counting the number of pure $O$-sequences of fixed codimension $r$ and socle degree $e$.  We give an explicit asymptotic value for this number, and show that it is the same as the corresponding value for the $O$-sequences of codimension $r$ and socle degree $e$, and the number of differentiable $O$-sequences of codimension $r$ and socle degree $e$.  We conclude that for $r$ large, ``almost all" level sequences are pure.  Similarly, we are able to show that for $r$ large, ``almost all" Gorenstein Hilbert functions are SI-sequences, and give the explicit asymptotic value.   The enumeration of pure $O$-sequences of given codimension, socle degree {\em and type} seems even harder; things take an entirely different path here from the enumeration of the corresponding level Hilbert functions. Chapter \ref{enum} also includes a bijective enumeration of the number of type 1 pure $O$-sequences of given codimension and socle degree.

Recall that the WLP imposes important restrictions for the Hilbert function, including a strong form of unimodality (see Remark \ref{wlpunim}).  It is thus of direct interest, in the study of the possible pure $O$-sequences, to ask under what conditions we can conclude that a monomial Artinian level algebra has the WLP.  In two variables, all algebras have not only the WLP but in fact also the Strong Lefschetz Property (SLP -- the analogous property when multiplication by $L$ is replaced by multiplication by $L^s$ for all possible $s \geq 1$ -- see Definition \ref{def of wlp}).  In three variables it was shown by Brenner and Kaid \cite{BK} that the WLP can fail when the type is 3, even for monomial almost complete intersections. This important class of algebras was, or is being, further studied in \cite{BMZ,CookN,MMN2} with respect to the failure of the WLP.

In Chapter \ref{type 2 three var} we consider the case of three variables.  We first give a result showing that the structure of any such ideal takes one of two possible forms.  Then we show that monomial Artinian level algebras  of type 2 of either form  necessarily have the WLP in characteristic zero.  Thus all pure $O$-sequences of type 2 in three variables are unimodal in the strong sense of Remark \ref{wlpunim}.

In Chapter \ref{type 2 arb var} we give a complete answer to the question of which combinations of type  and number of variables  may fail the WLP, and which may fail the SLP.  Specifically, we show in Theorem \ref{answernd} that all monomial Artinian level algebras in $r$ variables and type $d$ possess the WLP if and only if at least one of the following holds: (a) $r = 1$ or 2; (b) $d=1$; or (c) $r=3$ and $d=2$.  Furthermore, all monomial Artinian level algebras in $r$ variables and type $d$ possess the SLP if and only if either (a) or (b) holds.
This is accomplished via a long series of constructions, including some preparatory lemmas on tensor products.  For type 2 in four or more variables, we show that either injectivity or surjectivity may fail.  This is in striking contrast to the case of three variables analyzed in Chapter \ref{type 2 three var}.

In Chapter \ref{pure f-vectors} we focus on pure $f$-vectors.  First, we show that at least in the Cohen-Macaulay case, these $f$-vectors have the interval property described in the ICP, and we wonder if the analogous result holds for pure $f$-vectors in general.  We then consider pure $f$-vectors of type 2.  We show that all the entries of the $f$-vector in this case fall into a series of inequalities, and that unimodality is an immediate consequence.

Chapter~\ref{final comments} contains a collection of some natural open or open-ended questions arising from our work.

Finally, we end the monograph with an appendix, where  we recap the main definitions and notation that we have been using throughout our work.

%%%%%%%%%%%%%%%%%%%%%%%%%%%%%%%%%%%%%%%%%%%%%%%%%%

\chapter{Definitions and preliminary results} \label{defs and prelim results}

In this chapter we recall some standard terminology and notation
from commutative algebra and algebraic combinatorics, as well as some results needed in the rest of this work.

Let $R = k[x_1,\dots,x_r]$, where $k$ is an infinite field, and let  $\mathfrak{m} = (x_1, \dots ,x_r)$ be its maximal homogeneous ideal. We consider  a  standard graded Artinian $k$-algebra $A=R/I=\bigoplus _{i\ge
0}A_{i}$ where $I$ is a homogeneous ideal of $R$. Since we may assume, without loss of generality, that the ideal $I$ does not contain non-zero forms of degree 1, we define $r$ to be the {\em codimension} of $A$.

The {\em Hilbert
function} $h_A$ of $A$ is defined by $h_A(t):=\dim _kA_{t}$
and the {\em socle} \label{defsocle} of $A$ is defined as the annihilator of the homogeneous maximal ideal $\overline{\mathfrak{m}} = (\overline{x_1}, \dots ,\overline{x_r})\subseteq A$, namely $soc(A)=\{ a\in A \mid a\overline{\mathfrak{m}}=0 \}$.
Note that the Hilbert function of an Artinian $k$-algebra has finite support and is captured in its {\em $h$-vector} $\underline{h} = (h_0,h_1,\dots,h_e)$ where $h_0=1$, $h_{i}=h_A(i)>0$ and $e$ is the last index with this property. The integer $e$ is called the {\em socle degree} of $A$ and $e+1$ the {\em length} of $\underline{h}$.
We say
that $A$ is {\em level of type $t$} if the socle $soc(A)$ of $A$ is of
dimension $t$ and is concentrated in one degree. Furthermore, $A$
is {\em Gorenstein} if and only if it is level of type 1.

\vskip 2mm  For a homogeneous ideal $I\subseteq
R=k[x_1,x_2,\dots,x_r]$, we define the \emph{inverse system} \label{definversesystem} $R$-module
$I^\perp=\operatorname{ann}_{\mathcal R}(I)\subseteq\mathcal
R=k[y_1,y_2,\dots,y_r]$, where $R$ acts on $I^\perp $ by
contraction. That is, the external product is the linear action determined by
$$
x_i \circ y_1^{a_1}y_2^{a_2}\cdots y_r^{a_r} =
  \begin{cases}
    y_1^{a_1}y_2^{a_2}\cdots y_i^{a_i-1}\cdots y_r^{a_r},&\text{if $a_i>0$,}\\
    0,&\text{if $a_i=0$,}
  \end{cases}
$$
for any non-negative integers $a_1,a_2,\dots,a_r$. When considering only monomial ideals, we can simplify the notation by considering inverse system modules to also lie inside the  polynomial ring $R$. It follows that, in any degree $d$, the inverse system $I^\perp_d$ is spanned by the monomials of $R_d$ not in $I_d$. We will call the elements of $I^\perp_d$ \emph{derivatives}. See \cite{Ge,IK} for a comprehensive (and more general) treatment of the theory of inverse systems.

\vskip 2mm
\label{binomialexpansion}
Let $n$ and $d$ be positive integers. There exist uniquely
determined integers $k_d>k_{d-1}> \dots
>k_{\delta}\geq \delta \geq 1 $ such that
\[
n = {\textstyle n_{(d)}= \binom{k_{d}}{d}+ \binom{k_{d-1}}{d-1}+
\cdots + \binom{k_{\delta}}{\delta} .}
\]
 This is called the {\em
$d$-binomial expansion of $n$}. Following \cite{BG}, for any
integers $a$ and $b$, we define
$$ {\textstyle (n_{(d)})^b_a:= \binom{k_{d}+b }{ d+a}+\binom{k_{d-1}+b }{ d+a-1}+
\cdots + \binom{k_{\delta}+b }{ \delta +a}}$$ where we set
$\binom{m}{p}=0$ whenever $p<0$ or $m<p$.

The basic results about the behavior of Hilbert functions are
summarized in the following theorem.

\begin{theorem} \label{macaulay} Let $A$ be a graded algebra. Let $L\in A=R/I$ be  a general linear form. Denote by $h_d$
the degree $d$ entry of the Hilbert function of $A$ and by $h'_d$
the degree $d$ entry of the Hilbert function of $A/LA$. The following holds:
\begin{itemize} \item[(i)] (Macaulay) $h_{d+1}\le ((h_d)_{(d)})_1^1$.
\item[(ii)] (Gotzmann) If $h_{d+1}=((h_d)_{(d)})_1^1$ and $I$ is
generated in degrees $\le d$ then $h_{d+s}=((h_d)_{(d)})_s^s$ for
all $s\ge 1$. \item[(iii)] (Green) $h'_{d}\le ((h_d)_{(d)})_0^{-1}.$
\end{itemize} \end{theorem}
\begin{proof}
 (i) See \cite{macau}, or \cite{BH}, Theorem 4.2.10.

 (ii) See \cite{Go}, or \cite{BH}, Theorem 4.3.3.

 (iii) See \cite{Gr}, Theorem 1.
\end{proof}

\begin{definition}\label{basic-defs}\rm
 Let $\underline{h} = (h_0,\dots,h_e)$ be a sequence of positive integers.

 (1) $\underline{h} $ is said to be an {\em $O$-sequence} if it satisfies
 Macaulay's bound for all $d$.

(2) $\underline{h}$ is called {\em unimodal} if $h_0 \leq h_1
\leq \dots \leq h_j \geq h_{j+1} \geq \dots \geq h_e$ for some
$j$.

(3) $\underline{h}$ is said to be {\em flawless}  if $h_{i}\le
h_{e-i}$ for $i\le \frac{e}{2}$.

(4) We will say that $\underline{h}$ is {\em non-decreasing} if
$h_{i+1}\ge h_{i}$ for all $i$.

 (5) $\underline{h}$ is called an {\em SI-sequence}  if it
satisfies the following two conditions:

\begin{itemize}
\item[(a)] $\underline{h}$ is symmetric, i.e.\ $h_{e-i} = h_i$ for
all $i=0,\dots,\lfloor \frac{e}{2} \rfloor$. \item[(b)] $(h_0,
h_1-h_0,h_2-h_1,\dots,h_j-h_{j-1})$ is an $O$-sequence, where $j =
\lfloor \frac{e}{2} \rfloor$; i.e.\ the ``first half'' of
$\underline{h}$ is a {\em differentiable} $O$-sequence.
\end{itemize}
\end{definition}

We remark that  Macaulay also showed that any $O$-sequence is
actually the Hilbert function of some standard graded algebra. A
{\em pure $O$-sequence} is the $h$-vector of a pure order ideal or, equivalently, the $h$-vector  of an Artinian level
monomial algebra. Note that a pure $O$-sequence is also the $f$-vector of a pure multicomplex. By \cite{Hi}, Theorem 1.1 and Corollary 1.3, any
pure $O$-sequence is flawless. The problem of characterizing pure
$O$-sequences  is far from being solved and the goal of the next chapters
is to shed more light on it.  We now recall the main two results on pure $O$-sequences that we will use in this monograph.

\begin{theorem}[Hibi \cite{Hi}, Theorem 1.1]  \label{hibi}
Let $\underline{h}$ be a pure $O$-sequence of socle degree $e$.  Then
\[
h_i \leq h_j
\]
whenever $0 \leq i \leq j \leq e-i$.  This has the following two important consequences:

\begin{itemize}
\item[(a)] $\underline{h}$ is {\em flawless}, i.e.\ $h_i \leq h_{e-i}$ for all $0 \leq i \leq \lfloor \frac{e}{2} \rfloor$.

\item[(b)] The ``first half'' of $\underline{h}$ is non-decreasing:
\[
1 = h_0 \leq h_1 \leq h_2 \leq \cdots \leq h_{\lfloor \frac{e}{2} \rfloor}.
\]
\end{itemize}
\end{theorem}

In \cite{Ha}, Hausel has then reproved and strengthened  Hibi's
result. Here, however, we only state the portion of Hausel's theorem
that we will  need in the rest of our work, and we phrase it in a slightly
different way from that of \cite{Ha} for our purposes below.

\begin{theorem}[Hausel \cite{Ha}, Theorem 6.2] \label{hausel}
Let $A$ be a monomial Artinian level algebra of socle degree $e$. If the field $k$ has characteristic zero, then for a {\em general} linear form $L$, the induced multiplication
\[
\times L : A_j \rightarrow A_{j+1}
\]
is an injection, for all $j = 0,1,\dots,\lfloor \frac{e-1}{2} \rfloor$.  In particular, over any field the sequence
\[
1, h_1 -1, h_2 - h_1 ,\dots, h_{\lfloor \frac{e-1}{2} \rfloor +1} - h_{\lfloor \frac{e-1}{2} \rfloor}
\]
is an $O$-sequence, i.e.\ the ``first half'' of $\underline{h}$ is a {\em differentiable $O$-sequence}.
\end{theorem}

\begin{remark} \label{g-remarks}
{\em
\begin{enumerate}
\item Without the assumption that the field has characteristic zero, the first part of Hausel's theorem would not be true.  Indeed, there are several observations and results in the literature (cf. for instance \cite{HMNW} Remark 2.9, \cite{MMN2} Example 7.10 and Question 7.12, \cite{LZ} Theorem 3.2) to show that in characteristic $p$, a monomial complete intersection may fail the stated property.

\item The property quoted in the first part of Hausel's theorem can also be phrased by saying that $\underline{h}$ satisfies the $g$-inequalities, and this result is often referred to as ``Hausel's $g$-theorem." We will call such a linear form a {\em $g$-element}.  Note that if the field $k$ is infinite and if $R = k[x_1,\dots,x_r]$, it was shown in \cite{MMN2}, Proposition 2.2 that the linear form $L$ can be taken to be $x_1+\cdots+x_r$.

\item At first glance, it would seem that the second part of Hausel's theorem would also require characteristic zero, since its proof uses the first part (the existence of a $g$-element).  However, the second part is purely about Hilbert functions.  Since the Hilbert function of a monomial ideal does not depend on the base field, the stated property continues to hold for the given monomial ideal even if we change the field.  To rephrase this in the language of algebraic combinatorics, an \emph{algebraic} $g$-theorem is stronger than an \emph{enumerative} $g$-theorem. For this reason, Hausel's theorem is indeed a generalization of Hibi's theorem (which has no assumption on the characteristic).  This observation will allow us, next, to draw conclusions about Hilbert functions over any base field even if we use $g$-elements in the proof.

More precisely, all purely combinatorial statements on monomials do hold over an arbitrary field. Thus even if some of the techniques we are employing in this monograph are algebraic, and so some of our {\em arguments}  need to rely on certain properties of a base field, any conclusion giving a numerical property of pure $O$-sequences is clearly only combinatorial in nature, and therefore, once shown ``over some field,'' it is in fact  true independently of the characteristic.
\end{enumerate}
}
\end{remark}

\begin{definition} \label{simplicialcomplex}
An {\em (abstract) simplicial complex} $\Delta $ on $[n]:= \{1,
\dots  ,n\}$ is a subset of the power set of $[n]$ which is  closed
under inclusion, i.e. if $F \subseteq G$ and $G \in \Delta $, then
$F \in \Delta $. The elements $F$ of $\Delta $ are called {\em
faces}, and the maximal elements under inclusion are called {\em
facets}. The simplicial complex $\Delta $ is said to be  {\em pure}
if all its facets have the same cardinality.  The {\em $f$-vector} of $\Delta $ is the vector $\underline{f_{\Delta }} = \underline{f} = (f_{-1} = 1 ,f_0 , \dots , f_{d-1})$, where $f_i$ counts the number of faces of $\Delta $ of cardinality $i+1$.

A simplicial complex  is called a {\em matroid} if, for all
$W\subseteq [n]$, the restriction $\Delta _W = \{F \in \Delta : F
\subseteq W \}$ is a pure simplicial complex.
\end{definition}

The connection to algebra is provided by the following construction. For a subset $F \subset  [n]$, we write $x_F$ for the squarefree monomial
  $\prod _{i\in F} x_{i}$. The {\em Stanley-Reisner ideal} of $\Delta $ is
$I_{\Delta } = \{ x_F : F \subseteq [n], F \notin \Delta \}$ and the corresponding
{\em Stanley-Reisner ring}
is $K[\Delta] = R/I_{\Delta}$. We will say that $\Delta$ has an algebraic property like Cohen-Macaulayness if $K[\Delta]$ has
this property. For more details on simplicial complexes, Stanley-Reisner rings and their
algebraic properties we refer to the books of Bruns-Herzog \cite{BH} and Stanley \cite{St3}.

\begin{definition}\label{def of wlp} We say that  a standard graded Artinian algebra $A=R/I$
has the {\em Weak Lefschetz Property (WLP)}
if there is a linear form $L \in A_1$ such that, for all
integers $j$, the multiplication map
\[
\times L: A_{j} \to A_{j+1}
\]
has maximal rank, i.e.\ it is injective or surjective.   (We will often abuse notation and say that $I$ has the WLP.)  We say that $A$ has the {\em Strong Lefschetz Property (SLP)} if there is an $L$ such that
\[
\times L^s : A_j \rightarrow A_{j+s}
\]
 has maximal rank for all $j$ and all $s$.
\end{definition}

\begin{remark}\label{wlpunim}{\em
   Part of the great
interest in the WLP stems from the fact that its presence puts
severe constraints on the possible Hilbert functions, which can
appear in various disguises (see, e.g., \cite{St-faces} and
\cite{MN3}).   Specifically, if $R/I$ has the WLP then its Hilbert
function is unimodal in the following strong sense: it is
differentiable in the first interval (in particular it is strictly
increasing), it is constant in the second  (possibly trivial)
interval, and then it is non-increasing in the third (\cite{HMNW},
Remark 3.3).  Furthermore, the fact that $R/I$ is level imposes the
further condition that in the third interval (once there has been a
strict decrease), it is actually strictly decreasing until it
reaches 0 (\cite{AS}, Theorem 3.6). }
\end{remark}

Though many algebras are expected to have the
WLP, establishing this property is often rather difficult. For
example, it is open whether every complete intersection of codimension $\ge 4$
over a field of characteristic zero has the WLP.

We conclude this chapter by recalling a concept from liaison theory,
which we do not state in the greatest generality.    \label{Glink}
 Recall that two homogeneous ideals, $I$ and $I'$, are {\em Gorenstein linked} if there is a homogeneous Gorenstein ideal $\mathfrak a$ such that $\mathfrak a : I = I'$ and $\mathfrak a : I' = I$.  
 \label{basicdoublelink}
 Let $J \subset I \subset R = k[x_1,\ldots,x_r]$ be homogeneous
ideals such that $\codim J = \codim I - 1$. Let $f \in R$ be a
 form of degree $d$ such that $J  : f = J$. Then the ideal $I' := f
\cdot I + J$ is called a {\em basic double link} of $I$. The name
stems from the fact that $I'$ can be Gorenstein linked to $I$ in two
steps if $I$ is unmixed and $R/J$ is Cohen-Macaulay and generically
Gorenstein (\cite{KMMNP}, Proposition 5.10). However, here we essentially only
need the relation among the Hilbert functions.

\begin{lemma}
  \label{lem:BDL}
For each integer $j$,
\[
\dim_k (R/I')_j = \dim_k (R/I)_{j-d} + \dim_k (R/J)_{j} - \dim_k
(R/J)_{j-d}.
\]
\end{lemma}

\begin{proof}
  This follows from the exact sequence (see \cite{KMMNP}, Lemma 4.8)
\[
0 \to J(-d) \to J  \oplus I (-d) \to I' \to 0.
\]
\end{proof}

%%%%%%%%%%%%%%%%%%%%%%%%%%%%%%%%%%%%%%%%%%%%%%%%

\chapter{Differentiability and unimodality}  \label{pureness and diff}

As we said  in the introduction, this chapter is devoted to studying the relationship between pure and differentiable $O$-sequences.

We begin with the main result of this chapter, namely a characterization of the ``first half" of pure $O$-sequences, by proving the exact converse to Hausel's theorem.

\begin{theorem}\label{firsthalf}
A finite $O$-sequence $\underline{h}$ is the ``first half" of a pure $O$-sequence if and only if it is differentiable.
\end{theorem}

\begin{proof}
By Hausel's theorem (Theorem \ref{hausel}), the ``first half"  of a
pure $O$-sequence $\underline{h}$ is differentiable. Let us prove
the converse.  Let $\underline{h}$ be a finite differentiable
$O$-sequence, say of socle degree $e$.  Let $\underline{g}$ be the
sequence defined by $g_0 = 1$, $g_i = h_i - h_{i-1}$ for $1 \leq i
\leq e$, and $g_i = 0$ for all $i \geq e+1$.  By hypothesis,
$\underline{g}$ is an $O$-sequence.  Let $S$ be the polynomial ring
in $h_1-1$ variables, and let $I'$ be the lex-segment ideal in $S$
with Hilbert function $\underline{g}$.  Now let $R$ be the
polynomial ring in $h_1$ variables, obtained by adding one variable
to $S$, and let $I''$ be the ideal $RI'$ in $R$.  Note that $R/I''$
has depth 1.  The Hilbert function of $R/I''$ is $\underline{h}$ up
to degree $e$, and then it takes the constant value $h_e$ for all
degrees $\geq e+1$.  If $\mathfrak m$ is the irrelevant ideal of
$R$, let $I = I'' + {\mathfrak m}^{2e+1}$.  Then $I$ is an Artinian
monomial ideal whose Hilbert function has ``first half'' equal to
$\underline{h}$ and $R/I$ is level of socle degree $2e$, as desired.
\end{proof}

Since the truncation of a pure $O$-sequence is again a  pure
$O$-sequence (see, for instance, \cite{Hi}), from Theorem
\ref{firsthalf} we immediately have the following:

\begin{corollary}\label{diff}
Any finite differentiable $O$-sequence $\underline{h}$ is pure.
\end{corollary}

Let us now focus on non-decreasing  pure $O$-sequences; if the socle degree is at most 3, we have a converse to the previous corollary:

\begin{theorem}\label{3}
Any non-decreasing pure $O$-sequence $\underline{h}$  of socle
degree $e\leq 3$ is differentiable.
\end{theorem}

\begin{proof}  The result is trivial if the socle degree is $e\leq 2$. Hence let $\underline{h}=(1,r,a,b)$ be the non-decreasing $h$-vector of a monomial level algebra $A=R/I$ of socle degree 3; we want to show that $\underline{h}$ is differentiable.

First of all note that, by Hausel's theorem (Theorem \ref{hausel}),   $\underline{h}$ is differentiable up to degree 2.
Hence it remains to show that $\underline{h}$ is differentiable from degree 2 to degree 3, i.e. that
\begin{equation} \label{bar}
b-a\leq ((a-r)_{(2)})^1_1 .
\end{equation}
Since (\ref{bar}) is a special case of  Proposition \ref{ub} (which does not depend on the present result), we defer its proof.
\end{proof}

Let us now make an excursion more properly into commutative algebra. Note that in the previous proof the only way that we have really used Hausel's theorem is that it guarantees the existence of a $g$-element for the algebra $A$, but we have never used that $A$ is monomial, or level, or (at least in an essential way) that it has socle degree 3. Indeed, basically the same argument proves, more generally, the following  purely algebraic result which also follows from \cite{HMNW}, Proposition 3.5:

\begin{theorem}\label{g}
Let $A$ be any Artinian algebra of socle degree $e$ whose $h$-vector $\underline{h}$ is non-decreasing, and suppose that there exists a $g$-element for $A$ up to degree $e-1$ (i.e. for a general linear form $L$ and for all $j\le e-1$, the map $\times L: A_{j} \to A_{j+1}$ is an injection). Then $\underline{h}$ is differentiable.
\end{theorem}

Also, notice the following fact: the very existence of a $g$-element for $A$ is essential for the conclusion of Theorem \ref{g}, in the sense that this assumption cannot be relaxed by simply requiring, combinatorially, that the $h$-vector of $A$ be  differentiable up to degree $e-1$. This is true even for level algebras, since, for instance, it can be shown that $\underline{h}=(1,13,13,14)$ is a level $h$-vector, but $\underline{h}$  clearly fails to be differentiable from degree 2 to degree 3.  (To construct this level $h$-vector, one can start with Stanley's well-known non-unimodal Gorenstein $h$-vector $(1,13,12,13,1)$ and add the fourth power of a general linear form in the same variables to the inverse system to get $(1,13,13,14,2)$, and then truncate.)

Theorem \ref{g} also carries a consequence concerning the WLP, since it provides an indirect way to construct algebras without the WLP. Take an Artinian algebra $A$ having a non-decreasing $h$-vector, differentiable up to degree $e-1$, but not from degree $e-1$ to degree $e$. Then, by Theorem \ref{g}, the Artinian algebra obtained by truncating $A$ after degree $e-1$ fails to have the WLP. The bottom line is:

\vspace{.5cm}

{\em If the $h$-vector of some Artinian algebra $A$ is non-decreasing but not differentiable, then $A$ starts failing the WLP {\em in an earlier degree} than it has to for numerical reasons.}

\vspace{.5cm}

Returning now to pure $O$-sequences, we note that the result  in Theorem \ref{3} is false in any higher socle degree.  (We have already seen an example in socle degree 4 above, but we give a different one now.)

\begin{prop}\label{4}
There exist non-decreasing  pure $O$-sequences $\underline{h} = (h_0,\dots,h_e)$ of any socle degree $e\geq 4$ that are not differentiable.
\end{prop}

\begin{proof}
We first construct an example in socle degree 4.  The $h$-vector $\underline{h}' = (1,4,10,20,35)$ is a pure $O$-sequence since it is the $h$-vector of the truncation of a polynomial ring in 4 variables, $w,x,y,z$; the pure order ideal arises using 35 monomials of degree 4 in $w,x,y,z$.  The $h$-vector $\underline{h}'' = (1,4,6,4,1)$ is also a pure $O$-sequence, since it arises as the pure order ideal of a monomial $abcd$ in 4 new variables.  Putting these 36 monomials together gives the pure $O$-sequence $\underline{h} = (1,8,16,24, 36)$.  Since the first difference is $(1,7,8,8,12)$, which is not an $O$-sequence (since $12>(8_{(3)})^1_1=10$), we obtain our desired example.

For socle degree 5 we make a similar construction using the pure $O$-sequences $\underline{h}' = (1,3,6,10,15,21)$ and $\underline{h}'' = (1,3,5,5,3,1)$ to obtain the pure $O$-sequence $\underline{h} = (1, 6, 11, 15, 18, 22)$. Hence  $\underline{h}$ is not differentiable, since its first difference is $(1,5,5,4,3,4)$, and  $4>(3_{(4)})^1_1=3$.

For socle degree $a \geq 6$ we choose the pure $O$-sequence
\[
\underline{h}' = (1,3,6,9,12,\dots,3a-6, 3a-3, 3a)
\]
  (which is the Hilbert function of the truncation of the algebra $R/(x^3)$ after degree $a$, where $R$ is a polynomial ring in 3 variables) and a pure $O$-sequence $(\dots 6,3,1)$ arising from a single monomial (in  three new variables) of degree $a$ with each variable having an exponent at least equal to 2.  These combine to give the pure $O$-sequence $(\dots, 3a, 3a, 3a+1)$, which is clearly not differentiable.
\end{proof}

Note that all the examples used in the proof of Proposition \ref{4} fail to have the WLP, precisely because the Hilbert function $\underline{h}$ is not differentiable while it is non-decreasing.  In contrast, the next result is an (in general sharp) upper bound on the growth of the first difference of a non-decreasing pure $O$-sequence of any socle degree, which also has a nice algebraic corollary concerning the WLP. For this result we do not need to make any assumption on the base field, even though $g$-elements only have to exist in characteristic zero --- see Remark \ref{g-remarks}.

\begin{prop}\label{ub}
Let $\underline{h}=(1,h_1,h_2,...,h_e)$ be a  pure $O$-sequence, with $e \geq 2$. Then, for any index $i$ satisfying $2 \leq i \leq e$,
$$h_i-h_{i-1}\leq ((h_2-h_1)_{(2)})^{i-2}_{i-2}.$$
Moreover, if equality holds in some degree $j$, then it holds in all
degrees $2\le i\leq j$. Also, if $X\subset R=k[x_1,...,x_{h_1}]$ is
a pure order ideal with $h$-vector $\underline{h}$,
then the multiplication by a general linear form of $R$ between the
vector spaces spanned  by the monomials of $X$ of degree $i$ and the
monomials of degree $i+1$, is injective for all $i\leq j-1$.
\end{prop}

\begin{proof}
The statement is trivial if $e = 2$, so we will assume that $e \geq 3$.  Let $I$ be a monomial Artinian  ideal such that $R/I$ has Hilbert function $\underline{h}$.  Let $L$ be a general linear form.  Consider the exact sequence
\[
0 \rightarrow [((I:L)/I)(-1)]_i \rightarrow  [(R/I)(-1)]_i \stackrel{\times L}{\longrightarrow} [R/I]_i \rightarrow [R/(I,L)]_i \rightarrow 0.
\]
From Hausel's theorem (Theorem \ref{hausel}), we know that $[I:L]_j = I_j$ for $j \leq \lfloor \frac{e-1}{2} \rfloor$.  It then follows from the above sequence that for $i \leq \lfloor \frac{e+1}{2} \rfloor$,
\[
h_i - h_{i-1} = \dim [R/(I,L)]_i.
\]
This is true in particular for $i = 2$.  But more generally, for all $i$ we have
\[
h_i - h_{i-1} \leq \dim [R/(I,L)]_i
\]
(with equality if and only if the kernel, $[((I:L)/I)(-1)]_i$, is zero).  Then a repeated application of Macaulay's theorem (Theorem \ref{macaulay}) gives for all $i \geq 3$
\[
h_i - h_{i-1} \leq (\dim[R/(I,L)]_{(2)})^{i-2}_{i-2} = ((h_2 - h_1)_{(2)})^{i-2}_{i-2}
\]
as claimed.  If equality holds in some degree $j$ then it represents maximal growth for $R/(I,L)$ at each step from degree 2 to degree $j$, so we have equality at every step.  Furthermore, if we have equality for some $j$ then we must have $(I:L) = I$ in all the degrees $\leq j-1$, so multiplication by a general linear form gives an injection in those degrees.
\end{proof}

\begin{remark}{\em
Notice that the last part of the statement of the above proposition can be
rephrased algebraically in the following way, which yields a useful
corollary concerning the WLP:

\vspace{.5cm}

{\em Let $A$ be a
monomial Artinian  level algebra of socle degree $e$ with a differentiable $h$-vector, $\underline{h}$, and suppose that the first difference of $\underline{h}$ has
maximal growth from degree 2 on. Then $A$ enjoys the WLP.}
\vspace{.5cm}

In contrast, in Chapter \ref {type 2 arb var} we will explore situations where a monomial algebra $A$ has a quotient that fails to have the WLP, focusing on the case of type 2.  Earlier examples of this phenomenon (but not having type 2) can be found in \cite{BK} and in \cite{MMN2}.

}\end{remark}

\vskip 2mm
We now exhibit the first non-unimodal pure $O$-sequence of socle degree 4. Our idea uses the same techniques employed in constructing the non-differentiable examples in the proof of Proposition \ref{4}. In particular, the existence of a non-unimodal $O$-sequence of socle degree 4 proves that the largest socle degree forcing all pure $O$-sequences to be unimodal is exactly 3: indeed, by Hibi's flawlessness result, any pure $O$-sequence $\underline{h}=(1,r,h_2,h_3)$ of socle degree 3 must have $h_2\geq r$, and is therefore unimodal.  Socle degree $5$ was known to Stanley \cite{St1}, p.60.

\begin{example}\label{nonunimodal soc deg 4}
{\em We give an example of non-unimodal pure $O$-sequence with socle degree 4. In fact, we will check that
$$\underline{h}=(1,49,81,79,81)$$ is a pure $O$-sequence. Indeed, $\underline{h}'=(1,5,15,35,70)$ is a pure $O$-sequence, since it is the $h$-vector of the truncation of a polynomial ring in 5 variables, and $\underline{h}''=(1,4,6,4,1)$ is also pure, for it corresponds to the maximal monomial $y_1y_2y_3y_4\in k[y_1,y_2,y_3,y_4]$. Hence, reasoning as in the proof of Proposition \ref{4}, we now consider one copy of $\underline{h}'$ and eleven copies of $\underline{h}''$ as $h$-vectors of pure $O$-sequences in twelve different rings, and we work in the tensor product of those rings. It follows that $$\underline{h}=(1,5,15,35,70)+ (0\cdot 1,11\cdot 4,11\cdot 6,11\cdot 4,11\cdot 1)=(1,49,81,79,81)$$ is a non-unimodal pure $O$-sequence, as desired.
}\end{example}

If we allow the socle degree to be large, we now show that pure $O$-sequences of any codimension can be non-unimodal with arbitrarily many peaks. This result answers positively Question iii) (section 6) of \cite{BZ}.

\begin{theorem} \label{arb many peaks}
Let $M$ be any positive integer and fix an integer $r \geq 3$.  Then there exists a pure $O$-sequence in $r$ variables which is non-unimodal, having exactly $M$ maxima.
\end{theorem}

\begin{proof}
We prove the result using inverse systems.  The main task is to prove the result in codimension 3.  Having done this, by adding a suitable power of a new variable the number of variables increases by one and the values of the Hilbert function all increase by one, so repeating this we obtain an example with $M$ maxima in any number $\geq 3$ of variables. Also, it is enough to prove that we can get {\em at least} $M$ maxima, since the truncation of a pure $O$-sequence is again a pure $O$-sequence.

In order to prove the result for $r=3$, we begin with the numerical computation of \cite{Za}, Remark 5.  With this as our guide, we have only to choose suitable inverse systems.  However, because the number of variables is so small, we have to use some care.

So we first recall the numerical computation.  For any positive integer $N \geq 1$, we will recall how a level algebra with $N$ maxima was constructed in \cite{Za}.  (In our modification, $N$ will not be equal to $M$.)  Consider the $h$-vector
\[
\underline{h}=(1,3,h_2,...,h_e),
\]
 where the  socle degree $e$ is  large enough  that, for some integer $t$, we have
 \[
 \begin{array}{lllllllllll}
 h_e=t, & h_{e-1}=t-3, &  h_{e-2}=t-6,\\
 h_{e-3}=t-9, & h_{e-4}=t-15, & h_{e-5}=t-21,\\
 h_{e-6}=t-27, &  h_{e-7}=t-36, & h_{e-8}=t-45, \\
 h_{e-9}=t-54
 \end{array}
 \]
  and so on, down to
  \[
  \begin{array}{llllllllllllllll}
  h_{e-3(N-2)-1}=t-{(3(N-2)+1)+2\choose 2}, \\  h_{e-3(N-2)-2}=t-{(3(N-2)+1)+2\choose 2}-3(N-1), \\ h_{e-3(N-2)-3}=t-{(3(N-2)+1)+2\choose 2}-6(N-1),  \hbox{and } \\ h_{e-3(N-2)-j}={e-3(N-2)-j+2\choose 2} \hbox{  for all $j\geq 4$.}
  \end{array}
  \]
Adding a general form of degree $e$ to an inverse system module generating $\underline{h}$, it is shown in \cite{Za} that  we obtain the level $h$-vector
\[
\begin{array}{rcllll}
\underline{H} & = & (1,3,6,...,H_{e-3(N-2)-4}={(e-3(N-2)-4)+2\choose 2},
\ \  H_{e-3(N-2)-3}=t+1, \\
&& H_{e-3(N-2)-2}=t, \ \ H_{e-3(N-2)-1}=t, \ \ t+1,t,t,t+1,...,t,t,t+1),
\end{array}
\]
because the two inverse systems are ``as independent as possible" in $R$.
In our situation, we have to show that everything can be done with monomials.  The main steps are producing the inverse system module, and producing a monomial that mimics the behavior of a general form of degree $e$ ``long enough.''

Notice that for any choices of $e$, $N$ and $t$ satisfying the above conditions, the $h$-vector $\underline{h}$ is a differentiable $O$-sequence.  Choose a monomial ideal $I_1$ in $k[x,y]$ whose Hilbert function is the first difference $\Delta \underline{h}$.  Let $I_2 = RI \subset R$, where $R = k[x,y,z]$.  The minimal generators of $I_1$ also generate $I_2$, but the Hilbert function of $I_2$ is $\underline{h}$ up to degree $e$.  Adding $(x,y,z)^{e+1}$ gives us a monomial algebra with Hilbert function exactly $\underline{h}$.  This is our inverse system module generating $\underline{h}$ (once we choose $N$ and $e$).

Choose an integer $B > \max \{ M,2 \}$.
Let $N$ be an integer satisfying $N \geq$ \linebreak $ 3(B-2)+4$.  Choose $e$ large enough, according to the set-up given above. It follows that
\[
e-(3(N-2)+3) \leq e-3(3(B-2)+3).
\]
Note that $3(B-2)+3 >0$ , since $B > 2$.

Let $F$ be a monomial in $I_2$ of degree $e - 3(3(B-2)+3)$, which we see must exist by considering the value of $h_{e-3(N-2)-3}$.  Notice that the element $G = F \cdot x^{3(B-2)+3} y^{3(B-2)+3} z^{3(B-2)+3}$ is clearly not in the inverse system module, since it is in $I_2$.  For the same reason, all the derivatives up to the $[3(B-2)+3]$-th are not in the inverse system module.  Since the derivatives of $G$ up to the $[3(B-2)+3]$-th \linebreak ones grow maximally, $G$ acts in the same way as the general form for its first $3(B-2)+3$ derivatives.  Then the computations from \cite{Za} guarantee that we have {\em at least} $B > M$ maxima.  Then we truncate the Hilbert function to obtain exactly $M$ maxima.
\end{proof}

\begin{example} \label{type 14}
Using the approach of the previous proof, we can produce the following infinite family of non-unimodal examples of type 14.  We are not aware of any non-unimodal example in three variables with a smaller  type.

Starting from $n=2m$ for $m\ge 11$, and from $n=2m+1$ for $m\ge 14$, we have
an infinite class of non-unimodal pure $O$-sequences with socle degree $e=4n$ and
type $14$. This class is given by the inverse system  
$$
\begin{array}{rl}
  M=\{&
x^{2m}y^{2m}z^{4m}, 
x^{m-2}y^{6m+3}z^{m-1}, 
x^{2m-3}y^{5m+4}z^{m-1}, 
x^{3m-4}y^{4m+5}z^{m-1}, \\&
x^{4m-5}y^{3m+6}z^{m-1},
x^{5m-6}y^{2m+7}z^{m-1}, 
x^{6m-7}y^{m+8}z^{m-1}, 
x^{7m-8}y^{9}z^{m-1}, \\& 
x^{m-2}y^{5m+3}z^{2m-1},
x^{2m-3}y^{4m+4}z^{2m-1}, 
x^{3m-4}y^{3m+5}z^{2m-1}, \\& 
x^{4m-5}y^{2m+6}z^{2m-1}, 
x^{5m-6}y^{m+7}z^{2m-1},
x^{6m-7}y^{8}z^{2m-1}
\}
\end{array}
$$
when $n=2m$, and by the inverse system
$$
\begin{array}{rl}
  M=\{&
x^{2m+1}y^{2m+1}z^{4m+2}, 
x^{m-2}y^{6m+6}z^{m}, 
x^{2m-3}y^{5m+7}z^{m}, 
x^{3m-4}y^{4m+8}z^{m}, \\&
x^{4m-5}y^{3m+9}z^{m},
x^{5m-6}y^{2m+10}z^{m}, 
x^{6m-7}y^{m+11}z^{m}, 
x^{7m-8}y^{3}z^{m}, \\& 
x^{m-1}y^{5m+2}z^{2m},
x^{2m-1}y^{4m+2}z^{2m}, 
x^{3m-1}y^{3m+2}z^{2m}, \\& 
x^{4m-1}y^{2m+2}z^{2m}, 
x^{5m-1}y^{m+2}z^{2m},
x^{6m-1}y^{2}z^{2m}
\}
\end{array}
$$
when $n=2m+1$.

  In order to prove this, notice  that, in either inverse system, the generator $x^ny^nz^{2n}$ gives maximal growth from
  degree $4n$ to degree $3n$. Hence, in degrees $3n$, $3n+1$, $3n+2$
  and $3n+3$, it contributes to the Hilbert function respectively by
 \begin{multline*}\textstyle
  \binom{n+2}{2}=\frac{(n+2)(n+1)}2, 
  \binom{n+1}{2}=\frac{(n+1)n}2,
  \binom{n}{2}=\frac{n(n-1)}2, 
  \binom{n-1}{2}=\frac{(n-1)(n-2)}2.
\end{multline*}
 All monomials in these degrees are divisible by $z^n$.
 We claim that the other thirteen generators of $M$ yield all the monomials not divisible by
 $z^{n}$, hence contributing to the Hilbert function in the corresponding degrees by 
 $$
 \binom{3n+2}2-\binom{2n+2}2=\frac{5n^2+3n}2,
 \binom{3n+3}2-\binom{2n+3}2=\frac{5n^2+5n}2,$$
$$ \binom{3n+4}2-\binom{2n+4}2=\frac{5n^2+7n}2,
 \binom{3n+5}2-\binom{2n+5}2=\frac{5n^2+9n}2.
$$

Since these sets of monomials are disjoint,  the  Hilbert function in those degrees is, respectively:
$$\begin{array}{rcl}
\frac{(n+2)(n+1)}2 + \frac{5n^2+3n}2&=& 3n^2+2n+1 , \\
 \frac{(n+1)n}2 + \frac{5n^2+5n}2 &= & 3n^2+2n, \\
\frac{n(n-1)}2 + \frac{5n^2+7n}2 &= &3n^2+2n,\\
\frac{(n-1)(n-2)}2 + \frac{5n^2+9n}2 & =&3n^2+2n+1,
\end{array}
$$
which  clearly gives non-unimodality.

It remains to show that the last thirteen generators indeed give all the
monomials not divisible by $z^n$ in degrees at most $3n+3$. In order to do
this, it is sufficient to show that they give all the monomials not
divisible by $z^n$ in degree exactly $3n+3$. 

We start with the case $n=2m$, and $m\ge 11$. 
 For $i=0,1,\dots,6$, the  $(m-1)m$ monomials of degree $6m+3$ with
 $x$-degrees $i(m-1),i(m-1)+1,\dots ,i(m-1)+m-2$ and $z$-degrees
 $0,1,\dots ,m-1$  
 are divisors of  $x^{(i+1)(m-1)-1}y^{(6-i)m+i+3}z^{m-1}$.  For $i=0,1,\dots,5$, the $(m-1)m$
 monomials of degree $6m+3$ with $x$-degrees
 $i(m-1),i(m-1)+1,\dots ,i(m-1)+m-2$ and $z$-degrees $m,m+1,\dots ,2m-1$
 are divisors of  $x^{(i+1)(m-1)-1}y^{(5-i)m+i+3}z^{2m-1}$.   
 
 In order for these to be all the monomials of degree $6m+3$ and $z$-degrees $0,1,\dots ,m-1$, one checks that for
 $i=6$, we get the monomials of $x$-degrees $6m-6,6m-5,\dots ,7m-8$,
 and indeed we get all such monomials when $7m-8\ge 6m+3$, i.e.,  $m\ge
 11$. For the monomials of $z$-degrees $m,m+1,\dots ,2m-1$, one checks
 that for $i=5$, we get $x$-degrees $5m-5,5m-4,\dots ,6m-7$, which are
 all the desired monomials when $6m-7\ge 5m+3$, i.e.,  $m\ge 10$. 

Now we consider the case $n=2m+1$, for $m\ge 14$. For $i=0,1,\dots,6$, the $(m-1)(m+1)$ monomials of degree $6m+6$ with
 $x$-degrees $i(m-1),i(m-1)+1,\dots ,i(m-1)+m-2$ and $z$-degrees
 $0,1,\dots ,m$   are divisors of $x^{(i+1)(m-1)-1}y^{(6-i)m+i+6}z^{m}$. For $i=0,1,\dots ,5$, the $m^2$
 monomials of degree $6m+6$ with $x$-degrees
 $im,im+1,\dots ,im+m-2$ and $z$-degrees $m,m+1,\dots ,2m$
 are divisors of the monomial $x^{(i+1)m-1}y^{(5-i)m+i+5}z^{2m}$. 

  In order for these to be all the monomials of degree $6m+6$ with $z$-degrees $0,1,\dots ,m-1$, one checks that for
 $i=6$, we get the monomials of $x$-degrees $6m-6,6m-5,\dots ,7m-8$,
 and indeed we get all such monomials when $7m-8\ge 6m+6$, i.e.,  $m\ge
 14$. To get all the monomials of degree $6m+6$ with $z$-degrees
 $m,m+1,\dots ,2m$, one checks that for $i=5$, we get the monomials of
 $x$-degrees $5m,5m+1,\dots ,6m-2$ and these are all the desired monomials  if $6m-2\ge 5m+6$,
 i.e.,  $m\ge 8$. 
 
 The smallest examples of this family correspond to $n=22$, where our Hilbert function  is 
$$
\begin{array}{l}
 (1, 3, 6, 10, 15, 21, 28, 36, 45, 55, 66, 78, 91, 105, 120, 136, 153,
 171, 190, 210, 231, 253, \\ \qquad
 276, 300, 325, 351, 378, 406, 435, 465, 496, 528, 561, 595, 630, 666,
 703, 741, 780,  \\ \qquad
 820, 861, 903, 946, 990, 1035, 1078, 1119, 1158, 1195, 1230, 1263, 1294, 1323, \\ \qquad
 1350, 1375, 1398, 1419, 1438, 1455, 1470, 1483, 1494, 1503, 1510, 1515, 1518, \\ \qquad
 \textbf{1519, 1518, 1518, 1519}, 1507, 1481, 1441, 1387, 1319, 1237, 1141, 1031, 907,  \\ \qquad
 769, 630, 504, 392, 294, 210, 140, 84, 42, 14),
    \end{array}
    $$
and to $n=29$, where our Hilbert function  is
$$
\begin{array}{l}
(1, 3, 6, 10, 15, 21, 28, 36, 45, 55, 66, 78, 91, 105, 120, 136, 153,
 171, 190, 210, 231, 253, \\ \qquad
  276, 300, 325, 351, 378, 406, 435, 465, 496, 528, 561, 595, 630,
  666, 703, 741,  780, \\ \qquad
  820, 861, 903, 946, 990, 1035, 1081, 1128,  1176, 1225, 1275, 1326, 1378,
  1431,  \\ \qquad
  1485, 1540, 1596, 1653, 1711, 1770, 1827, 1882, 1935, 1986, 2035, 2082, 2127, \\ \qquad
  2170, 2211, 2250, 2287, 2322, 2355, 2386, 2415, 2442, 2467, 2490, 2511, 2530, \\ \qquad
  2547, 2562, 2575, 2586, 2595, 2602, 2607, 2610, \textbf{2611, 2610, 2610, 2611}, 2599,   \\ \qquad
  2574, 2536, 2485, 2421, 2344,  2254, 2151, 2034, 1903, 1758, 1599, 1426, 1246,   \\ \qquad 
  1071, 909, 760, 624, 501, 391, 294, 210, 140, 84, 42, 14).
\end{array}
$$
\end{example}

\begin{remark}\label{soc deg 16}{\em 
Example \ref{type 14} produced the smallest socle type for which we are aware that a non-unimodal example exists in codimension 3.  Another challenge is to find the smallest socle degree for which a non-unimodal example exists.  The best example of which we are aware is Example 5.1 of \cite{BZ}, where it is shown that
\[
h = (1, 3, 6, 10, 15, 21, 28, 33, 36, 37, 36, 36, 37)
\]
is a pure $O$-sequence; notice that this has socle degree 12. }
\end{remark}

{\ }\\
\\
In \cite{Mig}, an example is constructed of a level set of points in   $\mathbb P^3$ with the property that an Artinian reduction (hence any Artinian reduction) has a non-unimodal Hilbert function, but it is remarked that no example is known of more than two maxima.  (By this we mean that there is at most one ``valley" in the sense of  \cite{Bo}.)  In \cite{BZ}, a monomial example is constructed (which can be lifted to reduced sets of points -- cf.\ for instance \cite{MN}), but again it is only done for two maxima.  Theorem \ref{arb many peaks} allows us to extend these results as follows.

\begin{corollary}
In any projective space $\mathbb P^r$, $r \geq 3$, there are reduced sets of points whose Artinian reduction is level with a non-unimodal Hilbert function having arbitrarily many maxima (i.e.\ arbitrarily many ``valleys'').
\end{corollary}

Theorem \ref{arb many peaks} is related to Michael and Traves' ``Roller-Coaster conjecture'' (see \cite{MT}) for the $f$-vectors of well-covered graphs (which make a considerably smaller subset of pure $O$-sequences).  Notice, however, that the argument we have just given to prove our ``roller-coaster'' result for arbitrary pure $O$-sequences (Theorem \ref{arb many peaks}) would not apply to prove the Michael-Traves conjecture on independence polynomials of well-covered graphs (in fact, it does not seem to be helpful, more generally, when we restrict to any class of $f$-vectors of simplicial complexes). Indeed, in those contexts, all monomials involved in the ``inverse system'' need to be squarefree, whereas our argument for arbitrary pure $O$-sequences essentially relies on the fact  that we can choose at least one of the monomials to be very far from squarefree.

\begin{remark}
{\em Notice that the socle type of the pure $O$-sequences produced in Theorem~\ref{arb many peaks} is very large.  It is interesting to ask what constraints (e.g.\ for unimodality) are placed by asking that the type be small.  We begin this inquiry with a study of type 2 pure $O$-sequences in Chapters \ref{type 2 three var} and \ref{type 2 arb var}.
}\end{remark}

%%%%%%%%%%%%%%%%%%%%%%%%%%%%%%%%%%%%%%%%%%%%%%%%%%

\chapter{The Interval Conjecture for Pure $O$-sequences} \label{ICP section}

In this chapter we present the Interval Conjecture for Pure $O$-sequences (ICP), and begin its study. As we said in the Introduction, especially in view of the many non-unimodality results, it seems nearly impossible to find a complete characterization of the set of pure $O$-sequences; therefore, the ICP may possibly be the most powerful result that can be proved short of a characterization, since it provides a particularly strong form of regularity for the structure of pure $O$-sequences, and could thus reveal itself to be the  successful perspective to approach this problem.

Also, even though it is too early to predict the scope of the applications of the ICP to other areas, some initial indications seem quite promising. First, in the second portion  of this chapter, we show how the ICP provides a very short  alternative proof of Stanley's matroid $h$-vector conjecture for matroids of rank 2 (or dimension 1). Second, after a first draft of this monograph  was written, Stanley's conjecture has been the focus of a paper  of T. H\'a, E. Stokes and the fifth author \cite{TZ}, where the ICP has proved to be one of the crucial tools to show the conjecture for the class of all matroids of rank 3.

\begin{conj}\label{ICP}
(The Interval Conjecture for Pure $O$-sequences (ICP)). Suppose
that,  for some positive integer $\alpha $, both
$(1,h_1,...,h_i,...,h_e)$ and $(1,h_1,...,h_i+\alpha ,...,h_e)$ are
pure $O$-sequences. Then $(1,h_1,...,h_i+\beta ,...,h_e)$ is also a
pure $O$-sequence  for each integer $\beta =0,1,..., \alpha .$
\end{conj}

\begin{remark}\label{rrr} {\em
The last author \cite{Za2} had recently formulated the analogous
conjecture for the class of Hilbert functions of Artinian level
algebras, of which pure $O$-sequences are a subset (they correspond
to {\em monomial} Artinian level algebras). However, notice that a
similar  Interval Conjecture would not hold if, instead, we wanted
to consider the (conjecturally smaller, according to Stanley) set of
$h$-vectors of matroid complexes. For instance (see Stokes's thesis
\cite{Sto}, p.51), $(1,4,4)$ and $(1,4,6)$ are $h$-vectors of
matroids, but $(1,4,5)$ is {\em not}.
}\end{remark}

The main result of  this chapter is that  the ICP holds true for
pure $O$-sequences $\underline{h}$ of socle degree $3$.

\begin{theorem}\label{ICP3}
The set $S$ of pure $O$-sequences of  length 4 (i.e. of socle degree
3) satisfies the ICP, i.e.,

\begin{enumerate}
\item[$i)$]
if $(1,h_1',h_2,h_3)$ and $(1,h_1'',h_2,h_3)$ are in $S$ and $h'_1\le h_1\le h_1''$, or
\item[$ii)$]
if $(1,h_1,h_2',h_3)$ and $(1,h_1,h_2'',h_3)$ are in $S$ and $h'_2\le h_2\le h_2''$, or
\item[$iii)$]
if $(1,h_1,h_2,h_3')$ and $(1,h_1,h_2,h_3'')$ are in $S$ and $h'_3\le h_3\le h_3''$,
\end{enumerate}
then $(1,h_1,h_2,h_3)$ is in $S$.
\end{theorem}

\begin{proof}
We will start by proving the interval property for $h_3$, then
proceed to $h_1$ and end by proving it for $h_2$, since we will
refer back to the previously proved parts.

\vskip 2mm $iii)$ Given $h_1=r$ and $h_2=a$, there is a  maximal
possible value for $h_3=b_{\max }$, \label{bmax} and a minimal possible value
$h_3=b_{\min }$. We have already proved that $(1,h_1,h_2,h_3)$ is
differentiable if it is non-decreasing (see Theorem \ref{3}).
Furthermore, since it is flawless by Hibi's Theorem \ref{hibi}, we
have that $h_2\ge h_1$. We have also proved that any differentiable
sequence is a pure $O$-sequence (see Corollary \ref{diff}). Hence we
know that any value for $h_3$ in the range $a = h_2\le h_3\le
b:_{\max }$ is possible.

We now prove that any value for $h_3$ in the  range $b_{\min }\le
h_3\le h_2=a$ is possible. Let $I$ be a monomial ideal with inverse
system $I^\perp$ such that the Hilbert function of $R/I$ is
$(1,r,a,b_{\min })$. Any monomial in $I^\perp$ has to have at least
one derivative which is unique to this monomial, by the minimality
of $b_{\min }$ (otherwise we can remove that monomial and lower $h_3$
without changing $h_1$ or $h_2$).   We will use two specific
consequences stemming from this remark.

\begin{enumerate}
\item If $x_i^2 x_j$ is in $I^{\perp}$ then $x_i^3$ is not
there, since $x_i^3$ would not have a derivative unique to it.

\item If $x_i x_j x_k$ is in $I^\perp$ then without
loss of generality we can assume that $x_j x_k$ is not  the
derivative of any other monomial. In this case we say that the
derivative $x_j x_k$ is {\em determined} by $x_i x_j x_k$.
\end{enumerate}

We will now present an algorithm which will  begin with $I^\perp$
and change it, step by step, so that we get each  pure $O$-sequence
$(1,r,a,b)$ with $b_{\min }\le b\le a$.  We will go through the degree
three monomials in $I^\perp$ one by one.  Some we will keep, and
some will be replaced, to form new inverse systems for which $h_1$
and $h_2$ remain the same, and $h_3$ either remains the same or
increases by one.  In the end we will arrive at $h_3 = a$ and we
will be finished.

We start with the cubes, and we keep them.  Then  consider monomials
in $I^\perp$ that are divisible by a square but not by a cube.  For
any such monomial $x_i^2 x_j$, we can add $x_i^3$ (if it is not
already there) without changing $h_1$ or $h_2$.  Initially $x_i^3$
will not be in $I^\perp$, as noted in (1) above, but after some
number of steps it may be there.  When we have finished running
through all {\em original} monomials in $I^\perp$ of the form
$x_i^3$ and $x_i^2 x_j$, let $B_1$ be the set of {\em original}
monomials after considering which we did not add a new monomial, and let $B_2$ be the
set of {\em original} monomials after considering which we did add a new monomial.
%The inverse system spanned by $B_1$ and $B_2$ has $|B_1| + 2 | B_2|$
%derivatives, and the inverse system we have after the addition of
%the $|B_2|$ new cubes has exactly the same number of derivatives.

We now proceed to the  squarefree monomials in $I^\perp$.  Let $x_i
x_j x_k$ be such a monomial.  By (2), we may assume that $x_j x_k$
is not the derivative of any other monomial in $I^\perp$.   If
$x_i^2$ is not among the derivatives of the previous monomials, we
can replace $x_ix_jx_k$ by $x_i^2x_j$, $x_i^2x_k$ and then add
$x_i^3$. In doing this, the number of degree three monomials
increases by two, in two steps, while we keep the number of
derivatives, since $x_jx_k$ is replaced by $x_i^2$. We denote by $B_3$ the
set of such monomials.

If $x_i^2$ is among the derivatives, we  can add $x_i^2x_j$ and
$x_i^2x_k$ if those have not yet appeared. We divide the set of such
monomials in $I^\perp$ into three subsets, $B_4$, $B_5$ and $B_6$,
depending on whether we add zero, one or two monomials in the
process. Denote by $J^{\perp}$ the resulting inverse system.

In order to see that we will reach  $h_3=h_2=a$ in this process, we
make the following table which summarizes what happens during the
algorithm:

$$\begin{array}{l|l|l}
  \hbox{\it Original}& \hbox{\it monomials} & \hbox{\it derivatives in } J^\perp_2\\
  \hbox{\it monomial}& \hbox{\it in }J^\perp_3 &  \hbox{\it determined by these}\\ \hline
  x_i^3 & x_i^3 & x_i^2\\
  x_i^2x_j\in B_1 & x_i^2x_j  & x_ix_j\\
  x_i^2x_j\in B_2 & x_i^3,x_i^2x_j & x_i^2,x_ix_j\\
  x_ix_jx_k\in B_3 & x_i^3,x_i^2x_j,x_i^2x_k & x_i^2,x_ix_j,x_ix_k\\
  x_ix_jx_k\in B_4 & x_ix_jx_k & x_jx_k\\
  x_ix_jx_k\in B_5 & x_ix_jx_k,x_i^2x_j & x_jx_k,x_ix_j\\
  x_ix_jx_k\in B_6 & x_ix_jx_k,x_i^2x_j,x_i^2x_k & x_jx_k,x_ix_j,x_ix_k\\
\end{array}
$$

\vskip 2mm \noindent Since in all cases we have exactly as many
derivatives  in $J^\perp_2$ as we have monomials in $J^\perp_3$, we
end up with $\dim_k J^\perp_2 = \dim_k J^\perp_3$, i.e.\ $h_2=h_3$.
These proves the ICP for $h_3$ by considering the modifications in
the cases where we replace a monomial in $I^\perp$ by 3 monomials as
a 2-step process where the first step is  replacing the original
monomial by two monomials and the second step is adding the third
monomial.

\vskip 2mm $i)$
Given $h_2=a$ and $h_3=b$, we know from Theorem \ref{hibi} that
$r = h_1\le h_2=a$. Thus if $r = a$ there is nothing to prove, and we may assume $r<a$.
 If $b\ge a$, the sequence is non-decreasing and
hence differentiable by Theorem~\ref{3}. Since, as we will prove
later  (see Proposition \ref{gor}(ii)), the set of differentiable
$O$-sequences has the interval property with respect to any
coordinate, we can now assume that $b<a$.

%By choosing an inverse system with no common  factors, we can always
%achieve $h_1=h_2$. In fact, we can take $\lfloor (a-b)/2\rfloor$
%squarefree monomials, one monomial of type $x_i^2x_j$ if $a-b$ is
%odd, and add cubes until we have $b$ monomials, with no common
%factors.

The lower bound for $h_1$ given $h_2=a$ and  $h_3=b<a$ is not known
from before, so we assume that we are given an inverse system
$I^\perp$ which achieves the minimal possible value for
$h_1=r_{\min }$.

Note that if  $b=1$ then $r_{\min }=r_{\max }=a$  and so there is
nothing to prove.

We will now argue that if $h_1=r<h_2=a$ and  $1<h_3=b<h_2$, we can
change $I^{\perp}$ into some other inverse system $J^{\perp}$ whose
$h$-vector is $(1,r+1,a,b)$. When removing one degree three monomial, $m$,
from the inverse system, it is easy to see that we decrease the $h$-vector by one of the
following:
\[
\begin{array}{ccccc}
(0,0,0,1),  &(0,0,1,1), &(0,1,1,1), &(0,0,2,1), &(0,1,2,1),\\   (0,2,2,1),   &(0,0,3,1),     &(0,1,3,1),    &(0,2,3,1), &(0,3,3,1).
\end{array}
\]
In each of the cases, we propose an action  in order to see that we
can increase $h_1$ by one, from $h_1=r$ to $h_1=r+1$, while
preserving $h_2$ and $h_3$. If $h_1$ decreases when we remove the
monomial $m$, we know that at least one of the variables dividing
$m$ was unique to $m$ and does not divide any other monomial in the
inverse system. In most cases, the action is to replace $m$ with a
monomial involving some new variables. We summarize this in the
following table:

\vskip 2mm
$$
\begin{tabular}{r|c|p{3in}}
\em Case&\em Difference & \em Action\\ \hline
1&$(0,0,0,1)$   & Remove $m$.     \\
2&$(0,0,1,1)$   & Replace $m$ by $x_i^3$, where $x_i$ is a new variable.\\
3&$(0,1,1,1)$   & Keep $m$.\\
4&$(0,0,2,1)$   & Replace $m$ by $x_i^2x_j$, where $x_i$ is new and $x_j$ old.   \\
5&$(0,1,2,1)$   & Replace $m$ by $x_i^2x_j$, where $x_i$ is new and $x_j$ was unique to $m$.    \\
6&$(0,2,2,1)$   & Keep $m$.     \\
7&$(0,0,3,1)$   & Replace $m$ by $x_ix_jx_k$ where $x_i$ is new and $x_jx_k$ was unique to $m$. \\
8&$(0,1,3,1)$   & Replace $m$ by $x_ix_jx_k$, where $x_i$, $x_j$ are
new and $x_k$ was unique to $m$.  \\
9&$(0,2,3,1)$   & Replace $m$ by $x_ix_jx_k$ where $x_i,x_j$ and $x_k$ are new.  \\
10&$(0,3,3,1)$  & Keep $m$. \\
\end{tabular}
$$

\vskip 2mm \noindent In cases $2,4,5,7,8$ and $9$, we  manage to
increase $h_1$ to $r+1$, fixing $h_2=a$ and $h_3=b$. Thus, we are
done if $I^{\perp}$ contains any monomial of type 2, 4, 5, 7, 8, or
9.  We assume that this is not the case.

Consider now the monomials in $I^\perp$ of types  6 and 10 (i.e. belonging to cases 6 and 10).  A monomial $m$ of type 10 is of the form $xyz$, where all variables are unique to $m$.  A monomial of type 6 must be of the form $x^2y$ where both variables are unique to $m$.  Removing all monomials of types 6 and 10 (we will replace them later), assume without loss of generality that $I^\perp$ contains only monomials of types 1 and 3.

Let $m$ be a monomial of type 3.  Then $m$ must either  be of the form $xy^2$ with $x$ unique to $m$ but $y^2$ not unique to $m$, or of the form $x^3$ with $x$ unique to $m$.  We remove all of the monomials of the latter form (we will replace them later).  Then we can partition the type 3 monomials in  the new $I^\perp$ into subsets $A_j$, each of which corresponds to a fixed square $x_j^2$, and the monomials  in $A_j$ are of the form $m = x_i x_j^2$ with $x_i$ unique to $m$ but $x_j^2$ not unique to $m$.  In addition to these, there may be some type 1 monomials, and these must be of the form $x_j^3$ for some $x_j$ appearing in a type 3 monomial.  Now we count.  The  first derivatives of the monomials in $A_j$ are $x_j^2$ and all of the $x_ix_j$, where $x_j$ is fixed for $A_j$ and the $x_i$ occur only in one monomial.  The second derivatives are all the unique $x_i$, plus $x_j$.  Hence the number of first derivatives and the number of second derivatives from the monomials in $A_j$ are both $|A_j|+1$.  Given two of the subsets $A_{j_1}$ and $A_{j_2}$, the first derivatives of the monomials in $A_{j_1}$ are distinct from those of $A_{j_2}$, and similarly for the second derivatives.  Hence the type 3 monomials
produce the same number of first derivatives as second derivatives.  Adding in the type 1 derivatives as described above does not produce any new first or second derivatives.  Thus it follows that for the new inverse system $I^\perp$, we have $h_1 = h_2$, and hence working backwards (replacing the removed monomials) we have $r=a$, contradicting our assumption $r<a$.
It follows that in order to have $r<a$ there must be at least one monomial of types 2, 4, 5, 7, 8 or 9.
This completes the argument for the ICP of $h_1$.

\vskip 2mm $ii)$  We now prove the interval property  for $h_2$. If
$r \leq 2$, this is not too difficult.  If $r=3$, it is not hard to verify that the only possible pure $O$-sequences of length 4 are
\[
\begin{array}{ccccccccccc}
(1,3,3,1) & (1,3,3,2) & (1,3,3,3) \\
(1,3,4,2) & (1,3,4,3) &  (1,3,4,4) &  (1,3,4,5) \\
(1,3,5,3) &  (1,3,5,4) &  (1,3,5,5) & (1,3,5,6) &  (1,3,5,7) \\
(1,3,6,3) & (1,3,6,4) & (1,3,6,5) & (1,3,6,6)\\
(1,3,6,7) & (1,3,6,8) & (1,3,6,9) & (1,3,6,10) .
\end{array}
\]
This clearly satisfies the interval property.

So we assume $r \geq 4$ for the remainder of the proof.
Given $h_1=r$ and $h_3=b$ there is a minimal possible
value for $h_2=a_{\min }$ and a maximal possible $h_2=a_{\max }$. \label{amax} The
first step is to produce an inverse system with $h$-vector $(1, r,
a_{\min }, b)$ that does not contain $x_1^2$ or $x_1 x_2$.

If $b\ge r$, then by Theorem \ref{hibi} it immediately follows that
 $(1,r,a_{\min },b)$ is non-decreasing, and hence we know
by Theorem~\ref{3} that it is differentiable. In this case, the
minimal possible value for $h_2$ allowed by differentiability can in
fact be achieved as follows:

Take the lex-segment ideal, $J$, in $S=k[x_1,x_2,\dots,x_{r-1}]$
with $h$-vector $(1, r-1, a_{\min } - r, b - a_{\min })$ and then put $I
= J \cdot R +(x_1,\ldots,x_r)^{4} \subset R$, where
$R=k[x_1,x_2,\dots,x_r]$. If the inverse system of $J$ is given by
$J^\perp=\bigoplus_{i=0}^3 J^\perp_i$, then we get the inverse
system of $I$ by $I^\perp_3 = \bigoplus_{i=0}^3 x_r^{3-i}J^\perp_i$.
Observe that the generators of the inverse system of $J$ might not
all be in degree three, while $I^\perp$ is generated by $I^\perp_3$.
In each degree, the inverse system $J^\perp_i$ is given by the
lex-segment of smallest monomials. Notice that $x_1 x_2 \in
I_2^\perp$ implies, by construction of $I^\perp$, that $x_1 x_2 \in
J_2^\perp$, thus $x_1 x_2 \notin J$. It follows that $\dim_k J_2
\leq 1$, so equivalently, $a_{\min } - r \geq \binom{r}{2} - 1$, hence
$\binom{r+1}{2} - 1 \leq a_{\min } \leq a_{\max } \le \binom{r+1}{2}$.  Thus the interval property for $h_2$ is trivial in this case. We
conclude that we may assume that $x_1 x_2$ and thus $x_1^2$ are not
in $I_2^\perp$ if $b \geq r$.

If $b< r$, we have that $a_{\min } =r$ (if it exists),  since the sequence is flawless.  If $r =5$ and $b=2$, we take the inverse system generated by $x_2x_3x_4$ and $x_1x_5^2$.  If $r \geq 4$ is arbitrary and $b = r-1$, we take the inverse system generated by $x_2^3,\dots,x_{r-1}^3, x_{r-1}x_r^2$.  Otherwise, we claim that we can use an inverse system generated by
monomials without common factors to achieve our goal of obtaining the desired inverse system without $x_1^2 $ or $x_1x_2$. Indeed, we use $\lfloor (r-b)/2\rfloor$
squarefree monomials, $\lfloor(3b-r)/2\rfloor$ cubes, and the
monomial $x_{r-1} x_r^2$ if $r-b$ is odd. Since $r \geq
4$, we can avoid taking a monomial of the form $x_1 x_2 x_i$ with $i
\geq 3$ as a generator of the inverse system. Moreover, using the fact that
$\lfloor(3b-r)/2\rfloor \leq r-2$, we see that we also can avoid
$x_1^3$ as a minimal generator. Thus, again we may assume that $x_1
x_2$ and $x_1^2$ are not among the derivatives in the inverse
system.

In the second step, we will start with an inverse system,
$I^\perp$, which has the maximal possible value of $h_2= \dim_k
I^\perp_2$ and gradually change this inverse system step by step
until we get to the inverse system $J^\perp$ that has the minimal
possible value of $h_2$. In each step, we shall see that we will not
decrease $h_2$ by more than one. This means that we will pass
through all the intermediate values of $h_2$.

 Let $J^\perp$
denote the inverse system which achieves the minimal value for $h_2
= \dim_k J^\perp_2$ given $\dim_k J^\perp_3 =b$. We have seen in the
first step that in both of the cases, when $b\ge r$ or when $b<r$,
we may assume that neither $x_1^2$ nor  $x_1x_2$ is a derivative of any
of the monomials in $J^\perp_3$.

We now order the monomials in $I^\perp_3$  and $J^\perp_3$ in the
lexicographic order with $x_1>x_2>\cdots>x_r$. Starting with the
largest monomial in $I^\perp_3$ not in $J^\perp_3$, the basic idea
is to exchange it first by $x_1^2x_2$, then by $x_1^3$, and at last
by the smallest monomial in $J^\perp_3$ which has not been added
yet. In each step we have to make sure that $h_2$ does not decrease
by more than one and that we do not already have the monomials that
we need to add.

Let $m$  in $I^\perp_3$ be the largest monomial that is not in
$J^\perp_3$, and let $\tilde{I}$ be the inverse system generated by
$I_3^\perp \setminus \{m\}$. When processing the monomial $m$, we
have to distinguish three cases depending on whether $x_1^2$ or
$x_1x_2$ or neither is among the derivatives of  $\tilde{I}$.

If neither $x_1^2$ nor $x_1x_2$ is in $\tilde{I}_2$ we know that
$x_1^2x_2$ has two derivatives that are not in the span of
$\tilde{I}_3$, which means that $h_2$ cannot drop by more than one
when exchanging $m$ by $x_1^2x_2$. We then exchange $x_1^2x_2$ by
$x_1^3$ and $h_2$ decreases by one. At the third step, we exchange
$x_1^3$ by the smallest monomial in $J^\perp_3$ which is not in
$I^\perp_3$ and $h_2$ drops by at most one, but potentially
increases.

When $x_1^2$ is in $\tilde{I}_2$, we have to be in a case when
$m=x_1^2x_i$ for some $x_i$, because our assumption on $J^\perp$
provides that, for all $j$, $x_1^2 x_j$ is not in $J^\perp$. Thus,
in this case, we can start by replacing $m$ by $x_1^3$, and $h_2$
drops by at most one, and then continue as above.

When $x_1x_2$ is in $\tilde{I}_2$, but $x_1^2$ is not, we have to
have $m=x_1^2x_i$ or $m=x_1x_2x_i$, for some $x_i$. Indeed, this
follows by the choice of $m$ and the fact that, for each $j$, $x_1^2
x_j$ and $x_1 x_2 x_j$ are not in $J^\perp$.

If $m=x_1^2x_i$, we replace $m$ by $x_1^3$ as above and if
$m=x_1x_2x_i$ with $i>1$, we introduce $x_1^2$ as a derivative when
replacing $m$ by $x_1^2x_2$, so $h_2$ drops by at most one. Once we
have done this we are back in the situation above.

This proves that for any $h_2$ with $a_{\min } \leq h_2 \leq a_{\max }$
we can produce a pure $O$-sequence of the form $(1, h_1, h_2, b)$
with some $h_1 \leq r$. In the third step we show that we can
always achieve $h_1 = r$ while fixing the other entries in the
sequence.

 We know that $h_1=r$ for the inverse system $I^\perp$ that
we start with and for the inverse system $J^\perp_3$ that we end up
with. We now  use the interval property for $h_1$ in order to
conclude that we can also get the pure $O$-sequences with $h_1=r$
for all the intermediate values of $h_2$ obtained above. The only
thing that can happen is that $h_1<r$ for some of the intermediate
examples. Hence it is sufficient to prove that $r\le r_{\max }$, given
$h_2$ and $b$.

If $h_2\le b$, then $r_{\max }$ is given by differentiability because
of Theorem \ref{3}. We know that $(1,r,a_{\min },b)$ is
differentiable, which implies that $(1,r,h_2,b)$ is differentiable
whenever $a_{\min }\le h_2\le b$. Hence $r\le r_{\max }$ when $h_2\le
b$. If $h_2>b$, we have that $r_{\max }=h_2$,  hence we get $r\le
a_{\min } \le h_2 = r_{\max }$. We conclude that in both cases we can
use the interval property for $h_1$ to prove that $(1,r,h_2,b)$ is a
pure $O$-sequence.
\end{proof}

There are a few other cases where  we can already prove the ICP. The
argument for the next result, where we consider pure $O$-sequences
of type 1, will use purely algebraic methods, since those sequences
coincide with the $h$-vectors of monomial Artinian  (Gorenstein)
complete intersections, which are very well known.

\begin{prop}
The ICP trivially holds for pure $O$-sequences of type 1.  More precisely,
let $e$ be even, and let $\underline{h} = (1, h_1, h_2, \dots, h_{\frac{e}{2}-1}, h_{\frac{e}{2}}, h_{\frac{e}{2}+1} \dots, h_{e-1} , h_e)$ be a pure $O$-sequence of type 1.  Since this is a complete intersection Hilbert function, we have $h_{\frac{e}{2}+1} = h_{\frac{e}{2}-1}, \dots, h_{e-1} = h_1, h_e = 1$.  Assume that we fix all the values except in degree $\frac{e}{2}$.  If $h_1 = 2$ then there are at most  two possibilities for $h_{\frac{e}{2}}$, which are consecutive.  If $h_1 \geq 3$ then there is a unique value for $h_\frac{e}{2}$.
\end{prop}

\begin{proof}
By inverse systems, the set of $O$-sequences of type 1 coincides with the set of Hilbert functions of monomial Artinian  complete intersections, which are the Gorenstein algebras of the form $A = R/I = k[x_1,\dots,x_r]/(x_1^{a_1},\dots, x_r^{a_r})$.  By the symmetry of Gorenstein $h$-vectors, if two such Hilbert functions -- say $\underline{h}'$ and $\underline{h}''$, given by algebras $A'$ and $A''$ -- differ in only one spot, this  has to be the middle degree.  Hence the assumptions in the statement of the theorem can be made  without loss of generality to prove the ICP.

If $h_1 = 2$, since $e$ is even we have two possibilities for $\underline{h}$:
\[
(1,2,3,\dots, a-1, a, a-1, \dots, 3,2,1) \hbox{\hspace{1cm} (where $a = \frac{e}{2} +1$)},
\]
or
\[
(1,2,3,\dots, b-1,b,b,\dots,b,b-1,\dots,3,2,1)
\]
where there is an odd number, $\ell \geq 3$, of $b$'s.  It is clear that in the second  case, if $\ell \geq 5$ then we cannot change the middle $b$, while if $\ell=3$ then we can only increase it to $b+1$. In the first case,  we can only change the middle $a$ to an $a-1$.

Now suppose that
\[
\underline{h} = (1, h_1, h_2, \dots, h_{\frac{e}{2}-1}, c, h_{\frac{e}{2}+1} \dots, h_{e-1} , h_e)
\]
where $h_1 \geq 3$ and  $c = h_{\frac{e}{2}}^{max}$.  We note that $\sum_{i=1}^r a_i = e+r$.  We first ask how many generators $I$ can have of degree $\frac{e}{2}+1$, if all other generators have degree $\leq \frac{e}{2}$.  Suppose there are $s$ such generators.  Then
\begin{equation} \label{1st eq}
e+r = \left (\frac{e}{2} +1 \right )\cdot s + \sum_{a_i \leq \frac{e}{2}} a_i.
\end{equation}
We have $r-s$ generators of degree $\leq \frac{e}{2}$.  Since $1\leq a_i$, we obtain
\[
e+r \geq \left ( \frac{e}{2} +1 \right ) \cdot s + (r-s), \ \ \ \hbox{ i.e. } \ \ \  e \geq \frac{e}{2} \cdot s.
\]
We conclude that $s \leq 2$.  But note that if $s=2$ then from  (\ref{1st eq}) we obtain $\sum_{a_i \leq \frac{e}{2}} a_i = r-2$, and since there are $r-2$ terms in that sum, they must all be 1.  So $h_1 = 2$, which we have already considered.

If we allow generators of degree $> \frac{e}{2}+1$, we no longer can have $\sum a_i \geq r-s$.  We deduce that if there are two generators of degree $> \frac{e}{2}$ then they  must both be of degree $\frac{e}{2}+1$, and furthermore $h_1 = 2$.  So we may assume without loss of generality that there is at most one generator of degree $> \frac{e}{2}$.

Suppose now that $J$ is a complete intersection whose generators of degree $< \frac{e}{2}$ are exactly the same as those of $I$, but which has one or more additional generators of degree $\frac{e}{2}$ (so $h_{\frac{e}{2}}$ drops from the maximum value).  Note first that because $I$ has at most one generator of higher degree, we already have either $r-1$ or all $r$ generators of degree $\leq \frac{e}{2}$.  This already shows that there are at most two values for $h_{\frac{e}{2}}$, which are consecutive, so the ICP holds.

But in fact we claim that there is only one value.   Indeed, dropping $h_{\frac{e}{2}}$ by one means that we fix $r-1$ of the degrees of the generators, and change the last one.  But this changes the socle degree, which contradicts our assumption.
\end{proof}

\begin{prop} \label{ICP codim 2}
The Interval Conjecture (cf. \cite{Za2}) holds in the following cases:
\begin{itemize}
\item for arbitrary level $h$-vectors of codimension two;
\item  for the subclass (a priori) of pure $O$-sequences of codimension two.
\end{itemize}
In fact, these classes coincide.
\end{prop}

\begin{proof}
For Artinian level algebras in codimension two (not necessarily monomial), it was shown in \cite{GL} Corollary 2.5 that a necessary and sufficient condition on the $h$-vector $(1,2,h_2,\dots,h_e)$ is
\begin{equation} \label{GL result}
h_{i+1} \leq 2h_i - h_{i-1} \hbox{ for all } i = 0,\dots,e.
\end{equation}
(The authors  of that paper credit Iarrobino and Valla, independently, for this result but no citation is given.  It can be derived from \cite{iarrobino} Theorem 4.6.)  It is clear that the interval conjecture holds for such algebras.   This proves the first part of the assertion.

Passing to Artinian {\em monomial} algebras, recall that a Cohen-Macaulay codimension two standard graded algebra has a minimal free resolution of the form
\[
0 \rightarrow \mathbb F_2 \stackrel{A}{\longrightarrow} \mathbb F_1 \rightarrow R \rightarrow R/I \rightarrow 0,
\]
where the maximal minors of $A$ (the {\em Hilbert-Burch matrix}) form a minimal generating set for $I$, and the degrees of the entries of $A$ determine the Hilbert function of $R/I$.  Now, given any Hilbert function satisfying (\ref{GL result}), choose an ideal $I$ such that $R/I$ is level with the given Hilbert function and $A$ is the Hilbert-Burch matrix of $R/I$, and let $B$ be the corresponding matrix of integers giving the degrees of the entries of $A$.  Then  a monomial Artinian  algebra can be constructed with this degree matrix by the standard trick of putting the corresponding power of $x$ along the main diagonal and the corresponding power of $y$ along the next diagonal.  The resulting matrix is then the Hilbert-Burch matrix of a standard graded monomial Artinian algebra, and by construction it has the given Hilbert function.  Hence all Hilbert functions satisfying (\ref{GL result}) are pure $O$-sequences, and so the ICP holds for pure $O$-sequences.  This also shows that the classes coincide.
\end{proof}

\begin{prop}\label{gor}
The ICP holds also in the following cases:\\
i) For the last entry of pure $O$-sequences of the form
$$\underline{h}=\left(1,h_1=r,h_2=\binom{r+1}{2},...,\\h_{e-1}=
\binom{r+e-2}{e-1},h_e\right);$$
that is, when all entries of
$\underline{h}$ up to degree $e-1$ are maximal for a polynomial  ring
in $r$ variables;\\
ii) For finite differentiable $O$-sequences.
\end{prop}

\begin{proof}
i) If $\underline{h}=(1,h_1=r,h_2=\binom{r+1}{2},...,h_{e-1}=\binom{r+e-2}{e-1},h_e)$ is a pure $O$-sequence, then there exist $h_e$ monomials of degree $e$ giving a total of $h_i$ different derivatives in degree $i$, for all $i<e$. Let us now add one new monomial of degree $e$ to the previous $h_e$ (if $h_e<\binom{r+e-1}{e}$). Because the values of $h_i$ is  maximal (given $r$) for $i<e$, this new set of monomials, similarly, gives an $O$-sequence where all $h_i$'s are unchanged if $i<e$, and where $h_e$ is replaced by $h_e+1$. This clearly proves the assertion.

ii) This fact is an immediate consequence of Corollary \ref{diff}, that all finite differentiable $O$-sequences are pure, and of the fact that the formula $((n)_{(d)})_1^1$, determining the bound of Macaulay's theorem, is an increasing function of $n$ for all given $d$.
\end{proof}

\begin{corollary}\label{2}
The set $S$ of pure $O$-sequences of length 3 (i.e. of
socle degree 2) satisfies the ICP. More precisely, $(1,r,h_2)$ is a pure $O$-sequence if and only if $\lceil r/2\rceil \leq h_2\leq \binom{r+1}{2}.$
\end{corollary}

\begin{proof} The fact that there exists a positive integer $a_{\min }$ such that $(1,r,h_2)$ is a pure $O$-sequence if and only if $a_{\min } \leq h_2\leq \binom{r+1}{2}$, is an obvious consequence of part i) of the previous proposition (with $e=2$). Also, by inverse systems, the minimum such $a_{\min } $ must satisfy the inequality $a_{\min } \geq \lceil r/2\rceil $, since any monomial of degree 2 has at most two non-zero derivatives.

Therefore, it remains to prove the opposite  inequality, $a_{\min }
\leq \lceil r/2\rceil $. We will do it by constructing a set of
$\lceil r/2\rceil$ monomials of degree 2 in exactly $r$ variables.
If $r=2m$ is even, these monomials can chosen to be, for instance,
$x_1x_2$, ..., $x_{2m-1}x_{2m}$; if $r=2m+1$ is odd, such monomials
can be $x_1x_2$, ..., $x_{2m-1}x_{2m}$, $x_{2m+1}^2$.

We now  show the ICP  on the degree 1 entry, i.e. that, if
$(1,r',h_2)$ and $(1,r'',h_2)$ are two pure $O$-sequences such that
$r'<r''$, then $(1,r,h_2)$ is also a pure $O$-sequence for all
integers $r$ such that $r'\leq r\leq r''$. But if $h_2$ satisfies
both bounds of the statement for $r'$ and $r''$,  it is clear that
is also satisfies both bounds for $r$, and the proof is complete.
\end{proof}

\begin{remark}\label{22} {\em
(i) The condition $\lceil r/2\rceil \leq h_2\leq \binom{r+1}{2}$ of Corollary \ref{2} was already known to Hibi (\cite{Hi}, Example 1.2), even if he did not provide a proof for it in his paper.

\medskip

\noindent (ii) Instead, when $e>2$,  it seems very hard in general to determine explicitly the minimum possible value of $h_e$ for the pure $O$-sequences of Proposition \ref{gor}. This further suggests that studying the ICP instead of seeking a complete characterization for pure $O$-sequences might be the right approach to this problem. See Remark \ref{mats rmk} below for a  result providing the minimum $h_3$ corresponding to a nice, infinite special set of values for $h_2$ (not  maximal for our polynomial ring).

\medskip

\noindent (iii) We have verified the ICP experimentally on \cocoa\   in the following cases:

\begin{itemize}
\item codimension $\leq 3$, socle type 3, socle degree $\leq 10$;

\item codimension $\leq 3$, socle type 2, socle degree $\leq 14$;

\item codimension $\leq 4$, socle type 2, socle degree $\leq 12$.

\end{itemize}

}\end{remark}

\begin{remark} \label{mats rmk}
  The ICP deals with the interior property of the set of pure
  $O$-sequences and as we have mentioned before, a solution of the full
  characterization problem would imply solutions to several well-known open problem
  in design theory. Even in the case of pure $O$-sequences of length
  four there are such relations. For example, we have the inequality
  $h_2\le 3h_3$ and if we have equality $h_2= 3h_3$, the inverse
  system has to be squarefree, forcing the inequality $h_2\le
  \binom{h_1}{2}$. Thus we can say that pure $O$-sequences of the form
  $$h=\left(1,r,\binom{r}{2},\frac{1}{3}\cdot \binom{r}{2}\right)$$
  are \emph{extremal} in the set of pure $O$-sequences of length four.
  In this case, the corresponding object in design theory is known
  as a \emph{Steiner triple system}, which means a set of $3$-sets where every
  $2$-set occurs exactly once as a subset. Steiner triple systems are
  well studied and it was known already to Kirkman~\cite{kirkman}
  in 1847 that the only restriction to the existence of such a system
  was the  numerical condition $r\equiv 1$ or $3\pmod 6$. In fact,
  the condition that $h_3$ is an integer is equivalent to $r\equiv 0$
  or $1 \pmod 3$, but we also have that each of the variables have to
  occur together with all the other variables exactly once in the
  inverse system, which means that the other variables have to be
  matched in pairs and $r$ has to be odd.
\end{remark}

As a consequence of the ICP for socle degree 2 pure $O$-sequences, we will now provide a very short proof of Stanley's conjecture that the $h$-vectors of matroid simplicial complexes are pure $O$-sequence, when such complexes have  dimension 1. This is a result first shown by Stokes in his thesis \cite{Sto,Sto2}; in all fairness, Stokes's proof, which is quite long, is however explicit and possibly  generalizable to classes of higher dimensional matroids, whereas our argument is very short but appears to be {\em ad hoc}. Before proving the result, we  recall the Brown-Colbourn inequalities:

\begin{lemma}(Brown-Colbourn \cite{BC})\label{bc}
Let $\underline{h}=(h_0=1,h_1,...,h_{d+1})$ be the $h$-vector of a matroid complex of dimension $d$ (where each $h_i\geq 0$). Then, for any index $j\leq d+1$ and for any real number $\alpha \geq 1$, we have
$$(-1)^j\sum_{i=0}^j(-\alpha )^ih_i\geq 0,$$
where the inequality is strict if $\alpha \neq 1.$
\end{lemma}

\begin{theorem}\label{d1}
The $h$-vector of a dimension 1 matroid complex is a pure $O$-sequence.
\end{theorem}

\begin{proof}
Let $(1,r,h_2)$ be the $h$-vector of a matroid of dimension $d=1$. If $h_2=0$, the result is obvious. Thus, let $h_2>0$. In view of the inequalities of Corollary \ref{2}, we only need to show that $h_2\geq \lceil r/2\rceil $.

But, for  $j=2$ and $\alpha =1$, the  Brown-Colbourn inequalities say that we have $1- r+h_2\geq 0$, that is, $h_2\geq r-1$. Now, it is immediate to see that, for any $r\geq 2$, $r-1\geq  \lceil r/2\rceil $, and the theorem follows.
\end{proof}

\begin{remark}\label{d2} {\em
Unfortunately, the ICP and the  Brown-Colbourn inequalities do not suffice to prove Stanley's conjecture for matroid complexes of dimension $d>1$. For instance, consider the dimension 2 matroids whose $h$-vectors have the form $\underline{h}=(1,100,100,h_3)$. It can be easily checked that  the best bound on $h_3$ that the Brown-Colbourn inequalities guarantee is $h_3\geq 26$, whereas any pure $O$-sequence of the form $(1,100,100,h_3)$ clearly needs to have $h_3\geq 34$.

This fact, however, was definitely to be expected, since the Brown-Colbourn inequalities are normally far from being sharp as the dimension of the complex grows; notice that, in fact,  the only positive values of $h_3$ for which $\underline{h}=(1,100,100,h_3)$ is the $h$-vector of a matroid are actually 99 and 100. (Thanks to Erik Stokes for pointing out this latter fact to the last author.)
}\end{remark}

%%%%%%%%%%%%%%%%%%%%%%%%%%%%%%%%%%%%%%%%%%%%%%%%%%

\chapter{Enumerating pure $O$-sequences} \label{enum}

A good question to ask in mathematics is always: How many are the objects that we are talking about? Answering this question is one of the goals of enumerative combinatorics. Answering it in a purely combinatorial fashion --- viz., by means of a {\em bijective proof} (\cite{St01}) --- is an important goal of enumerative combinatorics. %(Now, the next question, reluctantly \lq \lq [n]ameless here for evermore'' \cite{edgar}, is directed to commutative algebra: How much do you still need to learn from combinatorics?!?)

Thus we now discuss the enumeration of the objects studied in this monograph, and ask a few more questions: How many pure $O$-sequences are there of given codimension and socle degree? Also, if we take an arbitrary $O$-sequence $H$, what are the odds that $H$ is pure? What about if we start with $H$ level? It turns out that all these questions have very satisfactory asymptotic answers (essentially due to Linusson \cite{Li}): in fact, we are going to see that {\em almost} all $O$-sequences are both level and pure, in a sense that will be made precise below. In the second part  we will specialize the previous questions to the sets of pure $O$-sequences whose type $t$ is also fixed, and see how  the scenario dramatically changes.

Recall that if $f$ and $g$ are two  arithmetic functions, we say that $f$ is {\em asymptotic} to $g$, and write $f(r)\sim_r g(r)$, if $\lim_{r}f(r)/g(r)=1$. All limits are taken for $r$ going to infinity.

\begin{theorem}\label{almost} Fix a positive integer $e$. Let $O^r(e)$ be the number of  $O$-sequences of codimension $r$ and socle degree $e$, and let $P^r(e)$ be the number of {\em pure} $O$-sequences of codimension $r$ and socle degree $e$. Then:

\begin{itemize}
\item[(a)] {\em (essentially Linusson \cite{Li})} For $r$ large, almost all $O$-sequence are pure: precisely, we have
$$P^r(e)\sim_r O^r(e)\sim_r c_er^{\binom{e+1}{2}-1},$$
\label{c_e}
where $$c_e=\frac{\prod_{i=0}^{e-2}\binom{\binom{e+1}{2}-\binom{i+1}{2}-1}{i}}{(\binom{e+1}{2}-1)!}.$$

\item[(b)] Let $D^r(e)$ be the number of differentiable $O$-sequences of codimension $r$ and socle degree $e$.  Then  $D^r(e)$ is also asymptotic to the above arithmetic functions, i.e. $D^r(e) \sim_r c_e r^{\binom{e+1}{2}-1}$.
\end{itemize}
\end{theorem}

\begin{proof} From \cite{Li}, Theorem 2.7 and Proposition 2.8, we have that the number, $M^r(e)$, of $O$-sequences of socle degree {\em at most} $e$ and   codimension {\em at most} $r$ is, with our notation, $\sim_r \frac{c_e}{\binom{e+1}{2}}r^{\binom{e+1}{2}}.$ Hence, the number, say $S^r(e)$, of $O$-sequences of socle degree {\em exactly} $e$ and  codimension {\em at most} $r$ is
$$S^r(e)=M^r(e)-M^r(e-1)\sim_r \frac{c_e}{\binom{e+1}{2}}r^{\binom{e+1}{2}}-\frac{c_{e-1}}{\binom{e}{2}}r^{\binom{e}{2}}\sim_r \frac{c_e}{\binom{e+1}{2}}r^{\binom{e+1}{2}}.$$

It follows that $O^r(e)$ equals
$$S^r(e)-S^{r-1}(e)\sim_r \frac{c_e}{\binom{e+1}{2}}r^{\binom{e+1}{2}} - \frac{c_e}{\binom{e+1}{2}}(r-1)^{\binom{e+1}{2}},$$
and it is easy to see that this function is $\sim_r c_er^{\binom{e+1}{2}-1},$ as desired.

It remains to show that $P^r(e)\sim_r O^r(e)$. An argument can already be found in Linusson's paper (\cite{Li}, Corollaries 3.4 and 3.5), where he actually proves a stronger result. He shows (still with his definitions considering inequalities on codimension and socle degree, but this discrepancy can be dealt with as above) that the number of $O$-sequences is asymptotic to the corresponding number of $f$-vectors of simplicial complexes, and that in turn this is asymptotically equal  to the number of $f$-vectors of shellable simplicial complexes. Since shellable simplicial complexes are (by definition) pure, and pure $f$-vectors are pure $O$-sequences, the result easily follows.

We now give a second and more direct argument to show that $P^r(e)\sim_r O^r(e)$, which relies on a result of this monograph and on an injection that will also be the key to prove Theorem \ref{gorsi} below. This will also prove (b).  Let $O(r,e)$ be the set of  $O$-sequences of codimension $r$ and socle degree $e$, and $D(r,e)$ the subset of $O(r,e)$ of those $O$ sequences that are differentiable. Let $\int $ be the converse operation to differentiation for $O$-sequences (for instance, if $h=(1,3,5,1)$, then $\int h= (1,4,9,10)$). It is clear that the map $\int : O(r-1,e) \rightarrow D(r,e)$ is an injection ($\int $ is in general not surjective for its image is given by the {\em strictly increasing} sequences of $D(r,e)$. Alternatively, $\int $  would define a bijection between the set of codimension $r-1$ $O$-sequences of socle degree {\em at most} $e$ and $D(r,e)$).

Let $P(r,e)$ be the set of pure $O$-sequences of codimension $r$ and socle degree $e$. In Chapter 3, we have seen that differentiable $O$-sequences are pure.
Thus, we have the following chain of inequalities on cardinalities of sets: $$\#O(r-1,e)\leq \#D(r,e)\leq \#P(r,e)\leq \#O(r,e).$$
But the first and last values are, respectively, $O^{r-1}(e)$ and $O^r(e)$, which are both asymptotically equal, for $r$ large, to $c_e r^{\binom{e+1}{2}-1}$. Since the cardinality of $P(r,e)$ is $P^r(e)$, the result  follows.
\end{proof}

Since pure $O$-sequences are level, we immediately have:

\begin{corollary}\label{level} Fix a positive integer $e$. Let $L^r(e)$ be the number of  {\em level} Hilbert functions of codimension $r$ and socle degree $e$. Then, for $r$ large, almost all level sequences are pure, and
$$L^r(e) \sim_r c_e r^{\binom{e+1}{2}-1}.$$
\end{corollary}

Let us now turn our attention to the enumeration of pure $O$-sequences of codimension $r$, socle degree $e$ {\em and given type $t$}. Let $P^r_t(e)$ be such counting function.  The behavior of $P^r_t(e)$ and that of the function, $L^r_t(e)$, counting the corresponding number of arbitrary level sequences, become entirely different, and in particular there is no hope to specialize  the beautiful asymptotic results given above. We have:

\begin{prop} \label{000} Let $P^r_t(e)$ be the number of pure $O$-sequences of codimension $r$, socle degree $e$, and type $t$. Then $P^r_t(e)=0$ for $r>te$,  this bound being sharp.
\end{prop}

\begin{proof}
Since each monomial of degree $e$ is in at most $e$ variables, we have at once that $P^r_t(e)=0$ for all integers $r>te$. To prove that $P^{te}_t(e)>0$, it suffices to consider the pure $O$-sequence of codimension $te$ determined by the $t$ squarefree monomials $y_1\cdots y_e$, $y_{e+1}\cdots y_{2e}$, ..., $y_{(t-1)e+1}\cdots y_{te}$. (It is easy to see that this is the only possible construction, up to permutation of the variables, hence in fact $P^{te}_t(e)=1$.)
\end{proof}

A fascinating problem is to determine a closed formula for $P^r_t(e)$. It is not even clear if it is at all possible in general. There is a  combinatorial solution for the simplest case, type $t=1$.

\begin{prop}\label{p_r} $P^r_1(e)=p_r(e)$, the number of partitions of the integer $e$ with exactly $r$ parts.
\end{prop}

\begin{proof} Each pure $O$-sequence of codimension $r$ and socle degree $e$ is determined by some monomials of the form $y_1^{a_1}\cdots y_r^{a_r}$, where $a_1+ \dots + a_r=e$ and $a_i\geq 1$ for all $i$. If, moreover, we order the $a_i$'s, for instance by requiring that $a_1 \leq \dots \leq a_r$, then we have a bijection between the set of such monomials $y_1^{a_1}\cdots y_r^{a_r}$, and our set of type 1 pure $O$-sequences of codimension $r$ and socle degree $e$: indeed, the pure $O$-sequences given by $y_1^{a_1}\cdots y_r^{a_r}$ and $y_1^{b_1}\cdots y_r^{b_r}$ begin differing in degree $\min \{ a_d +1, b_d +1 \}$,  if $d$ is the smallest index such that $a_d \neq b_d$.

But the number of $r$-tuples $(a_1, \dots , a_r)$ of positive integers such that $a_1+ ... + a_r=e$ and $a_1 \leq ... \leq a_r$ is clearly given by $p_r(e)$ (one can just think of a Ferrers, or Young, diagram), as desired.
\end{proof}

As an immediate consequence, we have proved bijectively the following new characterization of integer partitions:

\begin{corollary}
The number, $p(e)$, of  partitions of the integer $e$ coincides with the number of pure $O$-sequences of type 1 and socle degree $e$.
\end{corollary}

As we said, the behavior of the function $L^r_t(e)$ is different from that of $P^r_t(e)$, and in particular $L^r_t(e)$ is very far from going to 0 for $r$ large. We will stick to the case $t=1$ below, and determine the asymptotic value of $L^r_1(e)$, by proving the fact that the asymptotic number of Gorenstein Hilbert functions of given socle degree is, in a sense, both the largest possible and the originally expected one: even though it is now well known that the result initially hoped for by Stanley and Iarrobino (see \cite{St3}), namely that all Gorenstein Hilbert functions need to be SI-sequences, is not true (see \cite{BI,Bo,BL,MNZ1,St2,Za0}), we will now show that, when it comes to enumerating Gorenstein Hilbert functions, that result is {\em almost true}! That is, the number of Gorenstein Hilbert functions which are not SI-sequences is negligible:

\begin{theorem}\label{gorsi} Fix a positive integer $e$. Let $G^r(e)=L^r_1(e)$ be the number of Gorenstein Hilbert functions of codimension $r$ and socle degree $e$, and let $SI^r(e)$ be the number of SI-sequences of codimension $r$ and socle degree $e$. Then, for $r$ large, almost all  Gorenstein Hilbert functions are SI-sequences. Precisely, we have
$$G^r(e)\sim_r SI^r(e)\sim_r c_{\lfloor e/2 \rfloor}r^{\binom{{\lfloor e/2 \rfloor}+1}{2}-1}.$$
\end{theorem}

\begin{proof} Let $SI(r,e)$ and $G(r,e)$ be, respectively, the set of SI-sequences and of Gorenstein Hilbert functions of codimension $r$ and socle degree $e$. With the same notation as at the end of the proof of Theorem \ref{almost}, let us first show that:
\begin{equation}\label{ddd}
\#O(r-1,\lfloor e/2 \rfloor)\leq \#D(r,\lfloor e/2 \rfloor )= \#SI(r,e)\leq \#G(r,e)\leq \#O(r,\lfloor e/2 \rfloor ).
\end{equation}

The first inequality has already been shown in the proof of Theorem \ref{almost}. The next (equality) follows from the fact that SI-sequences are symmetric; in fact, it is easy to see that function associating a differentiable sequence of $D(r,\lfloor e/2 \rfloor )$ with its obvious symmetrization in $SI(r,e)$ is a bijection. The inequality $\#SI(r,e)\leq \#G(r,e)$ holds because all SI-sequences are Gorenstein (\cite{CI,MN3}). Finally, to prove the last inequality of (\ref{ddd}), note that  Gorenstein Hilbert functions  are symmetric. Hence if we associate a Gorenstein sequence with its first half, we have clearly defined an injection of $G(r,e)$ into $O(r,\lfloor e/2 \rfloor )$. This proves (\ref{ddd}).

Since, by Theorem \ref{almost}, $$\#O(r-1,\lfloor e/2 \rfloor)\sim_r \#O(r,\lfloor e/2 \rfloor ) \sim_r c_{\lfloor e/2 \rfloor}r^{\binom{{\lfloor e/2 \rfloor}+1}{2}-1},$$
 the result  follows.
\end{proof}

Since differentiable sequences are obviously unimodal, a similar argument, based on the injection of Theorem \ref{almost} of the set of codimension $r-1$ $O$-sequences into the set of codimension $r$ differentiable $O$-sequences, immediately gives the following consequence that {\em almost all} pure $O$-sequences, as well as {\em almost all} level Hilbert functions, are unimodal:

\begin{corollary}
For any given positive integer $e$, almost all pure $O$-sequences and almost all level Hilbert functions of codimension $r$ and socle degree $e$ are unimodal, when $r$ goes to infinity.
\end{corollary}

As we have seen above, it is very useful to consider the asymptotic behavior, for $r$ large, of $L^r_t(e)$, the cardinality of the set of level Hilbert functions of given codimension, socle degree and type, whereas the cardinality of the corresponding set of pure $O$-sequences, $P^r_t(e)$, is zero for $r>te$. It was shown in Theorem \ref{gorsi} that the subset of Gorenstein Hilbert functions
%(level of type $t=1$)
preserves the property that almost all of its sequences are unimodal. This motivates the following combinatorial conjecture specifically on level Hilbert functions:

\begin{conj}
For any given positive integers $e$ and $t$, {\em almost all} codimension $r$ level Hilbert functions of socle degree $e$ and type $t$ are unimodal, when $r$ goes to infinity.
\end{conj}

We now address a slightly different asymptotic enumeration question, letting the socle type go to infinity.  Specifically, we will prove the following theorem.

\begin{theorem}\label{thmAsymptoticSocleDegree3}
Let $P(t)$ denote the number of pure $O$-sequences of socle degree 3 and socle type $t$.  Then
\[
\lim_{t \rightarrow \infty} \frac{P(t)}{t^2} = \frac{9}{2}.
\]
\end{theorem}

\begin{proof}
Consider a pure $O$-sequence $(1,r,a,t)$.  We have the following immediate fact:
\[
r \leq a \leq 3t
\]
 (by Hibi's theorem and the fact that $t$ monomials of degree 3 have at most $3t$ derivatives).
For fixed $t$, we thus see that the possible values of $r$ and $a$ fall into one of the following three marked areas, which we will treat separately.

\bigskip
\bigskip
\bigskip

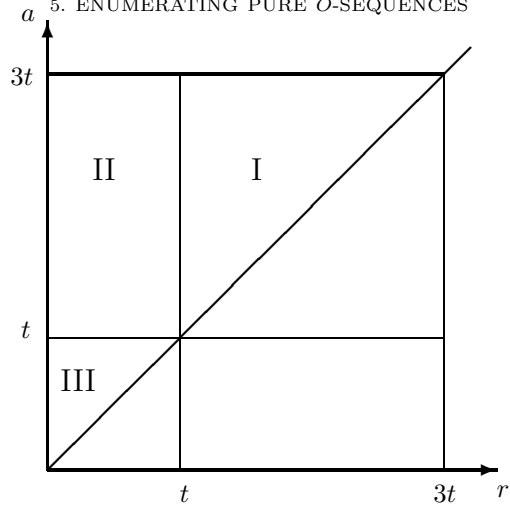
\begin{center}
\begin{picture}(160,160)(10,10)
\thicklines
\put (10,10){\vector(0,1){170}}
\put (10,10){\vector(1,0){170}}
\put (10,10){\line(1,1){160}}
\put (180,0){$r$}
\put (0,180){$a$}
\thinlines
\put (60,10){\line(0,1){150}}
\put (160,10){\line(0,1){150}}
\put (10,60){\line(1,0){150}}
\thicklines
\put (10,160){\line(1,0){150}}
\thinlines
\put (60,-2){$t$}
\put (156,-2){$3t$}
\put (-4,156){$3t$}
\put (0,60){$t$}
\put (87,120){{\Large I}}
\put (27,120){{\Large II}}
\put (15,40){{\Large III}}
\end{picture}
\end{center}
\bigskip  \bigskip

Let $P_1(t)$ denote the number of pure $O$-sequences occurring in Region I, and similarly for Regions II and III.  For the lines separating the regions, we will make clear below in which region they will be considered.

\medskip

\noindent \underline{Region I}: $t \leq r \leq a \leq 3t$.  We will compute, for $t\geq 7$, the exact value of $P_1(t)$, from which it will follow that
\begin{equation}\label{222}
\lim_{t \rightarrow \infty} \frac{P_1(t)}{t^2} = 2.
\end{equation}

The two key steps here will be showing that:
\begin{itemize}
\item[  i)] For all $t\geq 7$, $(1,t,3t,t)$ is a pure $O$-sequence;
\item[  ii)] For all $r,t\geq 1$ such that $t\leq r\leq 3t$, $(1,r,r,t)$ is  a pure $O$-sequence.
\end{itemize}

Then, since $(1,3t,3t,t)$ is pure (it is produced by an inverse system generated by $t$ squarefree degree 3 monomials in $3t$ different variables), it follows from  i) and the ICP in degree 1 that $(1,r,3t,t)$ is also pure for all $t\leq r\leq 3t$. Hence, using ii) and the ICP applied to degree 2, we conclude that  $(1,r,a,t)$ is pure for all values of $r,a,t$ satisfying  $7\leq t\leq r\leq a\leq 3t$.

Therefore, for each $t\geq 7$,
$$P_1(t)=\sum_{r=t}^{3t} (a_{\max }(r,t)-a_{\min }(r,t)+1) =\sum_{r=t}^{3t} (3t-r+1)
%=\sum_{j=1}^{2t+1}j
=2t^2+3t+1,$$
and (\ref{222}) follows. Thus, it remains to prove i) and ii).

{\em Proof of i).}   We want to show that, for any given $t\geq 7$, $(1,t,3t,t)$ is a pure $O$-sequence.  The idea is to take $t$ variables $y_1,\dots ,y_t$, and consider the monomials $y_i y_{i+1} y_{i+3}$, where $i=1,2,\dots ,t$ and the indices are taken modulo $t$.
To see that the $3t$ first partials given by those $t$ monomials are all different, consider $t$ vertex points on a circle, corresponding to the $t$ variables.  The $3t$ partials correspond to the $t$ chords joining adjacent vertices, the $t$ chords joining vertex $i$ to vertex $i+2$ (mod $t$), and the $t$ chords joining vertex $i$ to vertex $i+3$ (mod $t$).  Since $t \geq 7$, clearly no two of these chords coincide.

Notice, as an aside, that the least value of $t$ for which  $h=(1,t,3t,t)$ is a pure $O$-sequence is in fact 7. Indeed, since the $t$ monomials of degree 3 generating an inverse system with $h$-vector $h$ must have $3t$ distinct first partial derivatives, they  need all be squarefree. But the number of degree 2 squarefree monomials in $t$ variables is $\binom{t}{2}$, which is $\geq 3t$ only when $t\geq 7$.

{\em Proof of ii).} Fix two positive integers $r,t$, where $t\leq r \leq 3t$. We want to construct $t$ monomials of degree 3 producing the pure $O$-sequence $(1,r,r,t)$, such that each uses different variables. Those monomials will be either of the form $y_u y_v y_w$ (say, $i$ of them), or $y_u y_v^2$ ($j$ of them), or $y_u^3$ ($k$ of them).

It is now clear that our claim is equivalent to the combinatorial problem of showing the existence of a weak 3-composition of the integer $t$ as $t=i+j+k$, where $3i+2j+k=r$.

A solution is the following: if $r=t+m$, for some $0\leq m<t$, then pick $i=0$, $j=m$, and $k=t-m$; if $r=2t+m$, for some $0\leq m\leq t$, then pick $i=m$, $j=t-m$, and $k=0$.

\bigskip

\noindent \underline{Region II}: $0 < r < t \leq a \leq 3t$.  We will show that
\[
\lim_{t \rightarrow \infty} \frac{P_{2}(t)}{t^2} = 2
\]
in this region.

Notice that there are $(t-1)(2t+1) = 2t^2-t-1$ integer lattice points in this region.  For given $t$ it is not the case that all such lattice points occur as pure $O$-sequences.  However, we will show that as $t$ gets large, the proportion of the rectangle that does not occur goes to zero.  To this end, let us fix $t$ ``sufficiently large" (which will be made more precise below).

\medskip

\noindent \underline{\it Claim 1:}  {\it For $r$ satisfying $\sqrt{2t} \leq r \leq t$, the sequence $(1, r , t, t)$ is a pure $O$-sequence.}

\medskip

Since this sequence is non-decreasing as long as $r \leq t$, we know from Corollary \ref{diff} and Theorem \ref{3} that the sequence is pure if and only if it is differentiable.  This means that we need  $(1, (r -1), (t-r ) , 0)$ to satisfy Macaulay's bound.
We calculate that this holds whenever
\[
r \geq \frac{-1 + \sqrt{1+8t}}{2}.
\]
Clearly this holds, in particular, if $r \geq \sqrt{2t}$.
This finishes Claim 1.

\medskip

\noindent \underline{\it Claim 2:}  {\it Assume that $t  \geq 49$.  Let $r$ satisfy $3\sqrt{t} \leq r \leq t$.  Then $(1,r,3t,t)$ is a pure $O$-sequence.  }

\medskip

\underline{Step 1}.  We first show that for any positive integer $n$, the sequence $(1,3n, 3n^2,n^2)$ is a pure $O$-sequence.   We divide the $3n$ variables into three sets: $\{ x_1,\dots,x_n \},\\ \{x_{n+1},\dots,x_{2n} \}$, $\{ x_{2n+1},\dots,x_{3n}\}$.  For each choice of $i$ and $j$ with $0 \leq i \leq n-1$ and $ 0 \leq j \leq n-1$, we form the degree 3 monomial $x_{1+i} x_{n+1+j} x_{2n+1+ k}$ with $k =   i+j \ (\hbox{mod } n) $.  The group law for $\mathbb Z_n$ guarantees that no monomial of degree 2 divides more than one of these monomials, and we are done.

\medskip

\underline{Step 2}.  Now we show that for any positive integer $n \geq 7$ and for any $t$ satisfying $n^2 \leq t \leq (n+1)^2$, the sequence $(1,3n,3t,t)$ is a pure $O$ sequence.  We again use the partition of the $3n$ variables into three sets as described in Step 1.  The monomials of degree 3 chosen so far consist of one variable from each of the three sets.  Now, it follows from step i) of Region I that  as long as $n \geq 7$, we can form an additional $n$ monomials of degree 3 from the variables in any one of these three sets, having no divisors of degree 2 in common.  Thus we can take $n$ such monomials from each of the first two sets (added one at a time) and one from the third, and combining these with the $n^2$ monomials already chosen, we are done.   This completes Step 2.

\medskip

\underline{Step 3}.  Now with $t \geq 49$ fixed, let $n$ be the largest integer such that $n^2 \leq t$.  Then $3n \leq 3 \sqrt{t} \leq r \leq t$.  We saw in Steps 1 and 2 that $(1,3n,3t,t)$ is a pure $O$-sequence, and we saw in the argument for Region I that $(1,t,3t,t)$ is a pure $O$-sequence.  Then Claim 2 follows from the ICP for pure $O$-sequences of socle degree 3 (Theorem \ref{ICP3}) for $r$.

\medskip

\noindent \underline{\it Claim 3:} {\it  Assume that $t \geq 49$.  Then for all lattice points $(r,a)$ in the rectangular region $3\sqrt{t} \leq r \leq t \leq a \leq 3t$, the sequence $(1,r,a,t)$ is a pure $O$-sequence.  }

\medskip

Claim 3 follows immediately from Claims 1 and 2, using the ICP for pure $O$-sequences of socle degree 3 (Theorem \ref{ICP3}), this time for $a$.

From Claim 3 we have that for $t \geq 49$,
\[
(t - 3 \sqrt{t})(2t) \leq P_2(t) \leq 2t^2 - t - 1.
\]
Hence
\[
2  - \frac{6}{\sqrt{t}} \leq \frac{P_2(t)}{t^2} \leq 2 - \frac{1}{t} - \frac{1}{t^2},
\]
and the result for Region II is complete.

\bigskip

\noindent \underline{Region III}: $0 < r \leq a < t$.  We will show that
\[
\lim_{t \rightarrow \infty} \frac{P_{3}(t)}{t^2} = \frac{1}{2}
\]
in this region.

By Corollary \ref{diff} and Theorem \ref{3}, the points $(a,r)$  in this region that correspond to pure $O$-sequences  are precisely those for which $(1,r,a,t)$ is a differentiable $O$-sequence.  This also forces $r<a$ in each such sequence.

 We will begin by considering a scaled model: we consider points $(\bar r, \bar a)$ such that $0 < \bar r < \bar a < 1$.  Fix any point $(\bar r_0, \bar a_0)$ in this region.  This defines a triangular region $A(\bar r_0, \bar a_0)$ given by the points $(\bar r,\bar a)$ satisfying
\[
\bar r_0  \leq \bar r \leq \bar r_0  - \bar a_0  + \bar a \ \hbox{ and } \ \bar a < 1.
\]

\bigskip  \bigskip

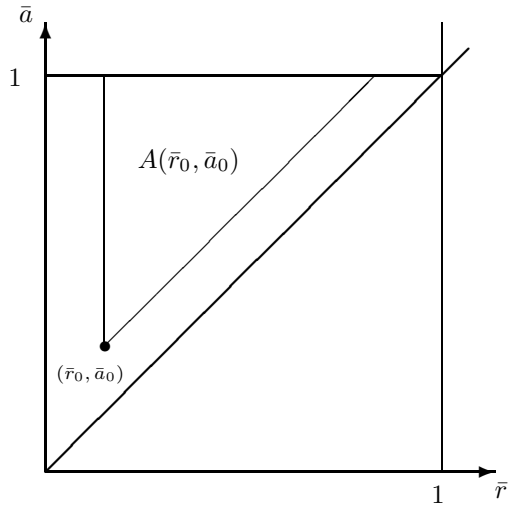
\begin{center}
\begin{picture}(160,160)(10,10)
\thicklines
\put (10,10){\vector(0,1){170}}
\put (10,10){\vector(1,0){170}}
\put (10,10){\line(1,1){160}}
\put (180,0){$\bar r$}
\put (0,180){$\bar a$}
\thinlines
\put (160,10){\line(0,1){170}}
\thicklines
\put (10,160){\line(1,0){150}}
\put (156,-2){$1$}
\put (-4,156){$1$}
\put (30,55){$\bullet$}
\put (14,45){{\scriptsize $(\bar r_0, \bar a_0)$}}
\thinlines
\put (32,57){\line(1,1){103}}
\put (32,57){\line(0,1){103}}
\put (45,125){$A(\bar r_0,\bar a_0)$}
\end{picture}
\end{center}

\bigskip

We will show that for $t \gg 0$, all points in the subregion $A(\bar r_0,\bar a_0)$ ``exist" (in a suitable sense).  Then letting $(\bar r_0,\bar a_0)$ get arbitrarily close to $(0,0)$ will give the result.  Notice that unlike the previous two regions, at this point we have not fixed $t$.
We will want to show that $(1, \lfloor \bar r_0 t \rfloor , \lfloor \bar a_0 t \rfloor, t)$ is a differentiable $O$-sequence for $t \gg 0$; hence it is a pure $O$-sequence by Corollary \ref{diff}.

\medskip

\noindent \underline{\it Claim 1:}  {\it There exists an integer $t_1$ such that for $t \geq t_1$, $(1, \lfloor \bar r_0 t \rfloor - 1,  \lfloor \bar a_0 t \rfloor  - \lfloor \bar r_0 t \rfloor )$ is an $O$-sequence. }

\medskip

This is clear from Macaulay growth: since $\bar r_0$ and $\bar a_0$ are fixed, for $t$ large enough $ \binom{\lfloor \bar r_0 t \rfloor }{2}$ will be larger than $ \lfloor \bar a_0 t \rfloor  - \lfloor \bar r_0 t \rfloor$.

\medskip

\noindent \underline{\it Claim 2:}  {\it There exists an integer $t_2$ such that for $t \geq t_2$, the growth from $\lfloor \bar a_0 t \rfloor - \lfloor \bar r_0t \rfloor$ in degree 2 to $t - \lfloor \bar a_0t \rfloor$ in degree 3 satisfies Macaulay's bound. }

\medskip

From $\bar a_0 t - 1 \leq \lfloor \bar a_0 t \rfloor \leq \bar a_0 t $ and $\bar r_0 t - 1 \leq \lfloor \bar r_0 t \rfloor \leq \bar r_0 t $, we obtain
\[
 \bar a_0 t - \bar r_0 t -1 \leq \lfloor \bar a_0 t \rfloor - \lfloor \bar r_0 t \rfloor \leq \bar a_0 t - \bar r_0 t +1,
\]
and hence
\[
\lfloor (\bar a_0 - \bar r_0) t \rfloor -1 = \lfloor \bar a_0 t - \bar r_0 t \rfloor -1 \leq \lfloor \bar a_0 t \rfloor - \lfloor \bar r_0 t \rfloor \leq \lfloor \bar a_0 t - \bar r_0 t \rfloor +2 = \lfloor (\bar a_0 - \bar r_0) t \rfloor +2.
\]
Similarly,
\[
(1-\bar a_0)t = t - \bar a_0 t \leq t - \lfloor \bar a_0 t \rfloor \leq t - \bar a_0 t + 1 = (1- \bar a_0 )t +1.
\]
Macaulay's bound from $\lfloor \bar a_0 t \rfloor - \lfloor \bar r_0t \rfloor  $ in degree 2 will be on the order of $t^{\frac{3}{2}}$ (up to scalar multiple), so for $t \gg 0$ this will exceed $t- \lfloor \bar a_0t \rfloor$.  This establishes Claim 2.

From Claims 1 and 2 we have that $(1, \lfloor \bar r_0 t \rfloor, \lfloor \bar a_0 t \rfloor, t)$ is a differentiable $O$-sequence for $t \geq \max \{ t_1, t_2 \}$, since a sequence whose first difference is an $O$-sequence is itself an $O$-sequence.  Hence it is a pure $O$-sequence.

\medskip

\noindent \underline{\it Claim 3:}  {\it There exists an integer $t_3$ such that for $t \geq t_3$, the sequence $(1, \lfloor \bar r_0 t \rfloor, t, t)$ is a pure $O$-sequence.  }

\medskip

From Claim 1 of Region II, we know that if $\sqrt{2t} \leq \lfloor \bar r_0 t \rfloor$ then $(1, \lfloor \bar r_0 t \rfloor, t, t)$ is a pure $O$-sequence.  Clearly this is true for $t \gg 0$, establishing Claim 3.

\medskip

\noindent \underline{\it Claim 4:}  {\it For any $(r,a)$ with $a < t$, if $(1,r,a,t)$ is a non-decreasing pure $O$-sequence then so is $(1, r+1, a+1, t)$.  }

\medskip

It is enough to check differentiability.  But if $(1, r-1, a-r, t-a)$ is an $O$-sequence then clearly so is $(1, r, a-r, t-a-1)$.  This proves Claim 4.

Now let $t \geq \max \{ t_1, t_2, t_3 \}$.  Notice that this depends only on the initial choice of $\bar r_0$ and $\bar a_0$.  We now have that $(1, \lfloor \bar r_0 t \rfloor, \lfloor \bar a_0 t \rfloor, t)$ is a pure $O$-sequence, thanks to Claims 1 and 2, and that $(1, \lfloor \bar r_0 t \rfloor, t,t)$ is a pure $O$-sequence thanks to Claim 3.  Then by the ICP for $h_2$, we have $(1,\lfloor \bar r_0 t \rfloor, a,t)$ is a pure $O$-sequence for all $\lfloor \bar a_0 t \rfloor \leq a \leq t$.  This corresponds to the vertical line segment joining the point $(\lfloor \bar r_0 t \rfloor, \lfloor \bar a_0 t \rfloor)$ to $(\lfloor \bar r_0 t \rfloor, t)$.  We also have from Claim 4  that for each point on the diagonal line segment with slope 1 starting at $(\lfloor \bar r_0 t \rfloor, \lfloor \bar a_0 t \rfloor)$ and ending at $( r_1,t)$ (for suitable $r_1$), the corresponding sequence is pure.  It follows using the ICP for $h_1$ (looking at horizontal line segments within the triangle)  that for any $(r,a)$ in the triangular region
\[
\lfloor \bar r_0 t  \rfloor \leq r \leq \lfloor \bar r_0 t \rfloor  - \lfloor \bar a_0 t \rfloor  + a \ \hbox{ and } \ a < t ,
\]
$(1, r, a, t)$ is a pure $O$-sequence.   Of course other points outside this triangle may also be pure.  We now count:
\[
\frac{1}{2} (t  - \bar a_0 t)(t - \bar r_0 t) \leq \frac{1}{2} ( t - \lfloor \bar a_0 t \rfloor)(t - \lfloor \bar r_0 t \rfloor ) \leq P_3(t) \leq \frac{1}{2} t^2.
\]
Hence
\[
\frac{1}{2} (1-\bar a_0)(1-\bar r_0) \leq \frac{P_3(t)}{t^2} \leq \frac{1}{2}, \ \  \hbox{ or equivalently } \ \ \frac{1}{2} - (\bar a_0 + \bar r_0 - \bar a_0 \bar r_0)  \leq \frac{P_3(t)}{t^2} \leq \frac{1}{2}.
\]
Notice that the left-hand side is the area of the triangle $A(\bar r_0, \bar a_0)$.
Now, given $\epsilon >0$ we can choose the point $(\bar r_0, \bar a_0)$ in Region III such that $\bar a_0 + \bar r_0 - \bar a_0 \bar r_0 < \epsilon$.  For this choice of $\bar r_0$ and $\bar a_0$ we can find $t_1, t_2, t_3$ as above and conclude that for $t \geq \max \{ t_1, t_2, t_3 \}$ we have
\[
\left | \frac{P_3 (t)}{t^2} - \frac{1}{2} \right | < \epsilon,
\]
and the result follows.

Combining the results obtained for Regions I, II and III, we conclude the proof.
\end{proof}

%%%%%%%%%%%%%%%%%%%%%%%%%%%%%%%%%%%%%%%%%%%%%%%%%%

\chapter{Monomial Artinian  level algebras of type two in three variables} \label{type 2 three var}

We have already seen in Remark \ref{wlpunim} that the WLP imposes strong conditions on the Hilbert function, hence in our case on the possible pure $O$-sequences.  In this monograph we completely answer the question of which combinations of the type and the number of variables force the WLP or the SLP (see Theorem \ref{answernd}).  A key ingredient in this classification is the main result of this chapter,  Theorem~\ref{prop-WLP-type2}.

We first give a more concrete description of  a level monomial ideal of type 2 in three variables.  This allows us, in turn, to give a more concrete formula for the Hilbert function, which involves only the sum of two complete intersection Hilbert functions (one shifted).  We then turn our attention to the WLP for such algebras.

For the following result, let us denote by $H(a,b,c)$ the Hilbert function of the ideal $( x^a,y^b,z^c )$, and as usual we denote by $H(a,b,c)(-r)$ the shifted function.

\begin{prop} \label{two forms}
Let $R = k[x,y,z]$ and let $I$ be a monomial ideal such that $R/I$
is level of type 2.  Then up to a change of variables, $I$ has one
of the following two forms:

\begin{enumerate}
\item[(a)] $( x^a, y^b, z^c, y^\beta z^\gamma )$,
where $b-\beta = c - \gamma$.  In this case the Hilbert function
is
\[
H(a,b-\beta,\gamma)(-\beta) + H(a,\beta,c).
\]

\item[(b)] $( x^a, y^b, z^c, x^\alpha y^\beta, x^\alpha
z^\gamma )$ where $a-\alpha = (b - \beta) + (c - \gamma)$.
(In particular, without loss of generality we may assume $a >
\alpha$.)  In this case the Hilbert function is
\[
H(a-\alpha, \beta, \gamma)(-\alpha) + H(\alpha,b,c).
\]

\end{enumerate}

\noindent Conversely, any Hilbert function $\underline{H}$ that can be decomposed as
\[
\begin{array}{rcl}
\underline{H} & = & H(a,b-\beta,\gamma) (-\beta) + H(a,\beta,c) \hbox{ with } b-\beta = c-\gamma, \hbox{ or } \\
\underline{H} & = & H(a-\alpha, \beta, \gamma)(-\alpha) + H(\alpha,b,c) \hbox{ with } a-\alpha = (b-\beta) + (c - \gamma)
\end{array}
\]
is a pure $O$-sequence.
\end{prop}

\begin{proof}
A monomial Artinian  level algebra of type 2 over $k[x,y,z]$
arises as the inverse system of two monomials of the same degree,
say $x^{a_1} y^{a_2} z^{a_3}$ and $x^{b_1} y^{b_2} z^{b_3}$, where
$a_1 + a_2 + a_3 = b_1 + b_2 + b_3$.  We distinguish two cases.

\medskip

\noindent \underline{Case 1}: either $a_1 = b_1$ or $a_2 = b_2$ or
$a_3 = b_3$.
%(Note that if two are equal then all three are
%equal, and if all three are equal then the type is 1 and not 2.)
Without loss of generality assume that $a_1 = b_1$.  Without loss
of generality we may also assume that $a_2 < b_2$ and $a_3 > b_3$.
Then  the annihilator of these two forms is the ideal
\[
I = ( x^{a_1 +1}, y^{b_2+1}, z^{a_3+1}, y^{a_2+1} z^{b_3+1}
).
\]
This is of the claimed form.  Now for simplicity of notation, set $a = a_1+1, b = b_2+1, c = a_3+1, \beta = a_2+1, \gamma = b_3+1$.  We see that
\[
I = y^\beta  ( x^a,y^{b-\beta}, z^\gamma) + (x^a,z^c),
\]
which is a basic double link.  The claimed Hilbert function follows immediately from Lemma \ref{lem:BDL}.

\medskip

\noindent \underline{Case 2}: $a_1 \neq b_1, a_2 \neq b_2, a_3
\neq b_3$.  If we compare $a_1$ with $b_1$, $a_2$ with $b_2$ and
$a_3$ with $b_3$, either two $a_i$ will be bigger than the
corresponding $b_i$ or vice versa.
%(We have no equalities by
%hypothesis, and it is impossible for all three $a_i$ to be bigger
%than the corresponding $b_i$, or vice versa, since the degrees of
%the two forms are equal.)
Without loss of generality assume that
$a_2 > b_2$ and $a_3 > b_3$, and so $a_1 < b_1$.  Then the
annihilator of the two forms is the ideal
\[
I = ( x^{b_1 +1}, y^{a_2+1}, z^{a_3+1}, x^{a_1+1} y^{b_2+1},
x^{a_1+1} z^{b_3+1} ).
\]
Setting the values for $a,b,c,\alpha, \beta, \gamma$ accordingly,
we obtain the desired classification.  Note that now
\[
I = x^\alpha (x^{a-\alpha},y^\beta,z^\gamma) + (y^b,z^c).
\]
Again this is a basic double link, so the Hilbert function result follows immediately from Lemma \ref{lem:BDL}.

For the converse, the proof above gives the construction of the monomial ideal, and the numerical conditions guarantee that the resulting algebra will be level.
\end{proof}

 We now show that all  monomial Artinian level algebras of type 2 in three variables have the WLP.  (Based on computer evidence, we conjecture  furthermore that the multiplication by $L^2$ also has maximal rank.)

\begin{theorem}\label{prop-WLP-type2}
A monomial Artinian  level algebra of type 2 in three variables over a field of characteristic zero has the WLP.
\end{theorem}

%%%%%%%%%%%%%%%%%%%%%%%%%%%%%%%%%%%%%%%%%%%%%%%%

The proof  relies on results about certain modules over  polynomial
rings in two variables.

\begin{remark}
  \label{rem-SLP-mod}
  {\em
In \cite{HMNW}, Proposition 4.4, it has been shown that all  graded
quotients of a polynomial ring in two variables $S = k[x, y]$ have
the SLP if $k$ has characteristic zero.
However, this result does not generalize to all graded $S$-modules.
For example, the module $S/(x, y) \oplus S/(x, y)^2$ fails to have
the WLP.
}
\end{remark}

The key to proving Theorem \ref{prop-WLP-type2} is the following
result.

\begin{prop}\label{prop-Key}
Given two triples of non-negative integers, $(a_1,b_1,c_1)$ and $(a_2,b_2,c_2)$, satisfying
$$
a_1+b_1+c_1=a_2+b_2+c_2, \quad b_2\le b_1\le a_1, \quad a_2\le a_1,
\quad\hbox{and}\quad c_1<c_2,
$$
consider the $S$-module
\[
M := ((x+y)^{c_2 - c_1}, x^{a_2+1},y^{b_2+1})/(x^{a_1+1},y^{b_1+1}).
\]
Then multiplication by $(x+y)^{c_1 + 1}$ on $M$ has maximal rank,
provided the field $k$ has characteristic zero.
\end{prop}

In order to have a numerical version of this result,  set
\[
A := S/\fa := S/(x^{a_1+1},y^{b_1+1}) \quad  \text{and} \quad B := S/\fb := S/(x^{a_2+1},y^{b_2+1}).
\]
Then we have the exact sequence
\begin{equation}
  \label{eq-sequence-with-mod}
0 \to M \to A \to B/(x+y)^{c_2 - c_1}B \to 0.
\end{equation}
Since $x+y$ is a strong Lefschetz element of $B$ by Remark \ref{rem-SLP-mod} and \cite{MMN2}, Proposition 2.2, we conclude that the Hilbert function of $M$ is
\begin{eqnarray}
  \label{eq-Hilb-mod}
h_M (j) & = & h_A (j) - h_{B/(x+y)^{c_2 - c_1}B} (j) \nonumber\\
& = & h_A(j) - \max\{0, h_B (j) - h_B( j - c_2 + c_1)\} \nonumber \\
& = & \min\{ h_A (j), h_A(j) - h_B(j) + h_B( j - c_2 + c_1) \}.
\end{eqnarray}
Hence Proposition \ref{prop-Key} is equivalent to the following statement.

\begin{lemma}
  \label{lem-key-numerical}

Adopt the above notation, let $d$ be any integer, and consider the multiplication map
\begin{equation*}
\begin{CD}
\ffi: M_{d - c_1} @>{(x+y)^{c_1 + 1}}>> M_{d+1}
\end{CD}
\end{equation*}
Then
\begin{eqnarray}
  \label{eq-num-claim}
\dim_k [\im \ffi]_{d+1} & = & \min\{ h_A(d+1), h_A(d-c_1), \nonumber\\
& & \hspace*{1cm}  h_A(d+1) - h_B(d+1) + h_B(d+1-c_2+c_1), \nonumber\\
& & \hspace*{1cm}  h_A(d-c_1) - h_B(d-c_1) + h_B(d-c_2) \}.
\end{eqnarray}
\end{lemma}

The proof of this result requires some preparation. First, we need:

\begin{lemma}
  \label{lem-det}
Let $p \geq q$ be non-negative integers. Then the $n \times n$
integer matrix
\[
N = \left ( \binom{p}{q-i+j} \right )_{1 \leq i, j \leq n}
\]
is regular.

Furthermore its determinant is
\[
\det N = \frac{\cH (n) \cH(p + n) \cH(p - q) \cH(q)}{\cH (p) \cH(p - q + n) \cH(q+n)},
\]
where $\cH (m)$ is the hyperfactorial of $m$ defined by $\cH(0) := 1$ and
\[
\cH(m) := \prod_{i = 0}^{m-1} i! \quad \mif m > 0.
\]
\end{lemma}

\begin{proof}
Regularity of $N$ follows from Theorem B.1 in \cite{codes}.  The
formula for the determinant is a consequence of Lemma 3 in
\cite{Kratt} (see also \cite{CookN}, Lemma 4.2).
\end{proof}

Second, we consider a  module that is simpler than the module $M$
appearing in Proposition \ref{prop-Key}.

\begin{lemma}
  \label{lem-SLP-mod}
Given two pairs of non-negative integers, $(a_1,b_1)$ and $(a_2,b_2)$, satisfying
\[
\quad b_2\le b_1\le a_1 \quad \hbox{and} \quad a_2\le a_1,
\]
consider the $S$-module
\[
K := (x^{a_2+1},y^{b_2+1})/(x^{a_1+1},y^{b_1+1}).
\]
Then  $(x+y)$ is a  strong Lefschetz element of $K$,  provided the
field $k$ has characteristic zero.
\end{lemma}

\begin{proof}
We have to show that, for each non-negative integer $c$ and any integer $d$, the multiplication map
\begin{equation*}
\begin{CD}
\psi: K_{d - c} @>{(x+y)^{c + 1}}>> K_{d+1}
\end{CD}
\end{equation*}
has maximal rank, that is
\begin{equation*}
  \dim_k [(x+y)^{c+1} K]_{d+1} = \min \{h_K (d-c), h_K (d+1)\}.
\end{equation*}
To this end set, as above,
\[
A := S/\fa := S/(x^{a_1+1},y^{b_1+1}) \quad  \text{and} \quad B := S/\fb := S/(x^{a_2+1},y^{b_2+1}).
\]
Recall that, for non-negative integers $b \leq a$, the Hilbert function of the complete intersection $S/(x^{a+1}, y^{b+1})$ is given by
\begin{eqnarray}
  \label{eq-hilb-ci}
  \quad \quad
\dim_k [S/(x^{a+1}, y^{b+1})]_j  & = &
\begin{cases}
j+1 & \mif 0 \leq j \leq b \\
b+1 & \mif b \leq j \leq a \\
a+b+1-j & \mif a \leq j \leq a+b \\
0 & \text{otherwise}
\end{cases}  \\
& = & \max\{0,   \min \{j+1, a+1, b+1, a+b+1-j \} \}, \nonumber
\end{eqnarray}
where the last formula is also true if $b > a$.

We will use this formula frequently. Note that this Hilbert function
is unimodal and symmetric about $\frac{a+b}{2}$.

We now show that it is enough to consider the map $\psi$ if it is
expected  to be injective. Indeed, using that both $a_2$ and $b_2$
are at most $a_1$, we see that the Hilbert function of $K$,
\[
h_K (j) = h_A (j) - h_B (j),
\]
is non-decreasing if $j \leq \max \{a_1, a_2 + b_2 + 1\}$. Hence,
$h_K (d+1) < h_K (d-c)$ implies in particular $d+1 > a_2 + b_2 + 1$,
so $h_K (d+1) = h_A(d+1)$. It follows that $\psi$ has maximal rank
if and only if
\begin{equation}
   \label{eq-num-info-K}
\dim_k [(x+y)^{c+1} K]_{d+1} = \min \{h_K (d-c), h_A (d+1)\}.
\end{equation}
% In particular, if $\psi$ is surjective, then its image is $K_{d+1} = A_{d+1}.
Assume now that $h_K(d-c) > h_A(d+1)$, so that we have to show
surjectivity of $\psi$. We have seen that the assumption implies
$d+1 > \max \{a_1, a_2 + b_2 + 1\}$. If $d-c > a_2 + b_2$, then we
get $K_{d-c} = A_{d-c}$ and $K_{d+1} = A_{d+1}$, hence
$\psi$ is surjective because $x+y$ is a strong Lefschetz element of
$A$ by Proposition 2.2 in \cite{MMN2}. If $d-c < \max\{a_1, a_2 +
b_2 + 1\}$, then there is an integer $d'$ such that
\[
\max \{a_1, a_2 + b_2 + 1\} \leq d' \leq d \quad \text{and} \quad h_K(d-c)
= h_K (d').
\]
The map $\psi$ decomposes as
\begin{equation*}
\begin{CD}
 K_{d - c}  @>{(x+y)^{c+d' -d}}>> K_{d'} @>{(x+y)^{d -d'}}>> K_{d+1},
\end{CD}
\end{equation*}
where the first map is expected to be injective, thus an
isomorphism, and  the second map is surjective because $x+y$ is a
strong Lefschetz element of $A$. Hence, surjectivity of $\psi$
follows from the expected injectivity.

This argument shows that it is enough to assume that
\begin{equation}
\label{eq-num-assump}
  h_K (d-c) \leq h_K (d+1)
\end{equation}
and then to show that $\psi$ is injective or, equivalently,
\begin{equation}
\label{eq-num-state-1}
  \dim_k [(x+y)^{c+1} K]_{d+1} =  h_K (d-c).
\end{equation}

Note that the Hilbert function of $A$ is non-decreasing until
degree $a_1$. Hence the multiplication map
\begin{equation*}
\begin{CD}
A_{d - c} @>{(x+y)^{c + 1}}>> A_{d+1}
\end{CD}
\end{equation*}
is injective if $d+1 \leq a_1$. Since $K$ is an ideal of $A$, this
implies  injectivity of $\psi$ in this case.

Similarly, if $d-c > a_2 + b_2$, then $K_{d-c} = A_{d-c}$ and
injectivity of $\psi$ follows again from the SLP of $A$. Thus we may assume
\begin{equation}
  \label{eq-lower-bound-d-for-K}
  a_1 \leq d \leq a_2 + b_2 + c.
\end{equation}

Assume now that $d \leq a_2 + c$.  In this case we have  $K_{d-c}
= K'_{d-c}$, where
\[
K' := ((y^{b_2+1}) + \fa)/\fa \cong (S/(x^{a_1 + 1}, y^{b_1 - b_2}))(-b_2 - 1).
\]
Hence, $x+y$ is a strong Lefschetz element of $K'$ by Remark
\ref{rem-SLP-mod} and \cite{MMN2}, Proposition 2.2.

The map $\psi$ is the composition of the maps
\begin{equation*}
\begin{CD}
 K_{d - c} = K'_{d - c} @>{\psi'}>> K'_{d+1} @>>> K_{d+1}
\end{CD}
\end{equation*}
where $\psi'$ is multiplication by $(x+y)^{c + 1}$ and the second
map is injective. However, $\psi'$ is injective by Assumption
\eqref{eq-num-assump} and the SLP of $K'$. It follows that  $\psi$
is injective, as desired.
%if $\psi'$ is. If $\psi'$ is surjective, we see as above that the image of $\psi'$ is $A_{d+1}$, thus $\psi$ is surjective, too. This proves the maximal rank property of $\psi$ in this case.

If $d \leq b_2 + c$, then we argue similarly. Thus, we may assume
throughout the remainder of this proof that
\begin{equation}
  \label{eq-lower-bound-d}
  d - c > \max \{a_2, b_2\}.
\end{equation}
We may also assume that
\begin{equation}
  \label{eq-upper-bound-d-for-K}
  d < a_1 + b_1
\end{equation}
because otherwise $K_{d+1} \subset A_{d+1} = 0$ and the desired
injectivity of $\psi$ follows from Assumption \eqref{eq-num-assump}.

We now rewrite our Claim \eqref{eq-num-state-1}. Observe that
\[
[(x+y)^{c+1} K]_{d+1} \cong [((x+y)^{c+1} \fb + \fa)/\fa]_{d+1}
\]
it is equivalent to
\[
\dim_k [(x+y)^{c+1} \fb + \fa]_{d+1} = \dim_k \fa_{d+1} + h_K (d-c).
\]
Let $V \subset \fb_{d-c}$ be a subspace such that  $\fb_{d-c} =
V \oplus \fa_{d-c}$. Then the last equality means that we have to
show
\[
\dim_k \left ( (x+y)^{c+1} V +  \fa_{d+1} \right ) = \dim_k V + \dim_k \fa_{d+1},
\]
that is
\begin{equation}
  \label{eq-num-state-2}
(x+y)^{c+1} V +  \fa_{d+1}  =  (x+y)^{c+1} V \oplus \fa_{d+1}.
\end{equation}

To prove this we start by choosing bases of $V$ and $\fa_{d+1}$,
which we then manipulate in several steps. Set
\begin{eqnarray*}
W_1 & = &  \langle x^{d-a_1 - i} y^i \s 0 \leq i \leq d-a_1 \rangle \\
W_2 & = &  \langle x^i y^{d-b_1 - i} \s 0 \leq i \leq d-b_1 \rangle.
\end{eqnarray*}
Note that $0 \leq d - a_1 < b_1$ by Assumptions \eqref{eq-lower-bound-d-for-K} and \eqref{eq-upper-bound-d-for-K}. It follows that
\[
\fa_{d+1} = x^{a_1 + 1} W_1 \oplus y^{b_1 + 1} W_2.
\]

Next, set
\begin{eqnarray*}
V_1 & = &  \langle x^{d-c-a_2 -1 - i} y^i \s
\max\{0, d-c-a_1\} \leq i \leq d-c-a_2-1 \rangle \\
V_2 & = &  \langle x^i y^{d-c-b_2 -1 - i} \s
\max\{0, d-c-b_1\} \leq i \leq d-c-b_2-1 \rangle.
\end{eqnarray*}
Note that $V_1$ is trivial if $a_1 = a_2$, and, similarly, $V_2$ is
trivial if $b_1 = b_2$.

Using that $d-c-a_2 - 1 < b_2$  by Assumption
\eqref{eq-lower-bound-d-for-K}, observe that
\[
\fb_{d-c} = \fa_{d-c} \oplus x^{a_2 + 1} V_1 \oplus y^{b_2 + 1} V_2.
\]
Thus, it suffices to prove Equation \eqref{eq-num-state-2} in case,
where we take $x^{a_2 + 1} V_1 \oplus y^{b_2 + 1} V_2$ as subspace
$V$.

Now consider the isomorphism of vector spaces
\[
\al_1: S_{d+1} \to k[t]/(t^{d+2})
\]
that maps $x^{d+1-i} y^i$ onto $\bar{t}^i$, the residue class of $t^i$. For simplicity, we will often identify $t^i$ with  $\bar{t}^i$. The vector spaces above have the following images
\begin{eqnarray*}
V_1' & := & \al_1 ((x+y)^{c+1} x^{a_2 + 1} V_1) \\
     & = & \; (t+1)^{c+1} \langle t^i \s \max\{0, d-c-a_1\} \leq i \leq d-c-a_2-1 \rangle \\
V_2' & := & \al_1 ((x+y)^{c+1} y^{b_2 + 1} V_2)
        = (t+1)^{c+1} \langle t^i \s b_2 + 1 \leq i \leq \min \{d-c, b_1 \} \rangle \\
W_1' & := & \al_1 (x^{a_1 + 1} W_1) = \langle t^i \s 0 \leq i \leq d-a_1 \rangle \\
W_2' & := & \al_1 (y^{b_1 + 1} W_2) = \langle t^i \s b_1 + 1 \leq i \leq d+1 \rangle.
\end{eqnarray*}

Equation \eqref{eq-num-state-2} will be established once we have shown
that the sum \linebreak
$V_1' + V_2' + W_1' + W_2'$ is direct or, equivalently,
\begin{eqnarray}
\label{eq-key-claim-3}
\quad \quad \quad
\dim_k (V_1' + V_2' + W_1' + W_2') & = &
    \dim_k V_1' + \dim_k V_2' + \dim_k W_1' + \dim_k W_2' \\ \nonumber
& = & (d-c-a_2) + (\min \{d-c, b_1\} - b_2) \\
&  & \hspace*{.3cm}  + (d-a_1 + 1) + (d-b_1). \nonumber
\end{eqnarray}

Assumption \eqref{eq-upper-bound-d-for-K} provides  that the
intersection of $W_1'$ and $W_2'$ is trivial.

Assume that $d \leq a_2 + b_2$. Then it follows that all
polynomials in $V_1'$ have degree at most $b_2$, which is less than
the degree of any polynomial in $V_2'$. Similar, though simpler,
considerations show that the intersection of any other two of the
spaces $V_1', V_2', W_1'$, and $W_2'$ is trivial. Thus the sum of
these spaces is a direct sum, as desired. Hence, we may assume that
\begin{equation}
  \label{eq-lower-bound-d-4}
d > a_2 + b_2.
\end{equation}

Assume now that $d-c \geq a_1$. Since $a_1$ is an upper  bound for
$b_1, b_2$, and $a_2$, we get for the following Hilbert functions
\begin{eqnarray*}
h_A (d+1) & = & a_1 + b_1 -d > 0 \quad
   (\text{by Assumption \eqref{eq-upper-bound-d-for-K}}) \\
h_A (d-c) & = & a_1 + b_1 + c + 1 - d\\
h_B (d+1) & = & \max \{ 0, a_2 + b_2 -d \} \\
h_B (d-c) & = & a_2 + b_2 + c + 1 -d > 0 \quad
   (\text{by Assumption \eqref{eq-lower-bound-d-for-K}}).
\end{eqnarray*}

Now, a routine computation shows that Assumption
\eqref{eq-num-assump}  is equivalent to \linebreak $d \leq a_2 + b_2$,  a
contradiction to Assumption \eqref{eq-lower-bound-d-4}. Thus, we
also may assume that
\begin{equation}
  \label{eq-lower-bound-d-3}
d-c < a_1.
\end{equation}
Notice that along with Assumption \eqref{eq-lower-bound-d} this
shows  that  $V_1$ and $V_2$ are non-trivial. Next, we take
quotients modulo $W_2'$. For simplicity, we still denote these
quotients by $V_1', V_2'$, and $W_1'$. They are now subspaces of
$k[t]/(t^{b_1 + 1})$, and, in order to establish Equation
\eqref{eq-key-claim-3}, it is enough to show that
\begin{eqnarray}
\label{eq-key-claim-4}
 \dim_k (V_1' + V_2' + W_1') & = &
    \dim_k V_1' + \dim_k V_2' + \dim_k W_1'  \\ \nonumber
& = & (d-c-a_2) + (\min \{d-c, b_1\} - b_2) \\ \nonumber
&&  + (d-a_1 + 1).
\end{eqnarray}

% and we have to show that
%\[
%V_1' + V_2' + W_1' = V_1' \oplus V_2' \oplus W_1'.
%\]

To prepare for the next modification, we note that multiplication by
$t+1$ on $k[t]/(t^{b_1 + 1})$ is an automorphism as a $k$-vector
space. Indeed, the inverse map is given by multiplication by
\[
(t+1)^{-1} \!\!\!\!\!\mod t^{b_1+1} := \sum_{i = 0}^{b_1} \binom{-1}{i} t^i,
\]
where
\[
\binom{m}{i} := \frac{m (m-1) \cdots (m-i+1)}{i!}
\]
if $i \ge 0$.
Therefore the map
\[
\al_2: k[t]/(t^{b_1+1}) \to k[t]/(t^{b_1+1})
\]
given by multiplication by
\[
(t+1)^{-c-1} \!\!\!\!\!\mod t^{b_1+1} := \sum_{i = 0}^{b_1} \binom{-c-1}{i} t^i
\]
is also an automorphism. Again we list the images of the various subspaces considered above
\begin{eqnarray*}
V_1'' := \al_2 (V_1') & = & \langle t^i \s 0 \leq i \leq d-c-a_2-1 \rangle \\
V_2'' := \al_2 (V_2') & = & \langle t^i \s b_2 + 1 \leq i \leq \min \{d-c, b_1 \} \rangle \\
W_1'' := \al_2 (W_1') & = & (t+1)^{-c-1} \!\!\!\!\!\mod t^{b_1 + 1}  \langle t^i \s 0 \leq i \leq d-a_1 \rangle,
%\\
%W_2'' := \al_2 (W_2') & = & \langle t^i \s b_1 + 1 \leq i \leq d+1 \rangle
\end{eqnarray*}
where we used Assumption \eqref{eq-lower-bound-d-3} to simplify the
description of $V_1''$.

In the next step we consider the elements of $k[t]/(t^{b_1+1})$ as
polynomials whose degree is at most $b_1$ and denote by $D$
differentiation with respect to $t$. Recall that the characteristic
of the field $k$ is zero by assumption. Differentiating  $d-c-a_2$
times (which is a positive number by \eqref{eq-lower-bound-d}), we
get a $k$-linear map
\[
\al_3 := D^{d-c-a_2} : k[t]/(t^{b_1+1}) \to k[t]/(t^{a_2+b_1+c+1-d})
\]
whose kernel is precisely $V_1''$. Furthermore, one computes that
\[
D((t+1)^{-c-1} \!\!\!\!\!\mod t^{b_1 + 1}) = (-c-1) \cdot (t+1)^{-c-2} \!\!\!\!\!\mod t^{b_1}.
\]
Hence we get the following images:
\begin{eqnarray*}
V_2''' := \al_3 (V_2'') & = & \langle t^i \s a_2 + b_2 + c + 1 - d \leq i \leq \min \{a_2, a_2 + b_1 + c -d \} \rangle \\
W_1''' := \al_3 (W_1'') & = & (t+1)^{-d+a_2-1} \!\!\!\!\!\mod t^{a_2+b_1+c+1-d} \langle t^i \s 0 \leq i \leq d-a_1 \rangle,
\end{eqnarray*}
where we used twice the equality
\[
\langle t^i \s 0 \leq i \leq d-a_1 \rangle =
\langle (t+1)^i \s 0 \leq i \leq d-a_1 \rangle
\]
in order to compute $W_1'''$.

Note that  $a_2 + b_2 + c + 1 - d \geq 1$ by Assumption
\eqref{eq-lower-bound-d-for-K}. We are reduced to showing that
\[
\dim_k  (V_2''' + W_1''')  = (\min \{d-c, b_1\} - b_2) + (d-a_1 + 1).
\]

%%%x%%%

In the next step we consider multiplication by $(t+1)^{d-a_2+1}$,
\begin{equation*}
\begin{CD}
\al_4: k[t]/(t^{a_2+b_1+c+1-d})  @>{(t+1)^{d-a_2+1}}>> k[t]/(t^{a_2+b_1+c+1-d}).
\end{CD}
\end{equation*}
We get
$$
\begin{array}{rclcccccccccccccc}
\widetilde{V_2} & := &  \al_4 (V_2''') \\
& = & (t+1)^{d-a_2+1} \langle t^i \s a_2 + b_2 + c + 1 - d \leq i \leq \min
\{ a_2, a_2 + b_1 + c -d \} \rangle ,
\end{array}
$$
and
$$
\begin{array}{rclccccccccccc}
\widetilde{W_1} := \al_4 (W_1''') & = & \langle t^i \s 0 \leq i \leq d-a_1 \rangle.
\end{array}
$$

If $d-a_1 \leq a_2 + b_2 + c - d$, then degree considerations show that $\widetilde{V_2} \cap \widetilde{W_1} = \{0\}$, as desired.
Thus we may assume that
\begin{equation}
  \label{eq-def-m}
d-a_1  > a_2 + b_2 + c - d.
\end{equation}

Now we again compute some Hilbert functions:
\begin{eqnarray*}
h_A (d+1) & = & a_1 + b_1 -d \quad (\text{by} \; \eqref{eq-lower-bound-d-for-K})\\
h_A (d-c) & = & \min \{d-c+1, b_1 + 1\} \quad (\text{by} \;  \eqref{eq-lower-bound-d-3}) \\
h_B (d+1) & = & 0 \quad (\text{by} \; \eqref{eq-lower-bound-d-4}) \\
h_B (d-c) & = & a_2 + b_2 + c + 1 -d \quad (\text{by} \;  \eqref{eq-lower-bound-d}).
\end{eqnarray*}
If $d-c \geq b_1$, then Assumption \eqref{eq-num-assump} provides
$d-a_1 < a_2 + b_2 + c - d$, a contradiction to \eqref{eq-def-m}.
Hence, we may assume
\begin{equation}
  \label{eq-low-bound-d-last}
d-c < b_1,
\end{equation}
and we can rewrite $\widetilde{V_2}$ as
\[
\widetilde{V_2} = (t+1)^{d-a_2+1} \langle t^i \s a_2 + b_2 + c + 1 - d \leq i \leq a_2 \rangle.
\]
Finally, consider the matrix $N$ with $a_2+b_1+c+1-d$ columns whose rows are given by the coefficients of the polynomials listed in the bases of $\widetilde{V_2}$ and $\widetilde{W_2}$ in this order. Thus, $N$ has $(d-b_2-c) + (d-a_1 +1)$ rows, where the submatrix consisting of the last  $d-a_1 +1$ is the concatenation of the identity matrix and a zero matrix. Assumption \eqref{eq-num-assump} means that $N$ has at least as many columns as rows. We have to prove that the rows of $N$ are linearly independent. To this end we consider the submatrix $N'$ of $N$ formed by its first $d-b_2-c$ rows and its last $d-b_2-c$ columns. It is enough  to show that $N'$ is regular. Explicitly, the matrix $N'$ is
\[
N'  = \left ( \binom{d-a_2+1}{b_1+c+1-d-i+j} \right )_{1 \leq i, j \leq n}.
\]
Note that, by \eqref{eq-low-bound-d-last}, $b_1+c+1-d > 0$. Furthermore, using $b_1 \leq a_1$, Assumption \eqref{eq-def-m} provides
\[
d-a_2 > a_1  + b_2 + c - d \geq b_1 + c - d.
\]
Hence, Lemma \ref{lem-det} shows that $N'$ is regular, which
concludes the argument.
\end{proof}

We are ready to establish Proposition \ref{prop-Key} (in its
equivalent form Lemma \ref{lem-key-numerical}), which is the key
result for proving Theorem \ref{prop-WLP-type2}.

\begin{proof}[Proof of Proposition \ref{prop-Key}]
We begin with some computations of Hilbert functions. Adopt the
notation  above Lemma \ref{lem-key-numerical} and set
\[
K := \ker(A\rightarrow B) = \fb/\fa.
\]
By Lemma \ref{lem-SLP-mod}, $(x+y)^{c_1 + 1}$ is a strong Lefschetz
element of $K$. Thus we get (cf.\ Equation \eqref{eq-num-info-K})
%\begin{eqnarray}  \label{eq-hilb-K}
%  \label{eq-hilb-K}
%\dim_k [(x+y)^{c_1 + 1} K]_{d+1}  & = & \min \{h_A (d+1), h_A (d-c_1) - h_B (d-c_1) \}
%\end{eqnarray}
%\min\{h_K (d+1), h_K (d-c_1) \nonumber \\
%& = & \min\{h_A (d+1) - h_B (d+1), h_A (d-c_1) - h_B (d-c_1)\}.
\begin{equation}  \label{eq-hilb-K}
\dim_k [(x+y)^{c_1 + 1} K]_{d+1}   =  \min \{h_A (d+1), h_A (d-c_1) - h_B (d-c_1) \}
\end{equation}

Consider now multiplication by $(x+y)^{c_1 + 1}$ on   Sequence
\eqref{eq-sequence-with-mod}. It induces the following commutative
diagram
\begin{equation*}
\begin{CD}
0 @>>> M_{d-c_1} @>>> A_{d-c_1} @>>> [B/(x+y)^{c_2 - c_1}B]_{d-c_1} @>>>0 \\
& & @VV{\ffi}V   @VV{\alpha}V @VV{\beta}V &  &   \\
0 @>>> M_{d + 1}  @>>>  A_{d + 1} @>>> [B/(x+y)^{c_2 - c_1}B]_{d + 1} @>>>0.  \\
\end{CD}
\end{equation*}
Since $(x+y)^{c_1 + 1}$ is a strong Lefschetz element of $A$, the map $\alpha$ is injective if and only if $h_A(d-c_1) \leq h_A(d+1)$. This is certainly true if $d+1 \leq a_1$ or $d < c_1$.  If $d \geq \max \{a_1, c_1\}$ and $d-c_1 < b_1$, then  $h_A(d+1) = a_1 + b_1 - d$ and $h_A (d-c_1) = d-c_1+1$, thus $h_A(d-c_1) \leq h_A(d+1)$ if and only if $d \leq \frac{a_1 + b_1 + c_1 -1}{2}$. It follows that $\alpha$ is injective if either $d+1 \leq \max\{a_1, c_1\}$ or
\[
\max \{a_1, c_1\} \leq d \leq \min \left \{\frac{a_1 + b_1 + c_1 -1}{2}, b_1 + c_1 - 1 \right\}.
\]
The latter condition is equivalent to
\[
\max \{a_1, c_1\} \leq d \leq \frac{a_1 + b_1 + c_1 -1}{2}.
\]
Now, if the map $\alpha$ is injective, then so is $\ffi$. Hence, we may assume throughout the remainder of this proof that $\alpha$ is surjective,  that
\begin{equation}
  \label{eq-estimate-hilb-A}
h_A (d+1) < h_A (d-c_1),
\end{equation}
and that
\begin{equation}
  \label{eq-first-restr}
d \geq \max \left \{a_1, c_1, \frac{a_1 + b_1 + c_1}{2}\right\}.
\end{equation}

If the map $\beta$ is injective, then the Snake Lemma provides that
$\ffi$ is surjective, and we are done. Therefore, we may also assume
that $\beta$ is not injective, so in particular $[B/(x+y)^{c_2 -
c_1}B]_{d-c_1} \neq 0$. Using that $x+y$ is also a strong Lefschetz
element of $B$, this implies
\begin{equation}
  \label{eq-estimate-hilb-B}
h_B (d-c_1) > h_B (d-c_2).
\end{equation}
Since the Hilbert function $h_B$ is weakly decreasing if $j \geq \min\{a_2, b_2\}$, we conclude that
\begin{equation}
  \label{eq-upper-bound-d}
d - c_2 < \min\{a_2, b_2\}.
\end{equation}

Next, we claim that
\begin{equation}
  \label{eq-other-hilb-B}
h_B (d+1) \leq h_B (d+1 - c_2 + c_1) \quad \mif \quad d \geq   \frac{a_1 + b_1 + c_1}{2}.
\end{equation}
Indeed, the latter assumption implies in particular that $a_1 + b_1
- d \leq d$. Furthermore, using $a_1 + b_1 + c_1 = a_2 + b_2 + c_2$,
we get
\[
h_B (d+1 - c_2 + c_1) = h_B (a_2 + b_2 + c_2 - c_1 - d - 1) =
h_B (a_1 + b_1 - d - 1).
\]
Since $d+1 - c_2 + c_1$ and $a_1 + b_1 - d - 1$ have the same
distance to $\frac{a_2 + b_2}{2}$, it follows that one of the two
numbers is in the interval $[\frac{a_2 + b_2}{2}, d+1]$,
in which $h_B$ is weakly decreasing. We conclude that  $h_B (d+1)
\leq h_B (d+1 - c_2 + c_1)$, as claimed.

Taking Inequalities \eqref{eq-estimate-hilb-A}  and
\eqref{eq-other-hilb-B} into account, our Assertion
\eqref{eq-num-claim} becomes
\begin{eqnarray}
  \label{eq-rev-claim}
\dim_k [\im \ffi]_{d+1} & = & \min\{ h_A(d+1),  h_A(d-c_1) - h_B(d-c_1) + h_B(d-c_2) \} \nonumber \\
& = & \min\{ h_A(d+1),  h_A(d-c_1) - h_B(d-c_1) + h_A(d-c_2) \},
\end{eqnarray}
where the second equality follows from Relation
\eqref{eq-upper-bound-d} because $\min\{a_2, b_2\} \leq b_1$, thus
$h_A (d-c_2) = h_B(d-c_2) = d-c_2 +1$.
\smallskip

If $d < c_2$, then $[\im \ffi]_{d+1} = [(x+y)^{c_1 + 1} K]_{d+1}$, thus the above claim follows from
 Equality \eqref{eq-hilb-K}. This allows us to assume
\begin{equation}
  \label{eq-compare-d-and-c2}
d \geq c_2.
\end{equation}

Now we will consider certain differential operators. Their use is another reason for our assumption about the characteristic being zero. Let $D$ be the differential operator on $S$ defined by
\[
D := \frac{\partial}{\partial (x-y)} := \frac{\partial}{\partial x} - \frac{\partial}{\partial y}.
\]
Its action on polynomials is determined by the following formulas, where $i$ and $j$ are non-negative integers:
\begin{equation*}
  D ((x+y)^j) = 0, \quad D (x^i (x+y)^j) = i x^{i-1} (x+y)^j,  \quad D (y^i (x+y)^j) = i y^{i-1} (x+y)^j.
\end{equation*}
Thus, we get the following exact sequence of $k$-vector spaces
\begin{equation*}
\begin{CD}
0 @>>>  [(x+y)^{c_2 + 1}]_{d+1} @>>> S_{d+1} @>{D^{d-c_2+1}}>> S_{c_2} @>>> 0.
\end{CD}
\end{equation*}
Moreover, for any subspace $V \subset S_{d+1}$, the kernel of the restriction of $D^{d-c_2+1}$ to $V$ is
\begin{equation}
\label{eq-kernel}
  \ker (D^{d-c_2+1} |_V: V \to S_{c_2}) = V \cap [(x+y)^{c_2 + 1}]_{d+1}.
\end{equation}

We now compute the image of various subspaces of $S_{d+1}$. The
subspace $V_0 := [(x^{a_1+1})]_{d+1}$ has a basis $\{x^{d+1-i}
(x+y)^i \s 0 \leq i \leq d-a_1\}$. Since, by Condition
\eqref{eq-upper-bound-d}, $d-c_2 < \min\{a_2, b_2\} \leq \min\{a_1,
b_1\}$, we get that $V_0' := D^{d-c_2+1} (V_0)$ is generated by the
set $\{x^{c_2-i} (x+y)^i \s 0 \leq i \leq d-a_1\}$, thus  $V_0' =
[(x^{a_1 - (d-c_2)})]_{c_2}$. Similarly, we get $D^{d-c_2+1}
([(y^{b_1+1})]_{d+1}) = [(y^{b_1 - (d-c_2)})]_{c_2}$, thus
\begin{equation*}
  D^{d-c_2+1} (\fa_{d+1}) = [(x^{a_1 - (d-c_2)}, y^{b_1 - (d-c_2)})]_{c_2} = \fa_{c_2}',
\end{equation*}
where
\[
\fa' := (x^{a_1 - (d-c_2)}, y^{b_1 - (d-c_2)}).
\]

Next consider $V_2 := [(x^{a_2 + 1} (x+y)^{c_1 + 1})]_{d+1}$.
A basis of it is $\{x^{d+1-i} (x+y)^i \s c_1 + 1 \leq i \leq
d-a_2\}$.  Using again Condition \eqref{eq-upper-bound-d}, it
follows that $V_2' := D^{d-c_2+1} (V_2)$ is generated by the set
$\{x^{c_2-i} (x+y)^i \s c_1 + 1 \leq i \leq d-a_2\}$, thus $V_2'
=[(x^{a_2 - (d-c_2)} (x+y)^{c_1 + 1})]_{c_2}$. A similar computation
provides $D^{d-c_2+1} ([(y^{b_2 + 1} (x+y)^{c_1 + 1})]_{d+1}) =
[(y^{b_2 - (d-c_2)} (x+y)^{c_1 + 1})]_{c_2}$. It follows that
{\small
\begin{equation*}
  D^{d-c_2+1} ([(x+y)^{c_1 + 1} \fb]_{d+1}) = [(x+y)^{c_1 + 1} (x^{a_2 - (d-c_2)}, y^{b_2 - (d-c_2)})]_{c_2} = [(x+y)^{c_1 + 1} \fb']_{c_2},
\end{equation*}}
where
\[
\fb' := (x^{a_2 - (d-c_2)}, y^{b_2 - (d-c_2)}).
\]
As a result of these computations we conclude that there is the
following commutative diagram of vector spaces where the rows are
exact, the vertical maps are surjective, and $K' = \fb'/\fa'$:
\begin{equation*}
{\tiny \begin{CD}
0 @>>> \fa_{d+1} @>>> [(x+y)^{c_1 + 1} \fb + (x+y)^{c_2 + 1} + \fa]_{d+1} @>>> [\im \ffi]_{d+1}  @>>>0 \\
& & @VV{D^{d-c_2+1}}V   @VV{D^{d-c_2+1}}V @VVV &  &   \\
0 @>>> \fa_{c_2}'  @>>>  [(x+y)^{c_1 + 1} \fb' +  \fa']_{c_2} @>>> [(x+y)^{c_1 + 1} K']_{c_2} @>>>0.  \\
\end{CD}}
\end{equation*}
Combined with Formula \eqref{eq-kernel}, the  Snake Lemma implies
the following exact sequence of vector spaces
\begin{equation*}
  0 \to [(x+y)^{c_2 + 1} A]_{d+1} \to [\im \ffi]_{d+1} \to [(x+y)^{c_1 + 1} K']_{c_2} \to 0.
\end{equation*}
Thus, we have shown
\begin{equation*}
  \label{eq-dimensions-1}
\dim_k [\im \ffi]_{d+1} = \dim_k [(x+y)^{c_2 + 1} A]_{d+1}  + \dim_k [(x+y)^{c_1 + 1} K']_{c_2}.
\end{equation*}
Since $[(x+y)^{c_2+1} A]_{d+1}  \hookrightarrow [\im \varphi]_{d+1}
\subset A_{d+1}$, we see that if $[(x+y)^{c_2 + 1} A]_{d+1} =
A_{d+1}$, then $\ffi$ is surjective, and hence we are done.
 Otherwise, we get by the WLP of $A$ that
\begin{equation}
  \label{eq-compare-hilb-A}
\dim_k [(x+y)^{c_2 + 1} A]_{d+1} = h_A (d-c_2) < h_A (d+1).
\end{equation}
By Lemma \ref{lem-SLP-mod}, $x+y$ is a strong Lefschetz  element of
$K'$ and $A$. Thus, the previous equality provides
\begin{equation}
  \label{eq-dimensions}
\dim_k [\im \ffi]_{d+1} = h_A (d-c_2)  +
\min \{h_{K'} (c_2), h_{K'} (c_2- c_1 - 1)\}.
\end{equation}
Comparing this with our Claim in the  form \eqref{eq-rev-claim},  we
see that it remains to show
\begin{eqnarray}
  \label{eq-remaining-claim}
\lefteqn{ \min \{h_{A} (d+1) - h_{A} (d-c_2), h_{K} (d- c_1) \} } & &  \nonumber \\
 & \hspace*{5cm} =  &  \min \{h_{K'} (c_2), h_{K'} (c_2- c_1 - 1) \}.
\end{eqnarray}

However, using our knowledge about the generators  of $\fa'$ and
$\fb'$, for any integer $j$, Formula \eqref{eq-hilb-ci} provides
\begin{eqnarray}
  \label{eq-Hilb-relations}
h_{A'} (j) & = & \max \{ 0, h_A (j+ (d - c_2 +1)) - (d-c_2 +1)\}, \nonumber \\
h_{B'} (j) & = & \max \{ 0, h_B (j+ (d - c_2 +1)) - (d-c_2 +1)\}.
\end{eqnarray}
Using Assumptions \eqref{eq-compare-d-and-c2} and
\eqref{eq-upper-bound-d}, we get
\[
h_B (d-c_2) = d-c_2 + 1.
\]
Hence Assumption \eqref{eq-estimate-hilb-B} implies
\[
h_{B'} (c_2-c_1-1) = h_B (d-c_1)  - (d-c_2 +1) > 0.
\]
Since the Hilbert function of $B'$ is a lower bound for the  Hilbert
function of $A'$, it follows by \eqref{eq-Hilb-relations} that
\[
0 < h_{A'} (c_2-c_1-1)  =  h_A (d-c_1)  - (d-c_2 +1),
\]
thus,
\begin{equation}
  \label{eq-hilb-K-K'}
h_{K'} (c_2 - c_1 - 1) = h_{A'} (c_2 - c_1 - 1) - h_{B'} (c_2 - c_1 - 1)
 = h_K (d-c_1).
\end{equation}

Similarly, using $h_A (d-c_2) = d-c_2 + 1$, Relations
\eqref{eq-compare-hilb-A} and \eqref{eq-Hilb-relations} imply
\begin{equation}
  \label{eq-compare-A-A'}
h_A (d+1) - h_A (d-c_2) = h_{A'} (c_2).
\end{equation}

Finally, we distinguish two cases.

{\it Case 1}. Assume that $h_{K'} (c_2) < h_{K'} (c_2 -c_1 - 1)$.

Since the Hilbert function of $K'$ is non-decreasing up to degree
$a_2 + b_2 - 1 - 2(d - c_2)$, this implies $h_{B'} (c_2) = 0$. Using
also \eqref{eq-compare-A-A'}, we get
\[
h_{K'} (c_2) = h_{A'} (c_2) = h_A (d+1) - h_A (d-c_2).
\]
Combined with Equality \eqref{eq-hilb-K-K'}, the desired claim in
the form \eqref{eq-remaining-claim} follows.

{\it Case 2}. Assume that $h_{K'} (c_2) \geq h_{K'} (c_2 -c_1 - 1)$.

Then, using \eqref{eq-hilb-K-K'} and \eqref{eq-compare-A-A'}, we
obtain
\begin{eqnarray*}
h_K (d-c_1) = h_{K'} (c_2 - c_1 - 1) & \leq & h_{K'} (c_2) \\
& = & h_{A'} (c_2) - h_{B'} (c_2) \\
& = & h_A (d+1) - h_A (d-c_2) - h_{B'} (c_2)\\
& \leq & h_A (d+1) - h_A (d-c_2).
\end{eqnarray*}
This implies the desired Equality \eqref{eq-remaining-claim} also in
this case, thus the proof is complete.
\end{proof}

Our desired result follows now easily.

\begin{proof}[Proof of Theorem \ref{prop-WLP-type2}]
By assumption, the inverse system of our algebra $A$ is generated
by two monomials $x^{a_1}y^{b_1}z^{c_1}$ and $x^{a_2}y^{b_2}z^{c_2}$
of the same degree. Thus, we may assume that $b_2 < b_1\le a_1$, \;
$a_2\le a_1$, and  $c_1<c_2$. It follows that $A$ is isomorphic to
$R/I$, where $R := k[x, y, z]$ and
\begin{eqnarray*}
I & = & (x^{a_1 + 1}, y^{b_1 + 1}, z^{c_1 + 1}) \cap
(x^{a_2 + 1}, y^{b_2 + 1}, z^{c_2 + 1}) \\
& = & (x^{a_1 + 1}, y^{b_1 + 1}) +
z^{c_1 + 1} (x^{a_2 + 1}, y^{b_2 + 1}, z^{c_2 - c_1}).
\end{eqnarray*}
Thus we get the following exact sequence
\[
0 \to (R/(x^{a_2 + 1}, y^{b_2 + 1}, z^{c_2 - c_1})) (-c_1-1) \to A
\to R/(x^{a_1 + 1}, y^{b_1 + 1}, z^{c_1 + 1}) \to 0.
\]
Hence, setting $\tilde{A} := R/(x^{a_1 + 1}, y^{b_1 + 1})$ and
$\tilde{B} := R/(x^{a_2 + 1}, y^{b_2 + 1})$, we get for the Hilbert
function of $A$
\begin{eqnarray*}
h_{A} (j) & = & h_{\tilde{A}} (j) - h_{\tilde{A}} (j-c_1-1) \\
& & \; + h_{\tilde{B}} (j-c_1-1) - h_{\tilde{B}} (j-c_2-1).
\end{eqnarray*}
Let $L$ be any linear form. Then the multiplication map $A_d
\stackrel{L}{\longrightarrow} A_{d+1}$ has maximal rank if
and only if
\begin{eqnarray}
  \label{eq-hilb-restr}
h_{A/L A} (d+1) & = & \max \{0, h_{A} (d+1) - h_{A} (d) \}
\nonumber\\
& = & \max\{0, h_A(d+1) - h_A (d-c_1) + h_B(d-c_1) - h_B (d-c_2) \},
\end{eqnarray}
where $A$ and $B$ are the rings introduced below Proposition
\ref{prop-Key}. By Hausel's theorem, this is true if $d+1 \leq
\frac{a_1 + b_1 + c_1}{2}$. Thus, we assume $d \geq \frac{a_1 + b_1
+ c_1}{2}$. In this case, Proposition \ref{prop-Key} provides (see
\eqref{eq-rev-claim})
\begin{equation}
  \label{eq-cokerM}
\dim_k [(x+y)^{c_1 + 1} M]_{d+1} = \min \{h_A(d+1), h_A(d-c_1)
- h_B(d-c_1) + h_B(d-c_2) \}.
\end{equation}
Let now $L$ be $x+y-z$. Then $A/L A \cong S/\bar{I}$,
where
\[
\bar{I} = (x^{a_1 + 1}, y^{b_1 + 1}) +
(x+y)^{c_1 + 1} (x^{a_2 + 1}, y^{b_2 + 1}, (x+y)^{c_2 - c_1}).
\]
Thus, we have an exact sequence
\[
0 \to (x+y)^{c_1 + 1} M \to A \to A/L A \to 0.
\]
Hence Equation \eqref{eq-cokerM} implies the desired Equation
\eqref{eq-hilb-restr}, and the argument is complete.
\end{proof}

%%%%%%%%%%%%%%%%%%%%%%%%%%%%%%%%%%%%%%%%%%%%%%%%

Theorem \ref{prop-WLP-type2} has a consequence for pure $O$-sequences:

\begin{corollary} \label{pure 2 unimodal}
Pure $O$-sequences of type 2 in three variables are unimodal in the strong sense of Remark \ref{wlpunim}.
\end{corollary}

\begin{remark}{\em
The first and last authors \cite{BZ} showed that without the condition of being monomial, there do exist type two level algebras failing the WLP, even when the inverse system is generated by one monomial and one binomial.
}
\end{remark}

\begin{example}{\em
We remark that monomial ideals of Cohen-Macaulay type
2 \label{useCMtype} that are
not level do not necessarily have unimodal Hilbert function, or
even a Hilbert function that is differentiable while it is
non-decreasing.  Indeed, we have the following examples:

\begin{itemize}

\item  $I = \langle x^3, y^7, z^7, xy^2, xz^2 \rangle$, $R/I$ has
Cohen-Macaulay type 2 and Hilbert function
\[
(1,3,6,7,6,6,7,6,5,4,3,2,1).
\]

\item$I = \langle x^2, y^4, z^4, xy, xz \rangle$, $R/I$ has
Cohen-Macaulay type 2 and Hilbert function
\[
(1,3,3,4,3,2,1).
\]
These are the smallest examples of which we are aware.
\end{itemize}
}
\end{example}

We close this chapter with a short discussion of the characteristic.  We show via a simple example, using the methods of \cite{MMN2}, that the same monomial ideal can exhibit different behavior with respect to the WLP if the characteristic changes.  Furthermore, the characteristics for which the WLP fails arise as the prime factors of the determinant of a certain matrix of integers, and the list of these primes can have ``gaps."  Notice that the following example is level of type 2, and so we already know that it has the WLP in characteristic zero.

\begin{example}{\em
Let $I = (x^{10},y^7, z^7, x^4 y^3, x^4z^5 )$.  Then $R/I$ is level of type 2, and it can be checked (e.g.\ with \cocoa ) that the Hilbert function of $R/I$ is
\[
(1,3,6,10, 15,21,28,33, 36, 36,32,26,19,12,6,2).
\]
We will show that  the linear form $L = x+y-z$ is a Lefschetz element if and only if the characteristic of $k$ is not 2, 5 or 7.  Note that if $k$ is finite, the notion of a ``general'' linear form does not make sense, so we might as well assume that $k$ is infinite.  In this case,  we will conclude (thanks to \cite{MMN2}, Proposition 2.2) that $R/I$ has  the WLP if and only if the characteristic is not 2, 5 or 7.

Note first that thanks to \cite{MMN2}, Proposition 2.1 (c), it is equivalent to show that the multiplication by $L$ from degree 8 to degree 9 is surjective.  Our approach is exactly the same as the analogous arguments in \cite{MMN2}.  The stated map is not surjective if and only if $R/(I,L)$ is non-zero in degree 9.  Working in the ring $\bar R = R/(L)$, the ideal
\[
J = \frac{(I,L)}{(L)} \cong (x^{10}, y^7,(x+y)^7,x^4y^3,x^4(x+y)^5) \subset k[x,y]
\]
 contains the   regular sequence $(x^{10}, y^7)$.  Using linkage, $R/(I,L)$ is non-zero in degree 9 if and only if the ideal quotient $(x^{10},y^7) :J$ is non-zero in degree 6.  Letting $F = a_1 x^6 + \cdots + a_7 y^6$, the condition that $F \in (x^{10},y^7) :J$ gives rise to a system of seven linear equations in the variables $a_1,\dots,a_7$.  One checks that the determinant of the corresponding matrix is 70.  But there is a non-zero $F$ in the ideal quotient if and only if this determinant is zero, from which the result follows.

}\end{example}

%%%%%%%%%%%%%%%%%%%%%%%%%%%%%%%%%%%%%%%%%%%%%%%%%%

\chapter{Failure of the WLP and the SLP} \label{type 2 arb var}

The goal of this chapter is to show that apart from the positive results of Chapter \ref{type 2 three var} and previously known results from \cite{St2}, \cite{watanabe} and \cite{HMNW}, for any given number of variables and type, the SLP and the WLP may fail, even in characteristic zero.

We begin with type 2 in three variables.  We first observe that the SLP does not necessarily hold, even though we showed in the last chapter that the WLP does.

\begin{remark}\label{wlpnotslp}{\em
Let $R = k[x,y,z]$ and let $I$ be a monomial ideal such that $R/I$ is level of type~2.  Then for both kinds of ideals given in Proposition \ref{two forms}, $R/I$ does {\em not} necessarily have the SLP.
For the kind given in (a), consider first a monomial ideal in four variables of the form given in Theorem~\ref{nonwlp in n vars}, i.e.\ $I = (x_1^4,x_2^3,x_3^3,x_4^4,x_1^2x_4^2)$.  Since
\[
R/I \cong \left (k[x_1,x_2,x_4]/(x_1^4,x_2^3,x_4^4,x_1^2x_4^2) \right )[x_3]/(x_3^3)
\]
(where $k[x_1,x_2,x_4]/(x_1^4,x_2^3,x_4^4,x_1^2x_4^2)$ is of the form (a) in Proposition~\ref{two forms}, hence level), and $R/I$ fails the WLP, it follows from Proposition 2.2 of \cite{HP} that the algebra $k[x_1,x_2,x_4]/(x_1^4,x_2^3,x_4^4,x_1^2x_4^2)$ fails to satisfy the SLP.  Indeed, \cocoa\ verifies that multiplication by the cube of a general linear form fails to be surjective from degree 2 to degree 5.

Turning to ideals of form (b) in Proposition \ref{two forms}, it is not hard to find an example, using \cocoa, to show that the SLP does not necessarily hold.  For example, for the ideal
\[
I  = (x^7,y^4,z^4,x^3y^2,x^3z^2),
\]
$R/I$ has Hilbert function $(1,3,6,10,13,13,10,6,2)$, while for a general linear form $L$,  \linebreak $R/(I,(L^3))$ has Hilbert function $(1,3,6,9,10,7,1)$.  Thus the multiplication by $L^3$ from degree 3 to degree 6 fails to be an isomorphism, and the SLP fails.  (According to our experiments on \cocoa, this is the smallest such algebra.)
}
\end{remark}

We now turn to  monomial Artinian level  algebras of type 2 in four or more variables.  We start  by showing that, in principle, we can easily
 describe  pure $O$-sequences of type 2. Let $x^a = x_1^{a_1}
\cdots x_r^{a_r} \in R = k[x_1,\ldots,x_r]$ be a monomial of
degree $|a| = a_1 + \cdots +a_r$. Then its annihilator is the
complete intersection
\[
\Ann (x^a) = (x_1^{a_1+1},\ldots,x_r^{a_r + 1}).
\]
Using the Koszul complex, one gets the well-known formula for the
Hilbert function of a complete intersection. This allows us to
describe all possible pure $O$-sequences of type 2.

\begin{prop}
  \label{prop:hilb-type-2}
Let $I \subset R$ be a monomial ideal such that $A = R/I$ is level
of type 2. Then there are two monomials $x^a, x^b \in R$ of the
same degree such that $I = \Ann (x^a, x^b)$ and the Hilbert
function of $A$ is given by
\begin{equation}
  \label{eq:formula-type-2}
h_A (j) = h_{R/\Ann (x^a)} (j) + h_{R/\Ann (x^b)} (j) - h_{R/\Ann
(x^c)} (j),
\end{equation}
where $c := \min \{a, b\}$ is obtained by taking the minimum
componentwise, i.e.
\[
c_i := \min \{a_i, b_i \} \quad (i = 1,\ldots,r).
\]
\end{prop}

\begin{proof}
Observe that $I = \Ann (x^a) \cap \Ann (x^b)$ and
\[
\Ann (x^a) + \Ann (x^b) = \Ann (x^c).
\]
Hence we get the following Mayer-Vietoris sequence
\[
0 \to R/I \to R/\Ann (x^a) \oplus R/\Ann (x^b) \to R/\Ann (x^c)
\to 0.
\]
Now the claim follows from the additivity of vector space
dimensions along exact sequences.
\end{proof}

\begin{remark} \em
The reader can look at \cite{Za3}, Theorem 2.3 and Remark 2.4, for a description of the
$h$-vectors of Artinian (not necessarily monomial) level algebras of type 2 and higher.
\end{remark}

\begin{remark} \em
  \label{rem:all-type-2}
Conversely, given any two monomials in $R$ of the same degree,
their annihilator is a monomial Artinian  ideal $I$ such that
$R/I$ is level of type 2. Therefore, Formula
\eqref{eq:formula-type-2} describes indeed all possible pure
$O$-sequences of type 2.
\end{remark}

\begin{remark} \em
  \label{rem:suff-cond}
Since the Hilbert function of a complete intersection is strictly
unimodal (in the sense of Remark \ref{wlpunim}), it
follows that pure $O$-sequences of type 2 are strictly unimodal in degrees
$j \geq |c|$. This implies the following condition for
unimodality.
\end{remark}

\begin{corollary}
  \label{cor:suff-unimodular}
We keep the notation of Proposition \ref{prop:hilb-type-2}. If
$|c| \leq {\large \lceil} \frac{|a|}{2} \large{\rceil}$, then the Hilbert function
of $A$ is strictly unimodal.
\end{corollary}

We now turn to the question of the WLP.  We first need three preliminary results.

\begin{lemma}
  \label{hilb-tensor}
Let $A$ be an Artinian $k$-algebra whose $h$-vector
$(a_0,\ldots,a_{2t})$ is unimodal with peak in degree $t \geq 2$,
i.e.
\begin{equation}
  \label{eq-unimodal}
a_0 \leq a_1 \leq  \cdots \leq a_t \geq a_{t+1} \cdots \geq a_{2t} \geq 1,
\end{equation}
and assume that it satisfies
\begin{equation}
  \label{eq-inequal}
a_{t-j} \leq a_{t+j} \leq a_{t-j+1} \quad \text{whenever}\;
1 \leq j \leq t.
\end{equation}
Then the $h$-vector $(h_0,\ldots,h_{2t+2})$ of $A \otimes_k
k[x]/(x^3)$ is unimodal with peak in degree $t+1$, i.e.
\begin{equation}
  \label{eq-unimodal-h}
h_0 \leq h_1 \leq \cdots \leq h_{t+1} \geq h_{t+2} \cdots \geq h_{2t+2} \geq 1,
\end{equation}
and it satisfies
\begin{equation}
  \label{inequal-h}
h_{t+1-j} \leq h_{t+1+j} \leq h_{t-j+2} \quad \text{whenever}\;
1 \leq j \leq t+1.
\end{equation}
\end{lemma}

\begin{proof}
Setting $a_{2t+2} = a_{2t+1} = a_{-1} = a_{-2} = 0$, we have, for
all $j = 0,\ldots,2t+2$, that
\[
h_j = a_j + a_{j-1} + a_{j-2}.
\]
This provides the claim by a routine computation. We  give the
details only for proving unimodality.

The above formula immediately implies
\[
h_0 \leq h_1 \leq \cdots \leq h_t \quad \text{and} \quad
h_{2t} > h_{2t+1} > h_{2t+2}.
\]
Moreover, $h_t \leq h_{t+1}$ is equivalent to $a_{t-2} \leq
a_{t+1}$, which is true because $a_{t-2} \leq a_{t-1} \leq a_{t+1}$,
where the last estimate holds by Inequality \eqref{eq-inequal}
applied with $j = 1$.

If $2 \leq j \leq t-1$, then $h_{t+j} \geq h_{t+j+1}$ by the
unimodality of the $h$-vector of $A$. Using Inequality
\eqref{eq-inequal} with $j = 2$ we get $a_{t+2} \leq a_{t-1}$, which
implies $h_{t+1} \geq h_{t+2}$. This completes the proof of the
desired unimodality \eqref{eq-unimodal-h}. The proof of
\eqref{inequal-h} is similar. We leave it to the reader.
\end{proof}

\begin{lemma} \label{Lem:TensorProduct}
Let $A=A'\otimes A''$ be a tensor product of two graded  Artinian
$k$-algebras $A'$ and $A''$. Let $L' \in A'$ and $L'' \in A''$ be
linear elements, and set $L :=L'+L'' = L'\otimes 1+1\otimes L'' \in
A$. Then:
\begin{itemize}

\item[(a)] If the multiplication maps $\times L': A'_{i-1} \to
A'_{i}$ and $\times L'': A''_{j-1} \to A''_{j}$ are both not
surjective, then the multiplication map
\[
\times L : A_{i+j-1}\longrightarrow A_{i+j}
\]
is not surjective.

\item[(b)] If the multiplication maps $\times L': A'_{i} \to
A'_{i+1}$ and $\times L'': A''_j \to A''_{j+1}$ are both not
injective, then the multiplication map
\[
\times L : A_{i+j}\longrightarrow A_{i+j+1}
\]
is not injective.
\end{itemize}
\end{lemma}

\begin{proof}
(a) By assumption, $A'/(L')$ and $A''/(L'')$ are non-zero in degrees
$i$ and $j$, respectively. Hence the tensor product
$$
A'/(L') \otimes A''/(L'')
$$
is non-zero in degree $i+j$. However, this tensor  product is
isomorphic to
$$
A/(L',L'')
$$
which is a quotient of $A/L=A/(L'+L'')$. Hence $A/(L)$ is  also
non-zero in degree $i+j$, and the multiplication by $L$ is not
surjective from degree $i+j-1$ to degree $i+j$.

Part (b) follows from (a) by dualizing or by a direct argument, because
the tensor product of the two kernels is non-trivial.
\end{proof}

\begin{corollary} \label{cor:TensorProduct}
Let $A=A'\otimes A''$ be a tensor product of two graded  Artinian
$k$-algebras $A'$ and $A''$. If $\dim_k A'_{i-1}<\dim_k A'_{i}$ and
$\dim_k A''_{j-1}<\dim_k A''_{j}$, then the multiplication
$$
\times L : A_{i+j-1}\longrightarrow A_{i+j}
$$
by any linear form $L$ is not surjective.
\end{corollary}

\begin{proof}
Each linear form $L$ in $A$ can be written as $L=L'+L'' = L'\otimes
1+1\otimes L''$, where $L'$ is a linear form in $A'$ and $L''$ is a
linear form in $A''$. Because of the given inequalities, the
assumptions in Lemma \ref{Lem:TensorProduct}(a) are satisfied, and
the claim follows.
\end{proof}

Combining the above results we get:

\begin{prop} \label{nonwlp in n vars}
Let $R = k[x_1,\dots,x_r]$, with $r \geq 4$,  where $k$ is any
field, and consider the ideal
\[
I = (x_1^4,x_2^3,\dots,x_{r-1}^3,x_r^4,x_1^2x_r^2).
\]
Equivalently,  $I$ is given by the inverse system  generated by $m_1
= x_1x_2^2 \cdots x_{r-1}^2 x_r^3$ and $m_2 = x_1^3 x_2^2 \cdots
x_{r-1}^2 x_r$, so $R/I$ is an Artinian level algebra, and its
Hilbert function is a pure $O$-sequence $(h_0,\ldots,h_{2r})$  of
type 2. Furthermore, it satisfies
\begin{equation*}
%  \label{eq-unimodal}
h_0 \leq h_1 \leq \cdots \leq h_r \geq h_{r+1} \cdots \geq h_{2r} = 2,
\end{equation*}
and
\begin{equation*}
  \label{inequal}
h_{r-j} \leq h_{r+j} \leq h_{r-j+1} \quad \text{whenever}\;
1 \leq j \leq r.
\end{equation*}
In particular, the Hilbert function of $R/I$ is unimodal. However,
 $R/I$ fails the WLP.  Specifically, the
multiplication by a general linear form fails to have maximal rank (because surjectivity fails)
from degree $r$ to degree $r+1$.
\end{prop}

\begin{proof}
\noindent \underline{Step 1}: We first show the claimed properties
of the Hilbert function.

Setting $A' = k[x,y]/(x^4,x^2y^2,y^4)$, our  algebra
$$A=k[x_1,x_2,\dots,x_r]/(x_1^4,x_1^2x_2^2,x_2^4,x_3^3,x_4^3,\dots,x_r^3)$$
can be obtained as the  tensor product
$$
A \overset{\sim}{=} k[x,y]/(x^4,x^2y^2,y^4)\otimes k[x]/(x^3)
 \otimes \cdots \otimes k[x]/(x^3) = A'\otimes A''.
$$
The $h$-vector of $A' \otimes k[x]/(x^3)$ is $(1, 3, 6, 9, 9, 6,
2)$. It satisfies the assumptions of Lemma \ref{hilb-tensor}. Hence,
applying this result repeatedly, we get the desired properties of
the Hilbert function of $A$.
\smallskip

\noindent \underline{Step 2}: Now we have to show that
multiplication  by a general linear form, from degree $r$ to degree
$r+1$, is not surjective. To this end we apply
Corollary~\ref{cor:TensorProduct} with $i=3$ and $j=r-2$.
 Indeed,
$\dim_k A'_2=3<4=\dim_k A'_3$ and $\dim_k A''_{r-3}<\dim_k A''_{r-2}$,
since the Hilbert series $H(A'',t)=(1+t+t^2)^{r-2}$ has the largest
coefficient in degree $r-2$. This proves the proposition.
\end{proof}

\begin{remark}{\em
From Theorem~\ref{nonwlp in n vars} we note that in four variables, the example produced there has socle degree 8 and type 2.  Using \cocoa\ we have even been able to find a monomial level algebra of type 2 in four variables with socle degree 6.  Indeed, our search reveals that the smallest pure $O$-sequence of type 2 in codimension 4 corresponding to an algebra that does not have the WLP is
\[
(1, 4, 10, 16, 15, 8, 2).
\]
This is achieved, for instance, by the ideal
\[
(x_4^3, x_3^3, x_2^3, x_1^3, x_2^2x_4^2, x_1^2x_4^2, x_2^2x_3^2, x_1^2x_3^2),
\]
which arises from the inverse system generated by $x_1^2x_2^2x_3x_4,\  x_1x_2x_3^2x_4^2$.
}
\end{remark}

\begin{remark}{\em
\noindent (i)  Notice that the previous example shows how suddenly
the behavior of $O$-sequences changes: their first half can indeed
be completely  characterized (Theorem
\ref{firsthalf}), and moreover, Hausel's theorem even guarantees the
existence of a $g$-element throughout all of the first half. The
last example instead further substantiates the fact (already
highlighted by the non-unimodality and non-differentiability
results of Chapter~\ref{pureness and diff}) that, once past the first half, the behavior
of an $O$-sequence may become  very wild. In this instance, the
$g$-element even ceases to exist in the degree immediately after the
middle.

\medskip

\noindent (ii)  Harima and Watanabe \cite{HW08} have shown that the tensor product of two algebras with the WLP can fail to have the WLP, and that the same holds for the SLP.  Their examples are far from being level, though.  Theorem~\ref{nonwlp in n vars} shows that even in the most seemingly benign situations, the WLP and the SLP can fail to be preserved under tensor products. For instance, taking $n=4$ in Theorem~\ref{nonwlp in n vars}, the given algebra is isomorphic to $k[x_1,x_4]/(x_1^4,x_1^2x_4^2,x_4^4) \otimes k[x_2,x_3]/(x_2^3,x_3^3)$, both of which are level and have the SLP, and the tensor product is also level, but it fails even the WLP.
}
\end{remark}

We can also give examples of monomial level algebras in any even number of variables that fail the WLP because of injectivity.  Examples also exist experimentally for every odd number of variables that we have checked ($\geq 5$), but we have not included a proof since the statement is not quite as clean (it seems to fail the WLP only ``half'' of the time for a given number of variables), and here we already have arbitrarily  many cases in arbitrarily many  variables.

\begin{prop} Let $R = k[x_1,x_2, \dots , x_r]$, where $k$ is an infinite field and $r$ is even.  Let
\[
I = (x_1^N,x_2^N,x_1^{N-2}x_2^{N-2}, x_3^{N-1}, x_4^{N-1}, \dots, x_r^{N-1})
\]
where $N \geq 5$.
Equivalently,  the ideal $I$ is given by the inverse system $m_1 = x_1^{N-3} x_2^{N-1}x_3^{N-2} \cdots x_r^{N-2}$ and $m_2 = x_1^{N-1} x_2^{N-3}x_3^{N-2} \cdots x_r^{N-2}$, so  $R/I$ is an Artinian level algebra, and its Hilbert function is a pure $O$-sequence of type 2.  However, $R/I$ fails the WLP.  Specifically, the multiplication by a general linear form fails to have maximal rank (because injectivity fails) from degree $\frac{r}{2}N-r$ to degree $\frac{r}{2}N-r+1$.
\end{prop}

\begin{proof}
Notice that $\frac{r}{2}N - r = \frac{r(N-2)}{2}$.  Since the socle degree is $r(N-2)$, this is the first degree in which Hausel's theorem does not guarantee injectivity.  We have seen that we may use  $L=x_1+x_2+ \dots + x_{r-1}+x_r$ as the Lefschetz element. Then we have that
\[
F= \sum_{i=0}^{\frac{r}{2}N- r} (-1)^i (x_1+x_2+\dots + x_{\frac{r}{2}})^i (x_{\frac{r}{2}+1} + \dots + x_r)^{\frac{r}{2}N-r-i}
\]
 is non-zero in $R/I$, since  for example the term $x_1^{N-1} x_3^{N-3} x_4^{N-2} \cdots x_{\frac{r}{2}}^{N-2} x_{\frac{r}{2}+1}^{N-2}$ is non-zero.

However one can check that $LF$ is either the sum or the difference (depending on $r$) of $(x_1+ x_2 + \dots + x_{\frac{r}{2}})^{\frac{r}{2} N - r +1}$ and $(x_{\frac{r}{2}+1} + \dots + x_r )^{\frac{r}{2}N - r +1}$.  But this is zero in $A$, since both $k[x_1,x_2,\dots,x_{\frac{r}{2}}]/(x_1^N, x_1^{N-2}x_2^{N-2}, x_2^N, x_3^{N-1},\dots,x_{\frac{r}{2}}^{N-1})$ and $k[x_{\frac{r}{2}+1} , \dots, x_r]/(x_{\frac{r}{2}+1}^{N-1}, \dots , x_r^{N-1})$ have socle degree $\frac{r}{2} N - r$.

It remains to prove that
\begin{equation} \label{to prove 1}
h_{R/I} \left (\frac{r}{2}(N-2) \right )  < h_{R/I} \left (\frac{r}{2}(N-2) +1 \right ).
\end{equation}
To do this, we will be reducing the problem to a study of the Hilbert function of an Artinian complete intersection $\mathfrak c$ of type $(N-2,N-1,\dots,N-1)$ in a polynomial ring $R'$ with $r-1$ variables.  There will actually be two such complete intersections, but their Hilbert functions are the same.   In anticipation of this, we note the following facts, for use later (sometimes without comment):

\begin{equation} \label{somefacts}
\begin{array}{cl}
\bullet & \hbox{Viewing $ \mathfrak{c} \subset R' \subset R$, we have $\Delta h_{R/\mathfrak c} = h_{R'/\mathfrak c}$, where $\Delta h(t) = h(t) - h(t-1)$.}\\

\bullet & \hbox{The socle degree of $R'/\mathfrak c$ is $(r-1)(N-2)-1$.}\\

\bullet & \hbox{The Hilbert function of $R'/\mathfrak c$ is symmetric about $\displaystyle t_0 = \frac{(n-1)(N-2)-1}{2}$}\\
& \hbox{(which may or may not be an integer).}\\

\bullet & \hbox{The Hilbert function of $R'/\mathfrak c$ is strictly increasing until degree $\lfloor t_0 \rfloor$ (see \cite{RRR},}\\
& \hbox{Theorem 1)}.
\end{array}
\end{equation}

In the polynomial ring $R$, we will denote by $h_{(a_1,\dots,a_r)}(t)$ the Hilbert function of a complete intersection of type $(a_1,\dots,a_r)$, and if $a_1 = \dots = a_d = a$, we will write $h_{(a^d,a_{d+1},\dots,a_r)}(t)$.

First we write $I$ as a basic double link:
\[
I = x_1^{N-2} (x_1^2,x_2^{N-2}, x_3^{N-1}, \dots, x_r^{N-1}) +
(x_2^N,x_3^{N-1},\dots,x_r^{N-1}).
\]
It follows from Lemma \ref{lem:BDL} that for any $t$, we can write $h_{R/I}(t)$ as a sum of two codimension $r$ complete intersection Hilbert functions:
\[
h_{R/I}(t)  =  h_{(N-2,N,(N-1)^{r-2})}(t) + h_{(2,N-2,(N-1)^{r-2})}(t-(N-2))
\]
Note that these complete intersections have $r-1$ common generator degrees, i.e.\ they both represent quotients of a complete intersection in $R$ of type $(N-2,(N-1)^{r-2})$.  Let $h'$ be the Hilbert function of such a complete intersection (which has depth 1).  The desired inequality (\ref{to prove 1}) now becomes (after a small computation)
\[
\begin{array}{c}
\left [{h'} \left (\frac{r(N-2)}{2} \right ) - { h'} \left (\frac{(r-2)(N-2)}{2} -2 \right ) \right  ] + \left [ { h'} \left ( \frac{(r-2)(N-2)}{2} \right ) - {h'} \left ( \frac{(r-2)(N-2)}{2} -2 \right )  \right ] \\ \\
< \left [{ h'} \left (\frac{r(N-2)}{2} +1 \right ) - {h'} \left (\frac{(r-2)(N-2)}{2} -1 \right ) \right  ] \\ \\+ \left [ { h'} \left ( \frac{(r-2)(N-2)}{2} +1 \right ) - {h'} \left ( \frac{(r -2)(N-2)}{2} -1 \right )  \right ] .
\end{array}
\]
That is, we have to prove that
\[
\begin{array}{c}
2 \cdot \Delta {h'} \left ( \frac{(r-2)(N-2)}{2} -1 \right ) < \Delta { h'} \left ( \frac{r(N-2)}{2} +1 \right ) + \Delta { h'} \left ( \frac{(r-2)(N-2)}{2} +1 \right ).
\end{array}
\]
Thanks to the symmetry of the Hilbert function, the desired inequality becomes
\[
\begin{array}{c}
2 \cdot \Delta {h'} \left ( \frac{(r-2)(N-2)}{2} -1 \right ) < \Delta { h'} \left ( \frac{(r-2)(N-2)}{2} \right ) + \Delta { h'} \left ( \frac{(r-2)(N-2)}{2} +1 \right ).
\end{array}
\]
Since the Hilbert function is strictly increasing until degree $\lfloor t_0 \rfloor$, this inequality is obvious, and we have finished.
\end{proof}

Combining previous results with the material from Chapter \ref{type 2 three var} and this chapter, we now know that in two variables, monomial level algebras of any type have the WLP \cite{HMNW}, in three variables those of type 2 necessarily have the WLP but type 3 do not \cite{BK}, and in four or more variables those of type 2 do not necessarily have the WLP.  To complete the picture, we now show that no other combination of number of variables and type force the WLP.

\begin{prop} \label{type d geq 3}
For each $d \geq 3$ there exists a monomial ideal $I$ in $R = k[x,y,z]$ such that $R/I$ is level of type $d$, and $R/I$ fails the WLP.
\end{prop}

\begin{proof}
For $d=3$ we have the example of Brenner and Kaid, which has $h$-vector $(1,3,6,6,3)$ and fails injectivity (in particular) from degree 2 to degree 3.  For $d=4$, one can check using \cocoa\ that the ideal
\[
(z^2, y^3z, x^3z, x^5, x^3y^2, y^5, x^2y^4)
\]
has $h$-vector $(1,3,5,7,7,4)$ and fails the WLP.  (This example was constructed using the methods of Theorem 3.1 of \cite{Mig}.)

For the general case, we will begin with a class of level monomial algebras all having type 3, and build from them an example for each $d \geq 5$.  Let $n \geq 1$ be an integer, and let
\[
I_{2n} = (x^{1+2n}, y^{1+2n}, z^{1+2n},xyz).
\]
Using \cite{MMN2}, Proposition 6.1, Lemma 7.1 and Corollary 7.4, we observe that $R/I_{2n}$ is level with socle degree $4n$ and socle type 3, the values of the Hilbert function are equal in degrees $2n$ and $2n+1$, and $R/I_{2n}$ fails the WLP because the multiplication by a general linear form fails to be an isomorphism from degree $2n$ to degree $2n+1$.

Now recall that a codimension 2 complete intersection of type $(a,b)$ has $h$-vector of the form
\[
\underline{h} = (1,2,3,4,\dots,a-1,a,a,\dots,a,a-1,\dots,4,3,2,1).
\]
For any $n \geq 1$, we can choose $a$ and $b$ so that

\begin{itemize}
\item[(a)] $1+2n \leq a \leq b$;
\item[(b)] $\underline{h}(2n+1) \leq \underline{h}(2n+2)$;
\item[(c)] $\underline{h}(4n+1)$ takes any desired value between 2 and $4n+2$.
\end{itemize}

Given $d \geq 5$, choose $n$ so that $d \leq 4n+5$, and then choose $a$ and $b$ satisfying the above conditions and such that $\underline{h}(4n+1) = d-3$.  For example, when $d=5$ we can choose $n=1$ (giving the Brenner-Kaid ideal), and $a=3$, $b=5$, so $3 = \underline{h}(3) = \underline{h}(4)$, and $\underline{h}(5) = 2 = d-3$.

Now let
\[
I = x\cdot I_{2n} + (y^a,z^b) + (x,y,z)^{4n+2}.
\]
(The first two summands are a basic double link, and the last summand is simply to truncate in degree $4n+1$.)  Standard methods show that $R/I$ is level of type $d$, and that
\[
h_{R/I}(2n+1) \leq h_{R/I}(2n+2).
\]
Let $J = (y^a,z^b)$.  We now claim that $R/I$ fails the WLP. From the  fact that   $R/J(-1) \rightarrow R/J$ is injective in all degrees, and using the following commutative diagram,
  { \[
  \begin{array}{cccccccccccccccccccccc}
  0 \!\!&\!\!\! \rightarrow \!\!&\!\! [R/J(-2)]_{2n+2} \!\!&\!\! \rightarrow
  \!\!&\!\! [R/I_{2n}(-2)]_{2n+2} \oplus [R/J(-1)]_{2n+2} \!\!&\!\!
  \rightarrow \!\!&\!\! [R/I(-1)]_{2n+2} \!\!&\!\!\rightarrow
  \!\!\!&\!\! 0 \\[.5ex]
  && \downarrow && \downarrow && \downarrow \\[.5ex]
    0 \!\!&\!\!\! \rightarrow \!\!&\!\! [R/J(-1)]_{2n+2} \!\!&\!\!
    \rightarrow \!\!&\!\! [R/I_{2n}(-1)]_{2n+2} \oplus [R/J]_{2n+2}
    \!\!&\!\! \rightarrow \!\!&\!\! [R/I]_{2n+2} \!\!\!&\!\!
    \rightarrow \!\!&\!\! 0
 \end{array}
  \]}
it follows immediately that injectivity for $R/I$ fails from degree $2n+1$ to degree $2n+2$.
\end{proof}

    We now turn to four or more variables.

    \begin{lemma} \label{nplus1}
    Suppose that a monomial level algebra $R/I$ exists in $r$ variables, with type $d$ and socle degree $e$, failing the WLP.  Then there exists a monomial level algebra $R'/I'$ in $r+1$ variables, with type $d+1$ and socle degree $e$, failing the WLP.
    \end{lemma}

    \begin{proof}
    We use inverse systems.   Let $M$ be the inverse system in $k[y_1,\dots,y_r]$ giving $R/I$.  It suffices to replace  $M$ inside $k[y_1,\dots,y_r]$  by $M' = (M,y_{r+1}^e)$ in $k[y_1,\dots,y_{r+1}]$. Then the multiplication by $y_1+....+y_r$ is injective or surjective between two degrees $i$ and $i+1$ in $M$, if and only if the same is true for $M'$ with respect to $y_1+...+y_{r+1}$ (simply because $y_{r+1}$ kills all monomials of $M'_i$ but $y_{r+1}^i$).
    \end{proof}

We have seen in Proposition \ref{nonwlp in n vars} that type 2 level monomial algebras need not have the WLP in $\geq 4$ variables.  Combining Lemma \ref{nplus1} with Proposition \ref{type d geq 3}, then, we see that the only remaining open case is that of type 3 in four variables.  For this we borrow from Proposition \ref{nonwlp in n vars}.

\begin{example}{\em
Let $I$ be the ideal obtained from the inverse system generated by $m_1 = wx^2y^2z^3$, $m_2 = w^3x^2y^2z$, and $m_3 = w^2x^2y^2z^2$.  Then \cocoa\ verifies that $R/I$ has $h$-vector
\[
(1,4,10,18,25,26,20,10,3),
\]
 and injectivity fails from degree 4 to degree 5.
}\end{example}

We can summarize the results of this chapter and the previous one as follows:

\begin{theorem} \label{answernd}
Fix two positive integers $r$ and $d$. Then all level monomial Artinian algebras of codimension $r$ and type $d$ possess the WLP if and only if at least one of the following is true:
\begin{itemize}
\item[(a)] $r = 1$ or $2$;
\item[(b)]  $d=1$;
\item[(c)]  $r=3$ and $d=2$.
\end{itemize}
Furthermore, all monomial level algebras of codimension $r$ and type $d$ possess the SLP if and only if either (a) or (b) hold.
\end{theorem}

%%%%%%%%%%%%%%%%%%%%%%%%%%%%%%%%%%%%%%%%%%%%%%%%%%

\chapter{Remarks on pure $f$-vectors } \label{pure f-vectors}

Recall that a {\em pure $f$-vector} is the $f$-vector of a pure simplicial complex, i.e. one where the facets have the same dimension. The {\em type} of the $f$-vector is the number of facets.  As mentioned earlier, these form an important subclass of the pure $O$-sequences.  Indeed, given a pure simplicial complex, if we label the vertices with distinct variables then the squarefree monomials obtained as the product of the vertices of each facet generate an inverse system whose Hilbert function is precisely the $f$-vector of the original simplicial complex.

A beautiful motivation for the study of pure $f$-vectors concerns the existence of projective planes of order $d$.  Recall that a projective plane of order $d>1$ is a finite projective plane with $d+1$ points on each line.  Such a plane has $q = d^2 +d+1$ points and $q$ lines.  Then  a projective plane of order $d$ exists if and only if there exists a pure $O$-sequence of the form
\[
\underline{h} = \left (1,q,q\binom{d+1}{2},q\binom{d+1}{3},\dots,q \binom{d+1}{d}, q\binom{d+1}{d+1} \right ),
\]
which in turn holds if and only if this $\underline{h}$ is a pure $f$-vector.
 Indeed, if we start with a projective plane $\Lambda$ of order $d$, then assign a different variable to each point of $\Lambda$ and consider the inverse system generated by the $q$ squarefree monomials of degree $d+1$ (one for each line), each obtained as the product of the points on the corresponding line.  The entries of $\underline{h}$ are forced by the axioms for a projective plane.  The other implications come similarly.
See \cite{bjorner}, page 38; there  Bj\"orner also related, in a similar fashion,  other  combinatorial designs to pure $O$-sequences.

The determination of which $d$ allow the existence  of a projective plane of order $d$ is not known.  It is known from finite field theory that, if $d$ is a prime power, then such a plane exists. (In fact, for many such values of $d$, there exist several nonisomorphic projective planes of order $d$.) The wide-open conjecture is that the converse holds.  An important partial result is the theorem of Bruck and Ryser \cite{BR}, which says that if $d$ is a positive integer of the form $4k + 1$ or $4k + 2$ and $d$ is not equal to the sum of two integer squares, then $d$ does not occur as the order of a finite projective plane.

 Since a classification of pure $f$-vectors is  far from being achieved, it is of interest to attack special cases.  The first of these concerns the interval problem.  In Chapter \ref{ICP section} we proved it for pure $O$-sequences of socle degree 3.  Now we consider the same question for pure $f$-vectors.  We do not have a complete solution in the general case, but we first solve the problem for the Cohen-Macaulay case, which is a subclass of the class of pure simplicial complexes.

\begin{prop}
   \label{prop-ICP-for-f-CM-vectors}
The set of $f$-vectors of Cohen-Macaulay simplicial complexes has
the interval property.

More precisely,  if, for some positive integer $\alpha $, both
$\underline{f} = (1,f_0,\ldots,f_j,f_{j+1}, f_{j+2}, \\
\ldots,f_d)$
and $\underline{f'} = (1,f_0,\ldots,f_j,f_{j+1} + \alpha,
f_{j+2},\ldots,f_d)$ are $f$-vectors of Cohen-Macaulay simplicial
complexes, then, for each integer $\beta =0,1,..., \alpha$, also
$\underline{\tilde{f}} = (1,f_0,\ldots,f_j,f_{j+1} + \beta,
f_{j+2},\ldots,f_d)$ is an $f$-vector of a Cohen-Macaulay simplicial
complex.
\end{prop}

\begin{proof}
If $\underline{f} = (1,f_0,\ldots,f_d)$ is the $f$-vector of a
simplicial complex $\Delta$, then its $h$-vector $\underline{h} =
(h_0,\ldots,h_d)$ is defined by the equality of polynomials in $t$
\[
\sum_{ i = 0}^d h_i t^i = \sum_{i=0}^d f_{i-1} t^i (1-t)^{d-i},
\]
where we allow the possibility that the last entries of
$\underline{h}$ vanish (see, e.g., \cite{St3}).

Let $\underline{h} = (h_0,\ldots,h_d)$,  $\underline{h'} =
(h_0',\ldots,h_d')$, and $\underline{\tilde{h}} =
(\tilde{h}_0,\ldots,\tilde{h}_d)$ be the sequences associated to the
given sequences $\underline{f}$, $\underline{f'}$, and
$\underline{\tilde{f}}$. According to \cite{St1}, Theorem 6,
$\underline{f}$ is the $f$-vector of a Cohen-Macaulay simplicial
complex if and only if $\underline{h}$ is an $O$-sequence. Hence, we
have to show that $\underline{\tilde{h}}$ is an $O$-sequence.

Using the above formula, we get
\[
\sum_{ i = 0}^d [h_i' - h_i] t^i = \al t^j (1-t)^{d-j}.
\]
It follows that
\[
h_i' - h_i =  \left \{\begin{array}{ll}
  0 & \text{if} \; i < j \\
  {\displaystyle \al (-1)^{i-j} \binom{d-j}{i-j} } & \text{if} \; i \ge j.
\end{array}
\right.
\]
Thus, we obtain $h_i = \tilde{h}_i = h_i'$ if $i < j$ and
\[
\begin{array}{cccccccc}
  h_j & \le &  \tilde{h}_j & \le & h_j'    \\
  h_{j+1}' & \le &  \tilde{h}_{j+1} & \le & h_{j+1} \\
  h_{j+2} & \le &  \tilde{h}_{j+2} & \le & h_{j+2}' \\
   &&  \vdots .\\
\end{array}
\]
Using that $\underline{h}$ and $\underline{h'}$ are $O$-sequences, we
get
\[
\tilde{h}_{j+1} \le h_{j+1} \le ((h_j)_{(j)})_1^1 \le
((\tilde{h}_j)_{(j)})_1^1
\]
and
\[
\tilde{h}_{j+2} \le h_{j+2}' \le ((h_{j+1}')_{(j+1)})_1^1 \le
((\tilde{h}_{j+1})_{(j+1)})_1^1,
\]
so similarly, for all $i \ge j$,
\[
\tilde{h}_{i+1} \le ((\tilde{h}_i)_{(i)})_1^1.
\]
Hence $\tilde{h}$ is an $O$-sequence, as desired.
\end{proof}

Using the above notation, one may wonder if, starting with suitable Cohen-Macaulay simplicial complexes having $f$-vectors $\underline{f}$ and $\underline{f'}$, one can construct a Cohen-Macaulay complex with $f$-vector $\underline{\tilde{f}}$. Using liaison theory this is indeed possible.

\begin{remark}\label{rem-obtainable-by-DBL} {\em
Recall that an ideal is called {\em glicci} \label{useglicci} if it is in the Gorenstein liaison class of a complete intersection. Each such ideal is Cohen-Macaulay. A squarefree monomial ideal is {\em squarefree glicci} (see \cite{squarefree-glicci})  if it can be linked in an even number of steps to a complete intersection such that all the ideals occurring in the even numbered steps are squarefree monomial ideals.

Let now $\Delta$ be a Cohen-Macaulay simplicial complex. Then there is a shifted Cohen-Macaulay complex $\Delta^s$ with the same $f$-vector as $\Delta$. According to \cite{squarefree-glicci}, Corollary 3.4, the shifted complex $\Delta^s$ is squarefree glicci. Actually, the proof in \cite{squarefree-glicci} gives the even stronger result that $\Delta^s$ can be obtained from the product over the boundary complexes of simplices by a finite number of basic double links, where the result of each basic double link is a Cohen-Macaulay simplicial complex.

Using the notation of Proposition \ref{prop-ICP-for-f-CM-vectors}, let now $\Delta$ be a shifted Cohen-Macaulay complex $\Delta$ with $f$-vector $\underline{f}$. Then Proposition \ref{prop-ICP-for-f-CM-vectors} provides the existence of a shifted Cohen-Macaulay complex $\tilde{\Delta}$ with $f$-vector $\underline{\tilde{f}}$. It follows that if $f_0 = \tilde{f_0}$, then $\tilde{\Delta}$ and  $\Delta$ are related by a finite sequence of basic double links and, possibly, some links relating two products of boundary complexes of $f_0 - d$ simplices each.
}
\end{remark}

We now turn to pure $f$-vectors of type 2.  While these are combinatorially simple, they possess interesting properties that we record here.  To set the notation, we will let $e$ be the socle degree of the $f$-vector thought of as a pure $O$-sequence; that is, we write
\[
\underline{f}=(f_{-1}=1,f_0,\dots ,f_{e-2},f_{e-1}=2).
\]
We have two facets of dimension $e-1$ and $f_0$ vertices.  Let $h$ be the number of vertices common to both facets, so $f_0 = 2e-h$.  Hence $\underline{f}$ can be thought of as the pure $O$-sequence generated by two squarefree monomials having $h$ variables in common.

\begin{remark} \label{formula}
{\em
It can be shown that the corresponding Stanley-Reisner ring is Cohen-Macaulay if and only if $h = e-1$, although we will not need this fact.  Furthermore, for any $h$ we have
\[
f_i = 2 \cdot \binom{e}{i+1} - \binom{h}{i+1}.
\]
}
\end{remark}

\begin{prop}\label{fvtr-ineq} Using the above notation, let
\[
\underline{f}^{(e,h)}=\underline{f}=(f_{-1}=1,f_0,\dots ,f_{e-2},f_{e-1}=2)
\]
 be a pure $f$-vector of type 2. Then, for all $a=0,1,\dots ,\left \lfloor \frac{e}{2} \right \rfloor -1$, we have
$$f_{\left \lfloor \frac{e}{2} \right \rfloor -a-1}> f_{\left \lfloor \frac{e+1}{2} \right \rfloor +a}.$$
\end{prop}

\begin{proof} Recall from Remark \ref{formula} that $f_i=2\binom{e}{i+1}-\binom{h}{i+1}$. We will  show that, for any integers $e\geq 2$, $h\leq e-1$ and $0\leq a\leq \left \lfloor \frac{e}{2} \right \rfloor -1$, we have

\begin{equation}\label{inn}
2\binom{e}{\left \lfloor \frac{e}{2} \right \rfloor -a}-\binom{h}{\left \lfloor \frac{e}{2} \right \rfloor -a}> 2\binom{e}{\left \lfloor \frac{e+1}{2} \right \rfloor +a+1}-\binom{h}{\left \lfloor \frac{e+1}{2} \right \rfloor +a+1},
\end{equation}
where as usual we set $\binom{n}{m}=0$ whenever $n<m$ or $m<0$. The  idea here will be to keep $e$ and $a$ fixed and to do the induction on $h$ backwards, that is, showing that if Inequality (\ref{inn}) holds for any given $h\leq e-1$, then it also holds for $h-1$.

We first prove  Inequality (\ref{inn}) for $h=e-1$, i.e., we want to show that

\begin{equation}\label{e-1}
2\binom{e}{\left \lfloor \frac{e}{2} \right \rfloor -a}-\binom{e-1}{\left \lfloor \frac{e}{2} \right \rfloor -a}> 2\binom{e}{\left \lfloor \frac{e+1}{2} \right \rfloor +a+1}-\binom{e-1}{\left \lfloor \frac{e+1}{2} \right \rfloor +a+1}.
\end{equation}

Using the Pascal triangle equality, the symmetry of binomial coefficients, and the identity $e=\left \lfloor \frac{e+1}{2} \right \rfloor +\left \lfloor \frac{e}{2} \right \rfloor $, Inequality (\ref{e-1}) can easily be rewritten as

\begin{equation}\label{hhh}
\binom{e}{\left \lfloor \frac{e+1}{2} \right \rfloor +a}+\binom{e-1}{\left \lfloor \frac{e+1}{2} \right \rfloor +a}> \binom{e}{\left \lfloor \frac{e+1}{2} \right \rfloor +a+1}+\binom{e-1}{\left \lfloor \frac{e+1}{2} \right \rfloor +a}.
\end{equation}

It is well-known that, for any given $n$, the binomial coefficients $\binom{n}{m}$ are a strictly unimodal sequence in $m$, for $m=0,1,\dots ,n$, with a single peak corresponding to $m=\frac{n}{2} $ if $n$ is even, and two peaks corresponding to  $m=\lfloor \frac{n\pm 1}{2} \rfloor $ if $n$ is odd. Therefore, since  $0\leq a \leq \left \lfloor \frac{e}{2} \right \rfloor -1$, the  chain of inequalities
$$\left \lfloor \frac{e+1}{2} \right \rfloor \leq \left \lfloor \frac{e+1}{2} \right \rfloor +a < \left \lfloor \frac{e+1}{2} \right \rfloor +a+1 \leq e$$
yields strict inequality  between the first terms of the two sides of Inequality (\ref{hhh}). This completes the proof of the result for $h=e-1$.

Assume now that Inequality (\ref{inn}) holds for a given integer $h\leq e-1$, and let us prove it for $h-1$. That is, we want to show that

\begin{equation}\label{wwww}
2\binom{e}{\left \lfloor \frac{e}{2} \right \rfloor -a}-\binom{h-1}{\left \lfloor \frac{e}{2} \right \rfloor -a}> 2\binom{e}{\left \lfloor \frac{e+1}{2} \right \rfloor +a+1}-\binom{h-1}{\left \lfloor \frac{e+1}{2} \right \rfloor +a+1}.
\end{equation}

Employing the Pascal triangle equality, the previous inequality becomes

\begin{equation}\label{h-1}
\begin{array}{l}
2  \binom{e}{\left \lfloor \frac{e}{2} \right \rfloor -a}-\binom{h}{\left \lfloor \frac{e}{2} \right \rfloor -a} +\binom{h-1}{\left \lfloor \frac{e}{2} \right \rfloor -a-1} \\
\hbox{\hspace{3cm}} > 2\binom{e}{\left \lfloor \frac{e+1}{2} \right \rfloor +a+1}-\binom{h}{\left \lfloor \frac{e+1}{2} \right \rfloor +a+1}+\binom{h-1}{\left \lfloor \frac{e+1}{2} \right \rfloor +a}.
\end{array}
\end{equation}

The first two summands of each side of Inequality (\ref{h-1}) are the same as those of Inequality (\ref{inn}). Therefore, it is now enough to show that

\begin{equation}\label{ciao}
\binom{h-1}{\left \lfloor \frac{e}{2} \right \rfloor -a-1}\geq \binom{h-1}{\left \lfloor \frac{e+1}{2} \right \rfloor +a}.
\end{equation}

It is easy to see that, for $a\geq 0$ and $h\leq e-1$, we have  $\left \lfloor \frac{e+1}{2} \right \rfloor +a \geq \frac{h-1}{2}.$ Hence, using the unimodality of binomial coefficients, one moment's thought shows  that Inequality (\ref{ciao}) holds once we have simultaneously $\left \lfloor \frac{e}{2} \right \rfloor -a-1 \leq \left \lfloor \frac{e+1}{2} \right \rfloor +a$ and $h-1 -(\left \lfloor \frac{e}{2} \right \rfloor -a-1 ) \leq \left \lfloor \frac{e+1}{2} \right \rfloor +a$. But the  former inequality follows immediately from the assumption $a\geq 0$, and the latter  from  $h\leq e-1$.

This completes the proof of Inequality (\ref{wwww}) and the result follows.
\end{proof}

As a consequence, there is a string of inequalities that includes all the entries of the $f$-vector in the case of type two:

\begin{corollary} \label{string of ineq}
Let
\[
\underline{f} = (1,f_0,f_1,\dots,f_{e-2},2)
\]
be a pure $f$-vector of type 2.  If $e$ is even then
\[
f_{\frac{e}{2} -1} > f_{\frac{e}{2}} \geq f_{\frac{e}{2} -2} > f_{\frac{e}{2} +1} \geq \dots > f_{e-2} \geq f_0 > 2.
\]
If $e$ is odd then
\[
f_{\frac{e-1}{2}} \geq f_{\frac{e-1}{2}-1}  > f_{\frac{e+1}{2}} \geq f_{\frac{e-1}{2}-2} > f_{\frac{e+1}{2}+1} \geq \dots >  f_{e-2} \geq f_0 > 2.
\]
\end{corollary}

\begin{proof}
Combine Proposition \ref{fvtr-ineq} with Hibi's Theorem \ref{hibi}.
\end{proof}

This last result has the following consequence, which is unknown for pure $O$-sequences of type two in general, although in Corollary \ref{pure 2 unimodal} we proved it for the case of three variables.

\begin{corollary}
A pure $f$-vector of type 2 is unimodal.
\end{corollary}

\begin{remark}{\em
Corollary \ref{string of ineq} fails for pure $O$-sequences of type 2.  Indeed, the following $h$-vectors are pure $O$-sequences of type 2:
\[
{\tt (1,2,3,4,4,2)} \hbox{\ \ \ and \ \ \ } {\tt (1,3,5,7,6,2)}
\]
(cf. the appendix in \cite{GHMS}).  It also fails for pure $f$-vectors of (sufficiently) higher type, as Stanley's example shows (see page \pageref{stanley example}).
}
\end{remark}

%%%%%%%%%%%%%%%%%%%%%%%%%%%%%%%%%%%%%%%%%%%%%%%%%%

\chapter{Some open or open-ended problems} \label{final comments}

In this chapter, we collect some of the Questions which naturally come up from our work, apart from the ones stemming from the conjectures posed in this monograph.\\
\\
{\bf 9.1. Unimodality of pure $O$-sequences.} First, in the light of the unimodality result for pure $O$-sequences of type 2 and codimension three, and the non-unimodality  in higher type  (see \cite{BZ}), we  are led to pose the following questions:

\begin{question}\label{q1} \rm
  \begin{enumerate}
 \item[(a)]
What is the largest type of pure $O$-sequences of codimension 3 forcing unimodality?  How about higher codimension?
\item[(b)] Does the largest type forcing unimodality for pure $O$-sequences of codimension 3 coincide with  the largest type forcing unimodality for  arbitrary Artinian level algebras of codimension 3?
\item[(c)] What is  the largest type forcing unimodality among
  Artinian reductions of reduced, level sets of points in $\mathbb
  P^3$?  How about higher codimension?
  \item[(d)] If we fix the codimension and the socle type, what is the behavior for the number of non-unimodal pure $O$-sequences as the socle degree gets large?  Is this number eventually zero?
      \item[(e)] Among subsets of $t$ monomials of degree $d$ in $r$ variables,
      ``how many" give non-unimodal pure $O$-sequences?  Does the number giving non-unimodal pure $O$-sequences divided by the total number of subsets of $t$ monomials of degree $d$ go to zero as $d$ goes to infinity?  Does the answer depend on the number of variables?
      \item[(f)] What is the largest socle degree of pure $O$-sequences of codimension 3 forcing unimodality?  
\end{enumerate}
\end{question}

So far, we know from Corollary~\ref{pure 2 unimodal} that pure $O$-sequences of type $2$
and codimension $3$ are all unimodal.  Boyle \cite{boyle} has shown that also pure $O$-sequences of type $3$ and codimension $3$ are unimodal.  On the other hand, we know from Example \ref{type 14} that there are non-unimodal pure $O$-sequences of type 14 and codimension 3, and from Remark \ref{soc deg 16} that there are non-unimodal pure $O$-sequences of socle degree 12.  The first example of a codimension 3 non-unimodal pure $O$-sequence was obtained in \cite{BZ}, where the authors found 
non-unimodal pure $O$-sequences of type 37 and codimension $3$.  Concerning Question (f), notice that if we do not specify the codimension then Theorem \ref{3} and Example \ref{nonunimodal soc deg 4} provide the answer.

Note that Question \ref{q1}(b) is open in codimension 3 and 4, but fails in codimension 5 and higher codimension, where Gorenstein non complete intersections may be non-unimodal, essentially by an example of  Bernstein and Iarrobino  in  \cite{BI}.

Observe that any monomial Artinian algebra lifts to the coordinate ring of a reduced set of points.  Thus if  Question~\ref{q1}(b) has an affirmative answer, and if we can identify this type, then we have an answer to (c).  However, we conjecture that the answer to (b) is no, in general.

Concerning arbitrary codimension, our computer experiments  together with Theorem
\ref{prop-WLP-type2} give some evidence for an affirmative
answer - at least in low codimension - to the  following question:

\begin{question}
\rm
Are all pure $O$-sequences of type 2 unimodal?
\end{question}

What about unimodality for special classes of pure $f$-vectors? Among the several open questions, let us recall part b) of Stanley's Twenty-Fifth Problem, from his article \cite{StImu} for the year 2000 IMU-AMS Volume ``Mathematics: Frontiers and Perspectives'', and the analogous question on matroid $h$-vectors: 

\begin{question} \rm
\begin{enumerate}
 \item[(i)] Are all $f$-vectors of matroid simplicial complexes unimodal, or even log-concave?
 \item[(ii)] What about matroid $h$-vectors (which, according to Stanley's conjecture, are all pure $O$-sequences)?
 \end{enumerate}
\end{question}

%%%
%%% Here we should add a short summary of the experiments done, ranges
%%% for r and e. Maybe also for higher types.
%%%
\noindent
{\bf 9.2.  The Interval Conjecture.} In Proposition \ref{prop-ICP-for-f-CM-vectors} we showed that
$f$-vectors of  Cohen-Macaulay simplicial complexes have the interval
property. Thus, it is  natural to wonder if this is also true for
the larger class of $f$-vectors of pure simplicial complexes.

\begin{question}
\rm Suppose that, for some positive integer $\alpha $, both
$\underline{f} = (1,f_0,\ldots, \\
f_j, f_{j+1},
f_{j+2},\ldots,f_d)$ and $\underline{f'} = (1,f_0,\ldots,f_j,f_{j+1}
+ \alpha, f_{j+2},\ldots,f_d)$ are $f$-vectors of pure simplicial
complexes. Is it then true that, for each integer $\beta =0,1,...,
\alpha$, also $\underline{\tilde{f}} = (1,f_0,\ldots,f_j,f_{j+1} +
\beta, f_{j+2},\ldots,f_d)$ is an $f$-vector of a pure simplicial
complex?
\end{question}

As we have mentioned in the Interval Conjecture chapter, there might be a potentially strong (if still largely unexplored) connection between the ICP and Stanley's matroid $h$-vector conjecture. Hence we recall here  what is arguably one of the main open problems in this area.

\begin{question}
\rm (Stanley's conjecture.) Are all matroid $h$-vectors pure $O$-sequences?
\end{question}

%%%
%%% Add more about what is known for other inverval conjectures, which
%%% hopefully also answers the referee's demand for more justification
%%% of the ICP.
%%% Also look at if the remark 4.2 can be extended with more details
%%% on what is known.
%%%
\noindent
{\bf 9.3. Enumerating pure $O$-sequences.} With regard to Chapter \ref{enum}, we have the following questions:

\begin{question}{\rm
\begin{enumerate}
\item[(a)] Determine, possibly with a closed formula, the number $P^r(e)$ of pure $O$-sequences of codimension $r$ and socle degree $e$.
\item[(b)]  Determine, possibly with a closed formula, the number $P^r_t(e)$ of pure $O$-sequences of codimension $r$, socle degree $e$ and type $t>1$.
\item[(c)]  Find a closed asymptotic formula for the number $L^r_t(e)$ of level Hilbert functions of codimension $r$, socle degree $e$ and type $t>1$, when $r\rightarrow +\infty $.
\item[(d)] Study the topics of this chapter from the opposite perspective: when the codimension $r$ is fixed and the socle degree $e$ is the variable.
\item[(e)] Suppose $r>1$. Is
\[
\lim_{e \rightarrow \infty} \frac{P_t^r(e)}{L_t^r(e)}
\]
zero?    (From the results of Chapter \ref{enum}, notice that this limit is zero when $t=1$, for any $r > 1$.)
If the answer is negative, for which values $r$ and $t$ does the limit exist, and
 what is this limit?  Related to this would be to replace the numerator and denominator by $\sum_t P_t^r(e)$ and $\sum_t L_t^r(e)$, respectively, and ask the analogous questions.
\end{enumerate}

%%% COMMENTS
%%% Should we also write more about open problems in the field, like
%%% Stanley's Conjecture, and what is known about them? We only refer
%%% to some other papers without mentioning what they prove.
%%% Actually, this looks like what the referee really wants by
%%% "thorough" when you read what is written under "weaknesses". This
%%% is not the aim of the chapter as it is explained now, so then we
%%% need to change the first paragraph of the chapter.
} \end{question}

%}

\appendix
%    Include appendix "chapters" here.
%%%%%%%%%%%%%%%%%%%%%%%%%%%%%%%%%%%%%%%%%%%%%%%%%%

\chapter{Collection of definitions and notation} \label{Coll. dfns}

In this appendix  and as an aid to the reader, we gather together the
main definitions and notation that we have used throughout the work.

\begin{definition}[Order ideal, pure order ideal, pure $O$-sequence,
  p. \!\!\pageref{useorderideal}]
  A \emph{(monomial) order ideal} is a finite collection $X$ of (monic) monomials such that, whenever $M\in X$ and $N$ divides $M$, then $N\in X$.

If all (say $t$) maximal monomials of $X$ have the same degree, then $X$ is \emph{pure} (of type $t$).

A \emph{pure $O$-sequence} is the vector, $\underline{h}=(h_0=1,h_1,...,h_e)$,
counting the monomials in each degree of some pure order ideal.  See also Definition \ref{basicdefs}.
\end{definition}

\begin{definition}[Gorenstein link, p. \!\!\pageref{Glink}]
Two ideals, $I$ and $I'$, are {\em Gorenstein linked} if there is a Gorenstein ideal $\mathfrak a$ such that $\mathfrak a : I = I'$ and $\mathfrak a : I' = I$.
\end{definition}

\begin{definition}[basic double link,   p. \!\!\pageref{basicdoublelink}]
Let $J \subset I \subset R = k[x_1,\ldots,x_r]$ be homogeneous
ideals such that $\codim J = \codim I - 1$. Let $f \in R$ be a
 form of degree $d$ such that $J  : f = J$. Then the ideal $I' := f
\cdot I + J$ is called a {\em basic double link} of $I$. The name stems from the fact that $I'$ can be Gorenstein linked to $I$ in two
steps if $I$ is unmixed and $R/J$ is Cohen-Macaulay and generically
Gorenstein (\cite{KMMNP}, Proposition 5.10). 
\end{definition}

\begin{definition}[glicci, p. \!\!\pageref{useglicci}]
  An ideal $I$ in $R=k[x_1,x_2,\dots,x_r]$ is called {\em glicci} if it is in the Gorenstein liaison class of a complete intersection.
\end{definition}

\begin{definition}[Socle, level of type $t$, socle degree,
  p. \!\!\pageref{defsocle}; Cohen-Macaulay type, p. \!\!\pageref{useCMtype}]
  Given an Artinian standard graded $k$-algebra $A$, the {\em socle} of $A$ is defined as the annihilator of the homogeneous maximal ideal $\mathfrak{m}$, namely $$soc(A)=\{ a\in A \mid a{\mathfrak{m}}=0 \}.$$ We say that $A$ is {\em level of type $t$} if the socle of $A$ has dimension $t$ and is concentrated in one degree, usually called the {\em socle degree} of $A$ and denoted by $e$.

 For a graded Cohen-Macaulay algebra $A$, the {\em Cohen-Macaulay type} is the minimal number of generators of the canonical module of $A$. In the case when $A$ is artinian, this equals $\dim_k soc(A)$. We say that $A$ is {\em level} if  its Artinian reduction is level.
\end{definition}

\begin{definition}[Inverse system, p. \!\!\pageref{definversesystem}]
 For a homogeneous ideal $I\subseteq
R=k[x_1,x_2,\dots,x_r]$, we define the \emph{inverse system} \label{definversesystem} $R$-module
$I^\perp=\operatorname{ann}_{\mathcal R}(I)\subseteq\mathcal
R=k[y_1,y_2,\dots,y_r]$, where $R$ acts on $I^\perp $ by
contraction. That is, the external product is the linear action determined by
$$
x_i \circ y_1^{a_1}y_2^{a_2}\cdots y_r^{a_r} =
  \begin{cases}
    y_1^{a_1}y_2^{a_2}\cdots y_i^{a_i-1}\cdots y_r^{a_r},&\text{if $a_i>0$,}\\
    0,&\text{if $a_i=0$,}
  \end{cases}
$$
for any non-negative integers $a_1,a_2,\dots,a_r$. When considering only monomial ideals, we can simplify the notation by considering inverse system modules to also lie inside the  polynomial ring $R$. It follows that, in any degree $d$, the inverse system $I^\perp_d$ is spanned by the monomials of $R_d$ not in $I_d$. We will call the elements of $I^\perp_d$ \emph{derivatives}.
\end{definition}

\begin{definition}[p. \!\!\pageref{basic-defs}] \label{basicdefs}
 Let $\underline{h} = (h_0,\dots,h_e)$ be a sequence of positive integers.
 \begin{enumerate}
 \item $\underline{h} $ is said to be an {\em $O$-sequence} if it is the Hilbert
     function of some standard graded Artinian $k$-algebra.
 \item $\underline{h} $ is said to be a {\em pure $O$-sequence} if it is the Hilbert
     function of some standard graded monomial Artinian level $k$-algebra.
 \item  $\underline{h}$ is called {\em unimodal} if $h_0 \leq h_1
 \leq \dots \leq h_j \geq h_{j+1} \geq \dots \geq h_e$ for some
 $j$.
\item  $\underline{h}$ is said to be {\em flawless}  if $h_{i}\le
 h_{e-i}$ for $i\le \frac{e}{2}$.
\item We will say that $\underline{h}$ is {\em non-decreasing} if
 $h_{i+1}\ge h_{i}$ for all $i$.
\item $\underline{h}$ is called an {\em SI-sequence}  if it
 satisfies the following two conditions:
\begin{enumerate}
 \item$\underline{h}$ is {\em symmetric}, i.e.\ $h_{e-i} = h_i$ for
   all $i=0,\dots,\lfloor \frac{e}{2} \rfloor$.
 \item $(h_0,h_1-h_0,h_2-h_1,\dots,h_j-h_{j-1})$ is an $O$-sequence, where $j
     =
   \lfloor \frac{e}{2} \rfloor$; i.e.\ the ``first half'' of
   $\underline{h}$ is a {\em differentiable} $O$-sequence.
 \end{enumerate}
\end{enumerate}
\end{definition}

\begin{notation}[p. \!\!\pageref{binomialexpansion}]
Let $n$ and $d$ be positive integers. There exist uniquely
determined integers $k_d>k_{d-1}> \dots
>k_{\delta}\geq \delta \geq 1 $ such that
\[
n = {\textstyle n_{(d)}= \binom{k_{d}}{d}+ \binom{k_{d-1}}{d-1}+
  \cdots + \binom{k_{\delta}}{\delta} .}
\]
This is called the {\em
  $d$-binomial expansion of $n$}. Following \cite{BG}, for any
integer $a$ and $b$, we define
$$ {\textstyle (n_{(d)})^b_a= \binom{k_{d}+b }{ d+a}+\binom{k_{d-1}+b }{ d+a-1}+
  \cdots + \binom{k_{\delta}+b }{ \delta +a}}$$ where we set
$\binom{m}{p}=0$ whenever $p<0$ or $m<p$.
\end{notation}

\begin{definition}[pp. \!\!\pageref{almost}-\pageref{thmAsymptoticSocleDegree3}]
  For positive integers $r$, $e$ and $t$ we define the following sets:
  \begin{itemize}
  \item $O(r,e)= \{\text{$O$-sequences of codimension $r$ and socle degree $e$}\}$.
  \item $P(r,e)= \{\text{pure $O$-sequences of codimension $r$ and socle degree $e$}\}$.
  \item $P(r,e,t)= \{\text{pure $O$-sequences of codimension $r$, socle degree $e$ and type $t$}\}$.
 \item $L(r,e)= \{\text{level $O$-sequences of codimension $r$ and socle degree $e$}\}$.
 \item $L(r,e,t)= \{\text{level $O$-sequences of codimension $r$, socle degree $e$ and type $t$}\}$.
 \item $D(r,e)= \{\text{differentiable $O$-sequences of codimension $r$ and socle degree $e$}\}$.
  \item $M(r,e)= \{\text{$O$-sequences of codimension $r$ and socle degree at most $e$}\}$.
  \item $S(r,e)= \{\text{$O$-sequences of codimension at most $r$ and socle degree $e$}\}$.
  \item $SI(r,e)= \{\text{$SI$-sequences of codimension $r$ and socle degree $e$}\}$.
  \item $G(r,e)= L(r,e,1)= \{\text{Gorenstein $O$-sequences of codimension $r$ and socle}$
\end{itemize}
{\ }{\ }{\ }{\ }{\ }{\ }{\ }{\ }{\ }{\ }{\ }{\ }degree $e\}$.\\
\\
We define the following numbers:
\begin{itemize}
\item $O^r(e)  = \#O(r,e)$.
\item $P^r(e) = \#P(r,e)$.
\item $P_t^r(e) = \#P(r,e,t)$.
\item $L^r(e) = \#L(r,e)$.
\item $L_t^r(e) = \#L(r,e,t)$.
\item $D^r(e) = \#D(r,e)$.
\item $M^r(e) = \#D(r,e)$.
\item $S^r(e) = \#D(r,e)$.
\item $SI^r(e) = \#D(r,e)$.
\item $G^r(e) = \#G(r,e)$.
\item $P(t) = \sum_{r\ge 1}^\infty P_t^r(3)$.
\end{itemize}
\end{definition}

\begin{notation}[pp. \!\!\pageref{bmax}-\pageref{amax}]
In the special case of pure $O$-sequences, $(1,r,a,b)$, of socle degree $3$, we use the notation
\begin{itemize}
\item $r_{\min }(a,b)=\min\{r| (1,r,a,b)\in P(r,3)\}$.
\item $r_{\max }(a,b)=\max\{r| (1,r,a,b)\in P(r,3)\}$.
\item $a_{\min }(r,b)=\min\{a| (1,r,a,b)\in P(r,3)\}$.
\item $a_{\min }(r,b)=\max\{a| (1,r,a,b)\in P(r,3)\}$.
\item $b_{\min }(r,a)=\min\{b| (1,r,a,b)\in P(r,3)\}$.
\item $b_{\max }(r,a)=\max\{b| (1,r,a,b)\in P(r,3)\}$.
\end{itemize}
Sometimes, we simplify the notation by suppressing the dependence on
the arguments, and we simply write for instance $r_{\min }$.
\end{notation}

\begin{notation}[$c_e$, p. \!\!\pageref{c_e}; $p_r(e)$, p. \!\!\pageref{p_r}]
  In chapter 5,  we use the following notation for a positive integer $e$:
  $$
  c_e=\frac{\prod_{i=0}^{e-2}\binom{\binom{e+1}{2}-\binom{i+1}{2}-1}{i}}{(\binom{e+1}{2}-1)!};
  $$
  also, we define   $p_r(e)$ as the number of integer partitions of $e$ with exactly $r$ parts.
\end{notation}

\begin{definition}[Simplicial complex, face, facet, pure simplicial
  complex, $f$-vector, matroid, p. \!\!\pageref{simplicialcomplex}]
An {\em (abstract) simplicial complex} $\Delta $ on $[n]:= \{1,
\dots  ,n\}$ is a subset of the power set of $[n]$ which is  closed
under inclusion, i.e. if $F \subseteq G$ and $G \in \Delta $, then
$F \in \Delta $. The elements $F$ of $\Delta $ are called {\em
faces}, and the maximal elements under inclusion are called {\em
facets}. The simplicial complex $\Delta $ is said to be  {\em pure}
if all its facets have the same cardinality.  The {\em $f$-vector} of $\Delta $ is the vector $\underline{f_{\Delta }} = \underline{f} = (f_{-1} = 1 ,f_0 , \dots , f_{d-1})$, where $f_i$ counts the number of faces of $\Delta $ of cardinality $i+1$.

A simplicial complex  is called a {\em matroid} if, for all
$W\subseteq [n]$, the restriction $\Delta _W = \{F \in \Delta : F
\subseteq W \}$ is a pure simplicial complex.
\end{definition}

\begin{definition}[WLP, Weak Lefschetz element, $g$-element, SLP, p. \!\!\pageref{def of wlp}]
%[\ref{def of wlp}]
We say that  a standard graded Artinian algebra $A=R/I$
has the {\em Weak Lefschetz Property (WLP)}
if there is a linear form $L \in A_1$ such that, for all
integers $j$, the multiplication map
\[
\times L: A_{j} \to A_{j+1}
\]
has maximal rank, i.e.\ it is injective or surjective.   Such an element is called a {\em Lefschetz element} for $A$. When only considering injectivity, it will be called a {\em $g$-element}.

We say that $A$ has the {\em Strong Lefschetz Property (SLP)} if there is an $L$ such that
\[
\times L^s : A_j \rightarrow A_{j+s}
\]
 has maximal rank for all $j$ and all $s$.

We often abuse notation and say that $I$ has the WLP or the SLP when $R/I$ does.
\end{definition}

%\appendix
%    Include appendix "chapters" here.
%\include{}

\backmatter
%    Bibliography styles amsplain or harvard are also acceptable.
%\bibliographystyle{amsalpha}
\printindex

\bibliographystyle{amsalpha}

\end{document}